\documentclass[a4paper,10pt]{elsarticle}

\usepackage{array}
\usepackage{subcaption}
\usepackage[dvipsnames]{xcolor}
\usepackage{tabularx}
\usepackage{graphicx,wrapfig,lipsum}
\usepackage{ifplatform}
\usepackage{preview}
\usepackage{environ}
\usepackage{graphics}
\usepackage{epstopdf}	
\usepackage{moreverb}
\usepackage{dsfont}
\usepackage{grffile}
\usepackage{bm}
\usepackage{multirow}
\usepackage{frcursive}
\usepackage{mathrsfs}
\usepackage{nicefrac}
\usepackage{url}
\usepackage[toc]{appendix}

\usepackage{algpseudocode}
\usepackage{algorithmicx}
\usepackage{algorithm}
\usepackage{enumerate}

\algdef{SE}[DOWHILE]{Do}{doWhile}{\algorithmicdo}[1]{\algorithmicwhile\ #1}%

\usepackage[textsize=footnotesize backgroundcolor=yellow!70, bordercolor=orange]{todonotes}

\newcommand\BibTeX{{\rmfamily B\kern-.05em \textsc{i\kern-.025em b}\kern-.08em
T\kern-.1667em\lower.7ex\hbox{E}\kern-.125emX}}

\usepackage{lmodern}  % for bold teletype font
\usepackage{xcolor}   % for \textcolor
\usepackage{listings}
\usepackage[square,numbers]{natbib}
\newcommand{\dt}{\Delta t}
\newcommand{\dx}{\Delta x}

\newcommand{\F}{\mathcal{F}}
\newcommand{\G}{\mathcal{G}}

\makeatletter
\newcommand\footnoteref[1]{\protected@xdef\@thefnmark{\ref{#1}}\@footnotemark}
\makeatother

\newcommand{\ha}{\frac{1}{2}}

\newcommand{\bi}{l}
\newcommand{\bj}{k}

\def\be{\begin{equation}}
\def\ee{\end{equation}}
\newcommand{\ei}[0]{\end{itemize}}
\newcommand{\beann}[0]{\begin{eqnarray*}}
\newcommand{\eeann}[0]{\end{eqnarray*}}
\def\bea{\begin{eqnarray}}
\def\eea{\end{eqnarray}}
\def\ba{\begin{array}{l}\displaystyle}
\def\ea{\end{array}}
\renewcommand{\epsilon}{\varepsilon }

\newfont{\numerikEleven}{ecrm1000}
\newfont{\numerikTen}{cmss10}
\newfont{\numerikNine}{cmss9}
\newfont{\numerikEight}{cmss8}
\newfont{\numerikSeven}{cmss7}
\newfont{\numerikSix}{cmss6}
\usepackage{amsmath,amssymb,amsthm}

\newtheorem{theorem}{Theorem}[section]

\newtheorem{algo}{Algorithm}

\theoremstyle{definition}

\theoremstyle{remark}

\numberwithin{equation}{section}

\renewcommand{\(}{\Bigl(}
\renewcommand{\)}{\Bigr)}

\newcommand{\ab}{\Bar{a}}
\newcommand{\bb}{\Bar{b}}
\newcommand{\cb}{\Bar{c}}
\newcommand{\db}{\Bar{d}}
\newcommand{\Hb}{\Bar{H}}
\newcommand{\p}{\partial}
\renewcommand{\epsilon}{\varepsilon}
\usepackage{stmaryrd}
\definecolor{GreenForest}{RGB}{0, 140, 80}
\newcommand{\ema}[1]{{\bf\textcolor{GreenForest}{#1}}}

\newcommand{\giovanni}[1]{{\color{red}{#1}}}

\newcommand{\jump}[1]{ \llbracket #1 \rrbracket}
\newcommand{\Jump}[1]{ \biggl\llbracket #1 \biggr\rrbracket} %%%%%%%%%%

\begin{document}
	
\begin{frontmatter}
\title{A third-order finite volume semi-implicit method\\ for the Shallow Water-Exner model}
%\title{A Third order staggered finite volume IMEX method for SW and SWE equation}
%\author{E.D. Fern\'andez-Nieto, J. Garres-D\'iaz, E. Macca, G. Russo}

% AUTHORS
\author[a]{Enrique D. Fernandez-Nieto\corref{}}
\ead{edofer@us.es}

\author[b]{Jose Garres-Diaz}
\ead{jgarres@us.es}

\author[c]{Emanuele Macca\corref{cor1}}  % SOLO LUI HA corref
\ead{emanuele.macca@unict.it}
\cortext[cor1]{Corresponding author} 

% \author[c]{Emanuele Macca}
% \ead{emanuele.macca@unict.it, Corresponding author}

\author[c]{Giovanni Russo}
\ead{russo@dmi.unict.it}

\address[a]{Dpto. Matem\'atica Aplicada I \& IMUS, Universidad de Sevilla, 41012 Sevilla, Spain}
\address[b]{Dpto. Matem\'atica Aplicada II \& IMUS, Universidad de Sevilla, 41092 Sevilla, Spain}
\address[c]{Dpto. Matematica ed Informatica, Universit\`a di Catania, 95125 Catania, Italy}

% \maketitle

\begin{abstract}
In this work, third-order semi-implicit schemes on staggered meshes for the shallow water and Saint-Venant-Exner systems are presented. They are based on a third-order extension of the technique introduced in Cassulli \& Cheng \cite{Casulli1992}. The stability conditions for these schemes depend on the velocity and not on the celerity, allowing us to reduce computational efforts, especially in subcritical flow simulations, which is the regime we are mainly interested in. The main novelty consists in  the third-order approximation of the pressure gradient term in the momentum equation through appropriate polynomial reconstructions. Concretely, CWENO conservative reconstruction is considered for the water thickness $h$  and a centered fourth-degree polynomial is adopted interpolating the cell averages of the free surface $\eta$. For time discretization, a third-order IMEX scheme is applied. In addition, a novel time-dependent semi-analytical solution for Saint-Venant-Exner system is introduced and compared with the numerical ones. Several tests are performed, including accuracy tests showing third-order accuracy, well-balance tests, and simulations of slow bedload processes for large time.
\end{abstract}

% \maketitle

\end{frontmatter}
\tableofcontents
\section{Introduction}

Geophysical flows have been widely studied for many years, both from the modeling and numerical simulation point of view. Concretely, sediments dynamics (transport, suspended load, multiple species) have gained interest among researchers, due to its application to environmental problems related to erosion in rivers, land development in coastal areas, etc. Here we focus on numerical methods applied to such problems.

The development of efficient solvers for atmospheric and hydrodynamics flows is essential when dealing with real applications, in particular when considering large scale simulations, where the computational cost associated to explicit schemes is unaffordable because of the stability restriction on the time step. One may consider different strategies leading to longer time steps in the simulations, which helps reducing CPU effort. Fully implicit methods are usually impractical since they involve solving large and complex nonlinear systems at each time step. On the contrary, semi-implicit methods, where some terms are treated explicitly whereas others are implicitly discretized, have been widely used to simulate compressible and incompressible flows, including shallow water flows. See \cite{Casulli1992,Bonaventura2002,Rosatti2011,Tumolo2013,Busto2020,Busto2021,GarresDiaz2021,Orlando2022,Frolkovic2022,GomezBueno2023,CaballeroCardenas2023,Grosso2023} among many other references. The main drawback of these methods is the fact that they allow slow waves to be accurately solved, whereas fast waves are not well reproduced. If one has to reproduce both fast and slow waves, then a different methods should be considered, such as multi-rate methods (see \cite{Bonaventura2020}), where different time steps are used associated to \emph{fast/slow} variables.

Focusing now on shallow water flows, we recall what has been done in  \cite{Casulli1992,Casulli1994}, where the pressure gradient is implicitly discretized (very often with a linearization), and the discrete momentum equation is used in the mass equation obtaining an elliptic equation for the new  free surface values. Once the elliptic problem is solved, the new values for the discharge are updated. As result, the scheme entails a less restrictive stability (CFL) condition, which does not depend on the celerity but just on the velocity. Thus, this strategy allows a significant reduction in the computational time in case of subcritical regime, making it suitable to the simulation of slow processes. Then, in principle, bedload transport models with a weak fluid/sediment interaction is a natural application for this technique.

Actually, bedload models with weak fluid/sediment interaction are characterized by two different time scales, since the sediment evolution is much slower than the surrounding fluid. Moreover, this phenomena occurs mostly in a subcritical regime. So, it is a natural framework to apply semi-implicit methods. Let us remark that the ideas of \cite{Casulli1992} were used for multi-layer shallow flows in \cite{Bonaventura2018} in subcritical regimes, also with the addition of a decoupled Exner equation, and in \cite{Busto2022} for shallow flows for all Froude numbers also in the 2D case. The same strategy has been already used to design efficient solvers for bedload transport problems. Concretely, in \cite{Macca2024} the authors add the Exner equation with the Grass's closure for the solid transport discharge, and in \cite{GarresDiaz2022} efficient schemes for a family of general bedload models with gravitational effects were proposed. Notice that most of this work consider a staggered mesh instead of the collocated standard mesh commonly used in finite volume methods. It is relevant to emphasize that all the methods in these works are at most second-order accurate. Sediment transport is a very slow processes, therefore third-order methods are useful to reduce the numerical diffusion of the scheme and produce accurate results. To our knowledge, such third- or higher-order schemes for shallow flows using this semi-implicit strategy have not been proposed yet in the literature.

The main goal of this work is to extend the ideas in \cite{Casulli1992} to construct a third-order semi-implicit Finite Volume/Finite Difference method for shallow water flows with application to bedload transport problems. 
To this aim, the main idea introduced in this paper is to approximate the pressure gradient implicitly by using a combination of CWENO reconstruction and a centered third-order polynomial preserving the average of the free surface variable within the cells. Also, a novel semi-analytical solution, which depends on time, for the Exner system with the Grass formula and linear friction is introduced in this paper. This solution is used in the numerical test section to illustrate the performance of the proposed method for a slow erosion process in a large domain.

\bigskip

The paper is organized as follows: Section \ref{sec:num_scheme} is devoted to presenting the third-order semi-implicit scheme for the shallow water system. Concretely, Subsections \ref{Sec:space_reconst} and \ref{sec:time_recons} deal with the spatial and time discretizations, respectively. This method is adapted to the case of bedload sediment transport in Section \ref{sec:num_scheme_exner}. In Subsection \ref{subsec:semi-analitical-sol-SVE} a novel time-dependent semi-analytical solution for the Exner system is introduced. Numerical tests are presented in Section \ref{sec:numerics} and the conclusions of the paper are in Section \ref{se:conclusions}. For completeness, some known semi-analytical solutions of the Exner system are reported in Appendix \ref{App_A}. Finally, CWENO coefficients are reported in Appendix~\ref{App_B}, and the stages of the fully discrete semi-implicit IMEX-RK scheme are summarized in Appendix~\ref{App_C}, for completeness.

\section{Numerical scheme} \label{sec:num_scheme}
In this Section we describe all the details of the numerical scheme. For the sake of clarity, in particular for those readers interested in applying this technique to the shallow water case, we introduce the method for the Saint-Venant system (SV hereinafter) while its extension to bedload problems is dealt in Section \ref{sec:num_scheme_exner}.

Thus, we start by establishing the usual SV system,
\begin{equation}
\label{SW}
    \begin{cases}
        \partial_t \eta + \partial_x q = 0 ,\\
        \partial_t q + \partial_x (q^2/h) + gh\partial_x\eta = 0,
    \end{cases}
\end{equation}
where $h$ denotes the water thickness, $b= b(x)$ the bottom topography (fixed in time) and $\eta = h + b$ is the free surface. $q$ denotes the fluid discharge, $u = q/h$ the velocity and $g$ the gravity acceleration (see Figure \ref{fig:set-up}). 
\begin{figure}[!ht]
	\centering	
	\includegraphics[width=0.9\textwidth]{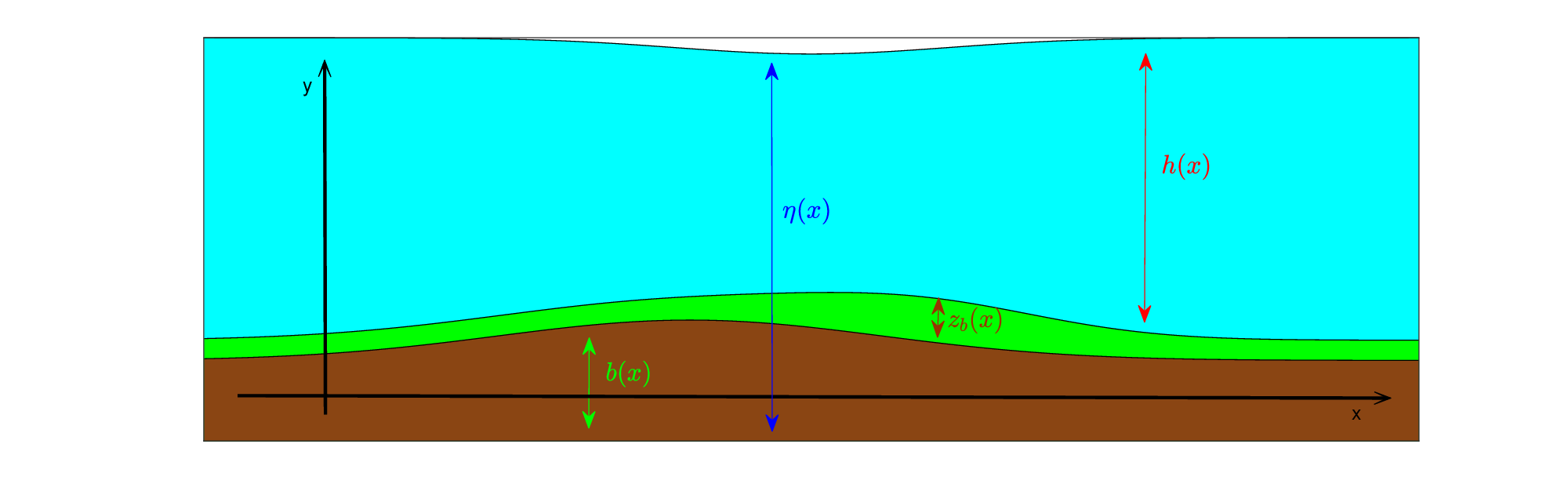}
	\caption{1D Exner model: water surface $\eta(x);$ water-flow $h(x);$ sediment layer $z_b(x)$ and bottom topography $b(x).$}
	\label{fig:set-up}
\end{figure}

\medskip

The computational domain is $\Omega = [x_{a},x_{b}]\times [0,T]$, where $T$ denotes the maximum simulation time.

%Before delving into the numerical schemes presented in this paper, we establish the temporal framework within the time domain $[0,T]$, where $T>0$. 
This domain is discretized into time intervals $[t^n,t^{n+1}]$, indexed by $n \in \mathbb{N}$, each with a time-step 
$\Delta t_n=t^{n+1}-t^n$, adhering to a CFL (Courant-Friedrichs-Lewy) condition \cite{CFL}, which will be discussed later in Section \ref{sec:numerics}. For simplicity of notation, we describe here the methods using a constant $\dt$ 
although in practice it is preferable to use adaptive time steps based on the CFL condition.% \footnote{Although it is preferable to dynamically adjust $\Delta t$ at each time step based on a CFL condition, for simplicity, a constant time step $\Delta t$ is adopted here to streamline notation in the method's description.}. 

Concerning the spatial discretization, we consider a staggered division of a 1D computational domain as follows. We assume a uniform discretization in $N$ cells or control volumes $I_i$, $1\leq i \leq N$. Conventionally, the cell endpoints are identified using half-indexes, such that the control volume is $I_i=[x_{i-1/2},x_{i+1/2}]$, with the cell center located at $x_i=(x_{i+1/2}+x_{i-1/2})/2$. The cell size, denoted by $\Delta x_i = x_{i+1/2}-x_{i-1/2}$, is typically uniform along the mesh, simply referred to as $\dx$.  Moreover, according to the staggered discretization, we introduce the dual cell $I_{i+1/2}=[x_{i},x_{i+1}]$. Then, we denote by $\eta_i^n$ (analogously $h_i^n,z_{b,i}^n$) and $q_{i+1/2}^n$ the approximation of the mean values over cells $I_i$ and $I_{i+1/2}$, respectively, at time $t=t^n$, i.e.,
\[
\eta_i^n\cong \frac{1}{\Delta x}\int_{I_i}\eta(x,t^n)\, dx,\qquad q_{i+1/2}^n\cong \frac{1}{\Delta x}\int_{I_{i+1/2}}q(x,t^n)\, dx. 
\]

By integrating the mass equation in \eqref{SW} over the cell $I_i$ and the momentum equation over $I_{i+\frac{1}{2}}$, the shallow water system, semi‑discretized in time, takes the form
\begin{align}
    \label{eta_t}
    \frac{d}{dt} \eta_i(t) &= - \frac{\hat{q}_{i+\frac{1}{2}} - \hat{q}_{i-\frac{1}{2}}}{\Delta x},\\
    \label{q_t}
    \frac{d}{dt} q_{i+\frac{1}{2}}(t) &= - \frac{\mathcal{F}_{i+1} - \mathcal{F}_{i}}{\Delta x} - g\,\frac{1}{\Delta x}\int_{I_{i+\frac{1}{2}}} h\,\partial_x\eta\,dx,
\end{align}
where $\hat{q}_{i+\frac{1}{2}}$ denotes the pointwise value approximation  of $q(x)$ at $x_{i+\frac{1}{2}}$, corresponding to the numerical flux for the mass equation at $x_{i+\frac12}$, and $\mathcal{F}_i$ denotes the numerical flux of the momentum equation at $x_i$.

For a sufficiently smooth function $w(x)$, the difference between its cell average and its pointwise value is $O(\Delta x^2)$. If we let $w_{i+\frac12}$ be the cell average of $w(x)$ on $I_{i+\frac12}$,
then
	\begin{align*}
		%\hat{w}_{i+\frac12}
		w(x_{i+\frac12}) &= w_{i+\frac12} - \frac{\Delta x^2}{24}\,w''(x_{i+\frac12}) + O(\Delta x^4) \\
		&= w_{i+\frac12} - \frac{1}{24}\Bigl(w_{i+\frac32} - 2w_{i+\frac12} + w_{i-\frac12}\Bigr) + O(\Delta x^4).
\end{align*}
Applying this relation to the flux yields
\[
\hat{q}_{i+\frac{1}{2}} = q_{i+\frac{1}{2}} - \frac{1}{24}\Bigl(q_{i+\frac{3}{2}} - 2q_{i+\frac{1}{2}} + q_{i-\frac{1}{2}}\Bigr),
\]

which, when substituted into \eqref{eta_t}, gives the fourth‑order accurate spatial discretization
\[
\frac{d}{dt}\eta_i(t)
= - \frac{q_{i+\frac{1}{2}} - q_{i-\frac{1}{2}}}{\Delta x}
  + \frac{1}{24\,\Delta x}\Bigl(q_{i+\frac{3}{2}} - 3q_{i+\frac{1}{2}} + 3q_{i-\frac{1}{2}} - q_{i-\frac{3}{2}}\Bigr).
\]

By treating the discharge $q$ implicitly while keeping the flux $\mathcal{F}$ explicit in time, one obtains a quasi‑discrete, fully implicit first‑order (in time) scheme for the Saint‑Venant system:
\begin{align*}    
    \eta_i^{n+1} &= \eta_i^n - \frac{\Delta t}{\Delta x}\bigl(q_{i+\frac{1}{2}}^{n+1} - q_{i-\frac{1}{2}}^{n+1}\bigr)
      + \frac{\Delta t}{24\,\Delta x}\,\Delta_3 q^{n+1}_i, \\
    q_{i+\frac{1}{2}}^{n+1}
    &= q_{i+\frac{1}{2}}^{n} - \frac{\Delta t}{\Delta x}\bigl(\mathcal{F}_{i+1}^{n+1} - \mathcal{F}_{i}^{n+1}\bigr)
      - \frac{\Delta t}{\Delta x}\,g \int_{I_{i+\frac{1}{2}}} h^{n+1}\,\partial_x\eta^{n+1}\,dx,
\end{align*}
where
\[
\Delta_3q_i
\equiv q_{i+\frac{3}{2}} - 3q_{i+\frac{1}{2}} + 3q_{i-\frac{1}{2}} - q_{i-\frac{3}{2}}
\]
approximates $\Delta x^3\,q_i'''$.

To improve efficiency by avoiding solving a fully implicit system, one may treat certain linear corrections explicitly. This leads to the following first‑order, semi‑implicit scheme:
\begin{align}    
    \eta_i^{n+1} &= \eta_i^n - \frac{\Delta t}{\Delta x}\bigl(q_{i+\frac{1}{2}}^{n+1} - q_{i-\frac{1}{2}}^{n+1}\bigr)
      + \frac{\Delta t}{24\,\Delta x}\,\Delta_3 q^{n}_i,
      \label{real_third_ord_eta}\\ \label{real_third_ord_q}
    q_{i+\frac{1}{2}}^{n+1}
    &= q_{i+\frac{1}{2}}^{n} - \frac{\Delta t}{\Delta x}\bigl(\mathcal{F}_{i+1}^n - \mathcal{F}_{i}^n\bigr)
      - \frac{\Delta t}{\Delta x}\,g \int_{I_{i+\frac{1}{2}}} h^n\,\partial_x\eta^{n+1}\,dx.
\end{align}

In this formulation, the nonlinear flux $\mathcal{F}$ is evaluated explicitly, the height $h$ in the term $g\,h\,\p_x\eta$ is taken at time level $n$, and the free‑surface gradient $\p_x\eta$ is treated implicitly. Moreover, in the $\eta$ equation, the primary flux difference is handled implicitly, while the higher‑order correction $\Delta_3q_i/\Delta x$ remains explicit, producing a penta-diagonal linear system instead of a hepta-diagonal one\footnote{As will be shown in detail in Section \ref{Sec:space_reconst} (see Eq.~\eqref{elliptic_eta}), the semi‑implicit treatment entails solving a linear system. In particular, treating the term $\Delta_3q_i/\Delta x$ implicitly leads to a hepta-diagonal matrix rather than the penta-diagonal one obtained under explicit treatment. Remarkably, this narrower‑band matrix still permits the scheme to achieve third‑order accuracy. Further details on its structure and derivation will be presented in the next section.}. A detailed derivation of the resulting linear system matrix is given in Section \ref{Sec:space_reconst}.

Once the terms requiring explicit or implicit treatment are identified, high‑order IMEX time integrators can be employed to achieve higher accuracy in time (see Section \ref{sec:time_recons} and \cite{Boscarino-Filbet,BoscarinoSemplice}).

As we shall see in Section \ref{ssec:accuracy}, the effect of the term $\Delta_3q_i^{n}$ is usually very small, and could be neglected in practical applications, therefore one may also consider a simplified version of the previous system, which ignores this term:
%so that Nonetheless, numerical experiments indicate that neglecting the correction term $\Delta_3(q,x_i)/\Delta x$ still yields overall third‑order accuracy 
%(see Section \ref{sec:numerics}). Consequently, we focus on the following first‑order, quasi‑discrete semi‑implicit Saint‑Venant scheme:
\begin{align}
    \label{eta_1}
    \eta_i^{n+1} &= \eta_i^n - \frac{\Delta t}{\Delta x}\bigl(q_{i+\frac{1}{2}}^{n+1} - q_{i-\frac{1}{2}}^{n+1}\bigr),\\
    \label{q_1}
    q_{i+\frac{1}{2}}^{n+1} &= q_{i+\frac{1}{2}}^{n} - \frac{\Delta t}{\Delta x}\bigl(\mathcal{F}_{i+1}^n - \mathcal{F}_{i}^n\bigr)
      - \frac{\Delta t}{\Delta x}\,g \int_{I_{i+\frac{1}{2}}} h^n\,\partial_x\eta^{n+1}\,dx.
\end{align}
Systems \eqref{real_third_ord_eta}-\eqref{real_third_ord_q} (hereafter referred as the fully third-order scheme) and \eqref{eta_1}–\eqref{q_1} (referred as the simplified scheme) constitute first‑order (in time), quasi‑discrete\footnote{A fully discrete scheme results upon discretizing the integral term, as detailed in Section \ref{Sec:space_reconst}.} semi‑implicit methods for the one‑dimensional shallow water equations. A linearization of the pressure term—freezing $h$ at time $t^n$—avoids solving a nonlinear system at each time step.

Following \cite{Casulli1992}, once the pressure‑integral in the second equation of each system is suitably approximated, one substitutes the resulting $q_{i+\frac{1}{2}}^{n+1}$ into the first equation, yielding a linear system for the free‑surface elevations. The key to attaining third‑order spatial accuracy lies in the precise quadrature used for the pressure integral. This is developed in the next section. From now on, for simplicity, we consider system 
\eqref{eta_1}–\eqref{q_1} in the description of the details of the scheme. The extension to system \eqref{real_third_ord_eta}-\eqref{real_third_ord_q} is straightforward.

\subsection{Third order space reconstruction} \label{Sec:space_reconst}
Let's direct our attention to the momentum equation \eqref{q_1}. Concretely, we focus on the integral over $I_{i+1/2}$, which is the more subtle part of the scheme. 
%Here, $\F_i^n$ denotes the Rusanov-flux $\F\Bigl(U_{i}^-;U_i^+\Bigr)$ approximating the convective term, derived from the third-order CWENO reconstructed values $U_i^{\pm}$, as expressed by
%\begin{equation}
%    \label{rusanov}
%    \F_i^n = \frac{1}{2}\left( f(U_{i}^+) + f(U_i^-) - \alpha_i(U_i^+ - U_i^-)\right),
%\end{equation}
%Here both $h_{i}^\pm,q_i^\pm$ are computed the CWENO reconstructions from the values $h_{i+1/2},q_{i+1/2}$.
%\begin{remark}
%\label{rem_h_bar}
%    It's noteworthy that $q$ pertains to the edges while $h$ to the centers. Consequently, to address this distinction, $\hb$ is introduced through third-order upwind interpolation from $h$, representing $h$ at the edges. 
%\end{remark}
%Meanwhile, the integral over $I_{i+1/2}$ 
The idea is to consider polynomial approximation of $h$ and $\p_x\eta$ of appropriate degree that allows us to approximate this integral up to third order. To do so, we consider a second order polynomial for the height and a third-order polynomial for the free surface. A key aspect is to define the stencil for the polynomial reconstruction taking into account that a staggered mesh is used, being $h_i$ and $\eta_{i}$ the cell averaged approximation on the control volume $I_i$. To overcome this difficulty, the integral in \eqref{q_1} is defined as 
\begin{multline*}
    \int_{I_{i+\ha}} h^n\p_x\eta^{n+1}dx = \int_{x_{i}}^{x_{i+\ha}} h^n\p_x\eta^{n+1}dx + \int_{x_{i+\ha}}^{x_{i+1}} h^n\p_x\eta^{n+1}dx \\ %\approx \\ &
    \approx \int_{x_{i}}^{x_{i+\ha}} P^{h,n}_i \p_x P^{\eta,n+1}_{i+\ha}dx + \int_{x_{i+\ha}}^{x_{i+1}} P^{h,n}_{i+1}\p_x P^{\eta,n+1}_{i+\ha}dx. \nonumber
\end{multline*}
where $P^h(x),P^\eta(x)$ are polynomial of an appropriate degree reconstructing of the height and the free surface variables, with $P^h(x)$ evaluated at time step $n$, while $P^\eta(x)$ is evaluated at time step $n+1$ (see sketch in Figure \ref{skecth_scheme}). 

\begin{figure}[!ht]
	\centering
	\includegraphics[width=0.6\textwidth]{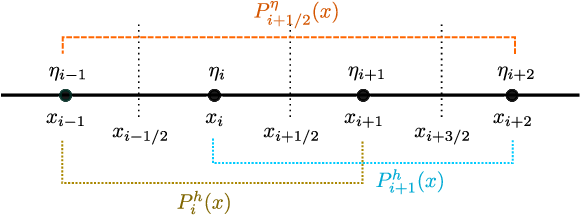}
	\caption{Sketch of the polynomial reconstruction $P^h_i(x),P^h_{i+1}(x),P^\eta_{i+1/2}(x)$ and the stencils around $x_{i+1/2}$.}
	\label{skecth_scheme}
\end{figure}

Thus, the polynomial for the height reconstruction is defined as
$$
    P^h_i(x) = a_{0,i}  + a_{1,i} (x-x_i) + a_{2,i} (x-x_i)^2,
$$
the second order polynomial of the CWENO reconstruction in terms of $\{h_j\}_{j=i-1,i,i+1}$. Concretely, defining the usual CWENO coefficients $\omega^L_i,\omega^C_i,\omega^R_i$ and $d^L,d^C,d^R$ (see Appendix \ref{App_B}), they are
\begin{align}
    a_{2,i}  &= \frac{\omega^C_i}{2d^C \dx^2}\(h_{i-1} - 2h_i + h_{i+1}\), \\ 
    a_{1,i}  &= \frac{\omega^L_i}{\dx}(h_i - h_{i-1}) - \frac{\omega^R_i}{\dx}(h_i - h_{i+1}) - \frac{\omega^C_i}{\dx d^C} 
    \( \ha(h_{i-1} - h_{i+1}) + d^L(h_i - h_{i-1}) - d^R(h_i - h_{i+1})\), \\
    a_{0,i}  &= h_i\omega^L_i + h_i\omega^R_i - \frac{\omega^C_i}{d^C}\(\frac{1}{24}(h_{i-1} - 26h_i + h_{i+1}) + d^Lh_i + d^Rh_i\)  .
\end{align}
The polynomial $P^h_{i+1}(x)$ is analogously defined. 

For the free surface variable we need a third-order polynomial in order to get a second order polynomial approximating $\p_x\eta$. Then, we consider the polynomial
\begin{multline*}\label{eq:pol_eta}
	P^{\eta}_{i+\ha}(x) = b_{0,i+\ha} + b_{1,i+\ha}(x-x_{i+\ha}) + b_{2,i+\ha}\((x-x_{i+\ha})^2 - \frac{\dx^2}{12}\) +   b_{3,i+\ha}(x-x_{i+\ha})^3,  
\end{multline*}     
where the term $\dx^2/12$ is introduced in order to have 
\[
\dfrac{1}{\dx}\displaystyle\int_{I_{i+\frac12}} P^{\eta}_{i+1/2}(x)dx = b_{0,i+\ha}
\]
whose coefficients are computed in such a way that
$$
\frac{1}{\Delta x} \int_{I_j} P^\eta_{i+1/2}(x)=\eta_{j}, \qquad  \mbox{for} \qquad j=i-1,i,i+1,i+2.
$$
Concretely, after some algebraic computations, we get
$$
b_{0,i+\ha}=\frac{-\eta_{i-1}+9 \eta_i+9 \eta_{i+1}-\eta_{i+2}}{16},
\qquad 
b_{1,i+\ha}=\frac{\eta_{i-1}-15 \eta_i+15 \eta_{i+1}-\eta_{i+2}}{12 \Delta x},
$$
$$
b_{2,i+\ha}=\frac{\eta_{i-1}- \eta_i- \eta_{i+1}+\eta_{i+2}}{4 \Delta x^2},
\qquad 
b_{3,i+\ha}=\frac{-\eta_{i-1}+3 \eta_i-3 \eta_{i+1}+\eta_{i+2}}{6 \Delta x^3}.
$$ 

\noindent Let be $q^*$ the explicit part of the momentum equation, 
\begin{equation}
    \label{q_star}
    q^*_{i+\ha} = q_{i+\ha}^n - \frac{\dt}{\dx}\Bigl(\F_{i+1}^n-\F_{i}^n\Bigr).
\end{equation}
then the discrete momentum equation reads
\begin{equation}
	q_{i+\ha}^{n+1} = q_{i+\ha}^* - \frac{\dt}{\dx}g \left( \int_{x_{i}}^{x_{i+\ha}} P^h_i \partial_xP^{\eta}_{i+\ha}dx + \int_{x_{i+\ha}}^{x_{i+1}} P^h_{i+1}\partial_xP^{\eta}_{i+\ha}dx\right) .    \label{eq:int_press}
\end{equation}

Once the polynomials $P^h_i(x),P^h_{i+1}(x),P^\eta_{i+1/2}(x)$ are defined, the integrals in the previous equation can be written in terms of the free surface values. For the sake of completeness, we write in details all this coefficients in what follows.

Defining the constants values
\begin{equation}\label{eq:constants}
    \zeta = \dx\left(\dfrac14 - \frac{\sqrt{3}}{12}\right), \qquad \tau = \dx\left(\dfrac14 + \frac{\sqrt{3}}{12}\right),
\end{equation}
we consider 
$$\beta^\pm_\delta = \frac{\delta}{2\dx^2}\pm \dfrac{\delta^2}{2\dx^3}, \qquad \gamma^\pm_\delta = \frac{\delta}{2\dx^2}\pm 3\dfrac{\delta^2}{2\dx^3},
$$
and $\alpha^{\delta,L/R}_{i+\ha}$  as
$$
\begin{array}{l}
	\alpha^{\delta,L}_{i+\ha} =\dfrac{1}{4} \(a_{0,i}  + a_{1,i}  \, \delta + a_{2,i}  \, \delta^2\),\\ \\
	\alpha^{\delta,R}_{i+\ha} =\dfrac{1}{4}\(a_{0,i+1}  - a_{1,i+1}  \, \delta + a_{2,i+1}  \, \delta^2\),
\end{array}
$$
for $\delta \in \{\zeta, \tau\}$. Now, the following notation is introduced: $\xi^{L/R}_{i+\ha,j}$ for $j \in \{i-1,i,i+1,i+2\}$ represents the factor that multiplies  $\eta_j^{n+1}$ coming from the left $(L)$ and right $(R)$ integrals in \eqref{eq:int_press}, respectively, when discretizing the momentum equation integrated on $I_{i+1/2}$. They are defined as follows,
%$\xi^j_{i+\ha,i+r}$ for $j \in \{i,i+1\}$, $r\in \{-1,0,1,2\}$, 
\begin{equation} \label{eq:defxiL}
    \begin{array}{lcl}      
    \xi^L_{i+\ha,i-1} &=& \alpha^{\zeta,L}_{i+\ha}\left(\frac{1}{12\dx}-\beta_\tau^+\right) + \alpha^{\tau,L}_{i+\ha}\left(\frac{1}{12\dx}-\beta_\zeta^+\right),\\
    \xi^L_{i+\ha,i} &=& \alpha^{\zeta,L}_{i+\ha}\left(\frac{-5}{4\dx}+\gamma_\tau^+\right) + \alpha^{\tau,L}_{i+\ha}\left(\frac{-5}{4\dx}+\gamma_\zeta^+\right),\\
    \xi^L_{i+\ha,i+1} &=& \alpha^{\zeta,L}_{i+\ha}\left(\frac{5}{4\dx}+\gamma_\tau^-\right) + \alpha^{\tau,L}_{i+\ha}\left(\frac{5}{4\dx}+\gamma_\zeta^-\right),\\
    \xi^L_{i+\ha,i+2} &=& \alpha^{\zeta,L}_{i+\ha}\left(\frac{-1}{12\dx}-\beta_\tau^-\right) + \alpha^{\tau,L}_{i+\ha}\left(\frac{-1}{12\dx}-\beta_\zeta^-\right),  
\end{array} 
\end{equation}
and
\begin{equation} \label{eq:defxiR}
\begin{array}{lcl}      
	\xi^R_{i+\ha,i-1} &=& \alpha^{\zeta,R}_{i+\ha}\left(\frac{1}{12\dx}+\beta_\tau^-\right) + \alpha^{\tau,R}_{i+\ha}\left(\frac{1}{12\dx}+\beta_\zeta^-\right),\\
	\xi^R_{i+\ha,i} &=& \alpha^{\zeta,R}_{i+\ha}\left(\frac{-5}{4\dx}-\gamma_\tau^-\right) + \alpha^{\tau,R}_{i+\ha}\left(\frac{-5}{4\dx}-\gamma_\zeta^-\right),\\
	\xi^R_{i+\ha,i+1} &=& \alpha^{\zeta,R}_{i+\ha}\left(\frac{5}{4\dx}-\gamma_\tau^+\right) + \alpha^{\tau,R}_{i+\ha}\left(\frac{5}{4\dx}-\gamma_\zeta^+\right),\\
	\xi^R_{i+\ha,i+2} &=& \alpha^{\zeta,R}_{i+\ha}\left(\frac{-1}{12\dx}+\beta_\tau^+\right) + \alpha^{\tau,R}_{i+\ha}\left(\frac{-1}{12\dx}+\beta_\zeta^+\right).  
\end{array}
\end{equation}

%
%\chi^R_{i+1/2,i-1} = ( azetaR*( 1/(12*dx) + gammaTau_m) + atauR*( 1/(12*dx) + gammaZeta_m) )*etam1
%\chi^R_{i+1/2,i}   = ( azetaR*(-5/(4*dx)  - gammaTau_m) + atauR*(-5/(4*dx)  - gammaZeta_m) )*etai 
%\chi^R_{i+1/2,i+1} = ( azetaR*( 5/(4*dx)  - gammaTau_p) + atauR*( 5/(4*dx)  - gammaZeta_p) )*etap1 
%\chi^R_{i+1/2,i+2} = ( azetaR*(-1/(12*dx) + gammaTau_p) + atauR*(-1/(12*dx) + gammaZeta_p) )*etap2

Thus, the discrete moment equation \eqref{eq:int_press} reads
%\begin{align}
%	\label{q_new}
%	q_{i+\ha}^{n+1} = q_{i+\ha}^* -&\dt g\Bigg((\xi_{i+\ha,i-1}^L+\xi_{i+\ha,i-1}^R ) \eta_{i-1}^{n+1} + (\xi_{i+\ha,i}^L +\xi_{i+\ha,i}^R)\eta_{i}^{n+1} \\ \nonumber 
%	& \qquad + (\xi_{i+\ha,i+1}^L +\xi_{i+\ha,i+1}^R)\eta_{i+1}^{n+1} + (\xi_{i+\ha,i+2}^L +\xi_{i+\ha,i+2}^R)\eta_{i+2}^{n+1}   \Bigg).    
%\end{align}
\begin{multline}
	\label{q_new}
	q_{i+\ha}^{n+1} = q_{i+\ha}^* \ -\ \dt g\Bigg((\xi_{i+\ha,i-1}^L+\xi_{i+\ha,i-1}^R ) \eta_{i-1}^{n+1} + (\xi_{i+\ha,i}^L +\xi_{i+\ha,i}^R)\eta_{i}^{n+1} \\  
	 + (\xi_{i+\ha,i+1}^L +\xi_{i+\ha,i+1}^R)\eta_{i+1}^{n+1} + (\xi_{i+\ha,i+2}^L +\xi_{i+\ha,i+2}^R)\eta_{i+2}^{n+1}   \Bigg).    
\end{multline}

Now, these values of $q_{i+\ha}$ are replaced in \eqref{eta_1}, obtaining the elliptic equation for $\eta$
\begin{align}
    \nonumber
    & \eta_{i-2}^{n+1} \, \frac{\dt^2}{\dx}g\(\xi^L_{i-\ha,i-2} +\xi^R_{i-\ha,i-2} \) 
    \\ \nonumber
    + \, &\eta_{i-1}^{n+1} \, \frac{\dt^2}{\dx}g\( \xi^L_{i-\ha,i-1} + \xi^R_{i-\ha,i-1} - \xi^L_{i+\ha,i-1} - \xi^R_{i+\ha,i-1} \)  
    \\ \nonumber
    +\, &\eta_{i}^{n+1} \, \bigg( 1+ \frac{\dt^2}{\dx}g\( \xi^L_{i-\ha,i} + \xi^R_{i-\ha,i} - \xi^L_{i+\ha,i} - \xi^R_{i+\ha,i} \)  \bigg)
    \\ \nonumber
    + \, &\eta_{i+1}^{n+1} \, \frac{\dt^2}{\dx}g\( \xi^L_{i-\ha,i+1} + \xi^R_{i-\ha,i+1} - \xi^L_{i+\ha,i+1} - \xi^R_{i+\ha,i+1} \) 
    \\
    \label{elliptic_eta}
    + \, & \eta_{i+2}^{n+1} \, \frac{\dt^2}{\dx}g\(-\xi^L_{i+\ha,i+2} -\xi^R_{i+\ha,i+2} \) = \eta_i^*
\end{align}
where $\eta_{i}^*$ collects the explicit contributions
\begin{equation}
\label{eq:etaStar}
\eta_{i}^* = \eta_{i}^n - \frac{\dt}{\dx}\(q^{*}_{i+\ha} - q^*_{i-\ha}\). 	
\end{equation} 
Once $\eta_{i}^{n+1}$, $i=1,\dots,N$ has been computed, then $q_{i+1/2}^{n+1}$, $ i=0,\dots,N$, follows from equation \eqref{q_new}.

\subsubsection{Convective term computation} \label{Aux_h}
In this subsection, we detail how the convective term $\F_i^n$ in \eqref{q_1} is computed. First, the fact of using the staggered discretization makes it necessary to interpolate the height values at the interfaces $h_{i+1/2}$. To this aim, we consider a third-order upwind interpolation from $h_i$ as follows. We construct the quadratic polynomials $P_{i+1/2}^L$,$P_{i+1/2}^R$ such that
$$\dfrac{1}{\Delta x} \displaystyle \int_{I_j} P^L_{i+1/2}(x)=h_{j}, \quad 
j=i-1,i,i+1; \qquad \dfrac{1}{\Delta x} \displaystyle \int_{I_k} P^R_{i+1/2}(x)=h_{k}, \quad k=i,i+1,i+2.
$$
Now, the interpolated height $h_{i+1/2}$ is defined as the upwind value, i.e.,
\begin{equation*}
	h_{i+\ha} = \begin{cases}
		P_{i+\ha}^L(x_{i+\ha}) \quad {\rm if }\; q_{i+\ha}\ge0  ,     \\
		P_{i+\ha}^R(x_{i+\ha}) \quad {\rm if }\; q_{i+\ha}<0   .     
	\end{cases}
\end{equation*}
where
\begin{align*}
	P_{i+\ha}^L(x_{i+\ha}) &= h_{i-1} + \frac{3}{2}\(h_{i} - h_{i-1}\) + \frac{3}{8}\(h_{i+1} - 2h_{i} + h_{i-1} \), \\    
	P_{i+\ha}^R(x_{i+\ha}) &= h_{i} + \frac{1}{2}\(h_{i+1} - h_{i}\) - \frac{1}{8}\(h_{i+2} - 2h_{i+1} + h_{i} \).
\end{align*}
 
Once we have $h_{i+1/2},q_{i+1/2}$,  the term $\F_i^n$ is computed using the Rusanov flux, as
%\begin{equation}
%	\label{rusanov}
%	\F_i = \F\Bigl(U_{i}^{-};U_i^{+}\Bigr) = \frac{1}{2}\left( f(U_{i}^+) + f(U_i^-) \,-\, \alpha_i(U_i^+ - U_i^-)\right),
%\end{equation}
\begin{equation}
	\label{rusanov}
	\F_i = \F\Bigl(h_i^-,q_i^-,h_i^+,q_i^+\Bigr) = \frac{1}{2}\left( f(h_i^+,q_i^+) + f(h_i^-,q_i^-) \,-\, \alpha_i(q_i^+ - q_i^-)\right),\qquad i=1,...,N,
\end{equation}
where
\[
    f(h,q) \equiv \frac{q^2}{h}
\]
and  $h_{i}^{\pm},q_{i}^{\pm}$ are the reconstructed values from $h_{i+1/2},q_{i+1/2}$ by the usual third-order CWENO procedure \cite{levy1999central}, recalled in Appendix \ref{App_B},  and $\alpha_i=\max\{|u_{i}^-|,|u_{i}^+|\}$ an approximation of the maximum wave speed.

%As previously mentioned, due to the disparity in the point-wise definitions of $q$ and $h$, an auxiliary reconstruction is necessary to determine $h$ at the edges of a volume $I_i$. Thus, it's worth noting that, for a third-order reconstruction, an upwinding solution can be employed between a third-order left polynomial $P_{i+\frac{1}{2}}^L(x)$ and a third-order right polynomial $P_{i+\frac{1}{2}}^R(x)$ evaluated at $x_{i+\frac{1}{2}}$. \\
%In practice,
%\begin{align*}
%    P_{i+\ha}^L(x_{i+\ha}) &= h_{i-1} + \frac{3}{2}\(h_{i} - h_{i-1}\) + \frac{3}{8}\(h_{i+1} - 2h_{i} + h_{i-1} \); \\    
%    P_{i+\ha}^R(x_{i+\ha}) &= h_{i} + \frac{1}{2}\(h_{i+1} - h_{i}\) - \frac{1}{8}\(h_{i+2} - 2h_{i+1} + h_{i} \);
%\end{align*}
%and 
%\begin{equation*}
%    \hb_{i+\ha} = \begin{cases}
%        P_{i+\ha}^L(x_{i+\ha}) \quad {\rm if }\; u_{i+\ha}\ge0       \\
%        P_{i+\ha}^R(x_{i+\ha}) \quad {\rm if }\; u_{i+\ha}<0        
%    \end{cases}
%\end{equation*}

\subsection{Third order time discretization} \label{sec:time_recons}
High order in time is obtained by adopting Implicit-Explicit Runge-Kutta methods (IMEX-RK) \cite{boscarino2024implicit}.
Such schemes are characterized by a double Butcher table of the form:
\begin{table*}[htbp]
	%\footnotesize
	\begin{center}
		\begin{tabular}{c|c}
       $c_E$ & $A_E$ \\
			\hline
            & \\[-3mm]
			& $b_E^\top$
		\end{tabular} 
		\qquad \qquad \qquad
		\begin{tabular}{c|c}
       $c_I$ & $A_I$ \\
			\hline
            & \\[-3mm]
			& $b_I^\top$
		\end{tabular} 
\label{IMEX_tableau}
	\end{center}
\end{table*}

The explicit table is on the left, while the implicit one is on the right. Column vectors $c$ and $b$ have $s$ components, matrix $A_E$ is lower triangular with zero diagonal, while $A_I$ is lower triangular, i.e.\ the implicit scheme is diagonally implicit Runge-Kutta (DIRK). 
        
Drawing from the concept proposed in \cite{Boscarino-Filbet}, where IMEX Runge-Kutta methods are employed for systems exhibiting stiffness that may not conform strictly to an additive or partitioned structure, we express system \eqref{SW} as a comprehensive system of ordinary differential equations (ODEs). Here, we incorporate appropriate discrete operators for approximating spatial derivatives. The pivotal insight from \cite{Boscarino-Filbet} lies in discerning which particular terms justify implicit treatment and which can be handled explicitly. 

Following \cite{Boscarino-Filbet}, we rewrite system \eqref{SW} in the form
\begin{equation}
    \label{ode_bosca_semi}
    U' = H(U,U),
\end{equation}
where we assume that the dependence of $H$ on the first argument is non stiff, while the dependence on the second argument is stiff. 
IMEX-RK schemes can be applied in this context 
by making use of two sets of stage values, as follows:
    \begin{itemize}
        \item Computation of the stage values: For $\bi=1,\ldots,s$ compute
        \begin{align*}
            U^{(\bi)}_E & = U^n + \Delta t\sum_{\bj=1}^{\bi-1}a_{\bi,\bj}^E{H}\left(U_E^{(\bj)},U_I^{(\bj)}\right)\\
            U^{(\bi)}_I & = U^n + \Delta t
            \sum_{\bj=1}^{\bi-1}a_{\bi,\bj}^I {H}\left(U_E^{(\bj)},U_I^{(\bj)}\right) + \Delta t \, a_{l,l}^I {H}\left(U_E^{(\bi)},U_I^{(\bi)}\right).
        \end{align*}
        \item Computation of the numerical solution:
        $$U^{n+1} = U^{n} + \dt\sum_{\bi = 1}^s b_\bi{H}\(U_E^{(\bi)},U_I^{(\bi)}\).$$
    \end{itemize}
Note that, because of the triangular structure, the stages are not coupled. 
Also, note that one set of stage fluxes, 
${H}\left(U_E^{(\bj)},U_I^{(\bj)}\right)$, appears in the method, which implies that there is only one vector $b = b^E = b^I$ which will be used in the evaluation of the numerical solution. For more information on this approach, please consult 
\cite{Boscarino-Filbet} or \cite{boscarino2024implicit}, Sec.~3.3.
In light of what exposed above,  
 system \eqref{SW} can be written in the form 
\eqref{ode_bosca_semi}, 
 with $U = [\eta,q]^\top$ and 
 \begin{equation}
    \label{ode_form_semi}
    H(U_E,U_I) = 
    \begin{bmatrix}
        -\p_x q_I \\
        -\p_x(q_Eu_E) - gh_E\p_x \eta_I 
    \end{bmatrix}
\end{equation}
After space discretization, the system can be written in the form
\[  
    U' = \widetilde{H}(U,U)
\]
 
% in both additive and partitioned form in order to apply the IMEX strategy. Here we consider the partitioned form. First, system \eqref{SW} is written in the form
%\begin{equation}
%    \label{ode_bosca_semi}
%    U' = H(U).
%\end{equation}
%where $U = [\eta, q]^T$ and $H(U)$ is given by
%\begin{equation}
%    \label{ode_form_semi}
%    H(U) = 
%    \begin{bmatrix}
%        -\p_x q \\
%        -\p_x(qu) - gh\p_x \eta 
%    \end{bmatrix}
%\end{equation}
%%with the subscript $E$ and $I$ denoting which term has to be treated explicitly and which implicitly. 
%
%Then, by doubling the variables, the semi-implicit scheme can be written in the form
%    \begin{equation}
%        \label{ode_bosca}
%        [U_E,U_I]' = \widetilde{H}(U_E,U_I).
%    \end{equation}
%
%

with
\begin{equation}
    \label{ode_form_part}
    \widetilde{H}(U_E,U_I) = 
    \begin{bmatrix}
       0   & -D^1_x(q_I) \\ - D^3_x(q_E \, u_E) &  -gh_E D^2_x(\eta_I) 
    \end{bmatrix},
\end{equation}
where $D_x^k$ denote suitable discrete approximation of $\partial_x$. In our case, $D_{x}^1(q_I)$ is defined at each control volume $I_i$;  $(h_ED_x^2(\eta_I))$ and $D^3_x((qu)_E)$ are defined at each dual control volume $I_{i+1/2}$. Concretely, they are:
\label{page:D}
\begin{itemize}
  	\item $D_{x}^1(q_I) = \dfrac{q_{i+\ha} - q_{i-\ha}}{\dx}$, since $q_{i\pm\ha}$ are suitably defined on cell edges in terms of $q_I$. 
   	\item According to \eqref{q_new}, we have    		
   	\begin{multline*}
    		(h_ED_x^2(\eta_I)){i+\frac12} = (\xi_{i+\frac{1}{2},i-1}^L+\xi_{i+\frac{1}{2},i-1}^R ) \eta_{i-1} + (\xi_{i+\frac{1}{2},i}^L +\xi_{i+\frac{1}{2},i}^R)\eta_{i}  \\
    		+ (\xi_{i+\frac{1}{2},i+1}^L +\xi_{i+\frac{1}{2},i+1}^R)\eta_{i+1} + (\xi_{i+\frac{1}{2},i+2}^L +\xi_{i+\frac{1}{2},i+2}^R)\eta_{i+2},
   	\end{multline*} 
   	with coefficients $\xi$, defined by (\ref{eq:defxiL}) and (\ref{eq:defxiR}), in terms of the explicit variables $h_E$, and $\eta$ values in terms of the implicit variables $\eta_I$.
   	\item $D^3_x(q_E \, u_E)= \dfrac{\mathcal{F}_{i+1}-\mathcal{F}_{i}}{\dx}$, where $\mathcal{F}_{i}$ is the Rusanov flux defined by \eqref{rusanov} in terms of $U_E$. 
\end{itemize}  
    
With this notation, the semi-discrete in time first order semi-implicit scheme can be written as:
    \begin{equation}
    	\begin{cases}
    		\label{First order}
    		\eta^{n+1} = \eta^n - \dt D_x^1(q^{n+1}),\\
    		q^{n+1} = q^n - \dt D_x^3(q^nu^n) - \dt gh^nD_x^2(\eta^{n+1}).
    	\end{cases}
    \end{equation}

Then, in practice, the fully-discrete first order of the semi-implicit scheme for the SW equations can be written as follows:
\begin{align}
    \label{eta_1-3} & \eta_i^{n+1} = \eta_i^{n}  - \frac{\Delta t}{\Delta x}\Bigl(q_{i+\frac{1}{2}}^{n+1} - q_{i-\frac{1}{2}}^{n+1}\Bigr), \\
    \label{q_1-3} & q_{i+\frac{1}{2}}^{n+1} = q_{i+\frac{1}{2}}^{n} - \frac{\Delta t}{\Delta x}\Bigl(\mathcal{F}_{i+1}^{n}-\mathcal{F}_{i}^{n}\Bigr) -\Delta t g\Bigg((\xi_{i+\frac{1}{2},i-1}^{L,n}+\xi_{i+\frac{1}{2},i-1}^{R,n} ) \eta_{i-1}^{n+1} + (\xi_{i+\frac{1}{2},i}^{L,n} +\xi_{i+\frac{1}{2},i}^{R,n})\eta_{i}^{n+1} \nonumber \\  
	& \qquad + (\xi_{i+\frac{1}{2},i+1}^{L,n} +\xi_{i+\frac{1}{2},i+1}^{R,n})\eta_{i+1}^{n+1} + (\xi_{i+\frac{1}{2},i+2}^{L,n} +\xi_{i+\frac{1}{2},i+2}^{R,n})\eta_{i+2}^{n+1}   \Bigg).
\end{align}

With this in mind, we apply an IMEX (Implicit-Explicit) scheme to system \eqref{ode_form_semi}. Here we consider the IMEX SSP3(4,3,3) scheme defined by the double Butcher table \ref{tableau} (see \cite{PareschiRusso,Boscarino-Filbet}). Applying this method to the system, we find $U^{n+1}$ following Algorithm \ref{al:imex}. Notice that only steps 2, 4, 6 and 8 in this algorithm involve solving a five-diagonal linear system, while the rest of steps consist of collecting terms.\\

\begin{table}[htbp]
	%\footnotesize
	\begin{center}
		\begin{tabular}{c|cccc}
			$0$ & $0$ & $0$ & $0$ & $0$  \\
			$0$ & $0$ & $0$ & $0$ & $0$ \\ 
			$1$ & $0$ & $1$ & $0$ & $0$ \\ 
			$1/2$ & $0$ & $1/4$ & $1/4$ & $0$ \\
			\hline
			& $0$ & $1/6$ & $1/6$ & $2/3$
		\end{tabular} 
		\qquad \qquad \qquad
		\begin{tabular}{c|cccc}
			$\ab$ & $\ab$ &  $0$ &  $0$ &  $0$  \\
			 $0$ & $-\ab$ & $\ab$ & $0$ & $0$ \\ 
			 $1$ & $0$ & $1-\ab$ & $\ab$ & $0$ \\ 
			$1/2$ & $\bb$ & $\cb$ & $\db$ & $\ab$ \\
			\hline
			& $0$ & $1/6$ & $1/6$ & $2/3$
		\end{tabular}
		\caption{Butcher tableaux of the explicit (left) and implicit (right) counterparts of the third-order IMEX SSP3(4,3,3) scheme. Here, we use $\ab = 0.24169426078821$, $\bb = 0.06042356519705$, $\cb = 0.12915286960590$ and $\db = 0.5 - \ab -\bb-\cb.$ }\label{tableau}
	\end{center}
\end{table}

%\begin{algorithm}
%	\caption{}
%	\begin{algorithmic}[1]\label{al:imex}
%		\State{$U_E^{(1)} = U^n$,}
%		\State{$U_I^{(1)} = U_1^* + \ab\dt \Hb(U_E^{(1)},U_I^{(1)})$\qquad with \qquad $U_1^* = U^n$, }
%		\State{$U_E^{(2)} = U^n$, }
%		\State{$U_I^{(2)} = U_2^* + \ab\dt \Hb(U_E^{(2)},U_I^{(2)})$\qquad with \qquad $U_2^* = U^n - \widehat{U}_1$}
%		\State{$U_E^{(3)} = U^n + \dfrac{1}{\ab}\widehat{U}_2$, }
%		\State{$U_I^{(3)} = U_3^* + \ab\dt \Hb(U_E^{(3)},U_I^{(3)})$\qquad with \qquad $U_3^* = U^n +\dfrac{(1-\ab)}{\ab}\widehat{U}_2 $, }
%		\State{$U_E^{(4)} = U^n + \dfrac{1}{4\ab}(\widehat{U}_2 +\widehat{U}_3)$,}
%		\State{$U_I^{(4)} = U_4^* + \ab\dt \Hb(U_E^{(4)},U_I^{(4)})$\qquad with \qquad $U_4^* = U^n +\dfrac{1}{\ab}\left(b\widehat{U}_1 + c\widehat{U}_2 + d\widehat{U}_3\right), $}
%		\State{$U^{n+1} = U^n + \dfrac{1}{6\ab}\left(\widehat{U}_2 + \widehat{U}_3 + 4\widehat{U}_4\right).$ }
%	\end{algorithmic}
%	where $\widehat{U}_k =  U_I^{(k)}-U^*_k$, for $k=1,...,4$.
%\end{algorithm}

\hrule\vspace{-2mm}
\begin{algo}{(to apply IMEX SSP3(4,3,3) method in Table \ref{tableau})}\label{al:imex}
	\hrule
	%\caption{}
	\begin{algorithmic}[1]
		\State{$U_E^{(1)} = U^n$,}
		\State{$U_I^{(1)} = U_1^* + \ab\dt \Hb(U_E^{(1)},U_I^{(1)})$\qquad\qquad with \qquad $U_1^* = U^n$, }
		\State{$U_E^{(2)} = U^n$,}
		\State{$U_I^{(2)} = U_2^* + \ab\dt \Hb(U_E^{(2)},U_I^{(2)})$\qquad\qquad with \qquad $U_2^* = U^n - \widehat{U}_1$,}
		\State{$U_E^{(3)} = U^n + \dfrac{1}{\ab}\widehat{U}_2$, }
		\State{$U_I^{(3)} = U_3^* + \ab\dt \Hb(U_E^{(3)},U_I^{(3)})$\qquad\qquad with \qquad $U_3^* = U^n +\dfrac{(1-\ab)}{\ab}\widehat{U}_2 $, }
		\State{$U_E^{(4)} = U^n + \dfrac{1}{4\ab}(\widehat{U}_2 +\widehat{U}_3)$,}
		\State{$U_I^{(4)} = U_4^* + \ab\dt \Hb(U_E^{(4)},U_I^{(4)})$\qquad\qquad with \qquad $U_4^* = U^n +\dfrac{1}{\ab}\left(b\widehat{U}_1 + c\widehat{U}_2 + d\widehat{U}_3\right), $}
		\State{$U^{n+1} = U^n + \dfrac{1}{6\ab}\left(\widehat{U}_2 + \widehat{U}_3 + 4\widehat{U}_4\right).$ }
	\end{algorithmic}
where $\widehat{U}_k =  U_I^{(k)}-U^*_k$, $k=1,...,4$.
\end{algo}
\vspace{-1mm}\hrule

%Applying this scheme to the system and optimizing the evaluations $\Hb$ \cite{Macca2024}, we find $U^{n+1}$ following Algorithm \ref{al:imex}.
%\begin{enumerate}
%    \item $U_E^{(1)} = U^n;$
%    
%    \item $U_I^{(1)} = U^n + \ab\dt \Hb(U_E^{(1)},U_I^{(1)});$
%    
%    \item $U_E^{(2)} =  U^n;$
%    
%    \item $U_I^{(2)} = U_*^1 + \ab\dt\Hb(U_E^{(2)},U_I^{(2)});$
%
%    \item $U_E^{(3)} = U^n + \Bar{U}^2/\ab;$
%
%    \item $U_I^{(3)} = U_*^2 + \ab\dt\Hb(U_E^{(3)},U_I^{(3)}); $
%
%    \item $U_E^{(4)} = U^n + (\Bar{U}^2 + \Bar{U}^3)/(4\ab);$
%    
%    \item $U_I^{(4)} = U_*^3 + \ab\dt\Hb(U_E^{(4)},U_I^{(4)}); $    
%    
%    \item $U^{n+1} = U^n + (\Bar{U}^2 + \Bar{U}^3 + 4\Bar{U^4})/(6\ab);$
%\end{enumerate}
%in which $\Bar{U}$ and $U_*$ are so defined:
%\begin{align*}
%    \Bar{U}^1 &= U_I^{(1)} - U^n \quad\quad \Bar{U}^2 = U_I^{(2)} - U_*^1 \quad \quad \Bar{U}^3 = U_I^{(3)} - U^2_* \quad\quad \Bar{U}^4 = U_I^{(4)} - U^3_* 
%    \\
%    U^1_* & = U^n-\Bar{U}^1  \quad\quad U^2_* = U^n + \frac{1-\ab}{\ab}\Bar{U}^2 \quad \quad U^3_* = U^n + \frac{\bb\Bar{U}^1 + \cb\Bar{U}^2 + \db\Bar{U}^3}{\ab}.
%\end{align*}

\vspace{2mm}
The complete derivation of the high-order fully-discrete scheme is provided in Appendix~\ref{App_C}, specifically in equations~\eqref{l-stages}-\eqref{num_sol_imex}.
\subsection{Numerical property}
In this section some numerical properties of the semi-implicit scheme introduced above have been reported such as: well-balance property and stability.

\subsubsection{Well-balanced property for lake at rest}\label{ssec:wb}
When simulating geophysical flows it is essential to consider well-balanced schemes, especially when the time dependent solution is a small perturbation of a stationary one. In a few words, a scheme is said to be well-balanced for a family of steady states if it preserves such states exactly, or at leat up to a certain order of approximation (see \cite{GomezBueno2023} and references therein). For the shallow water system, it is desirable to obtain numerical methods that are able to exactly preserve lake-at-rest solutions, that is, constant free surface and null velocity.

Here, we establish the well-balanced properties of the first and third-order IMEX schemes for lake-at-rest states. Concretely, the following result holds:
\begin{theorem}
	Let $U(x)=[\bar{\eta},0]^\top$, with $\bar{\eta}\in\mathbb{R}$ constant, a \emph{lake-at-rest} state. Then the semi-implicit scheme described by \eqref{q_star}, \eqref{q_new}, \eqref{elliptic_eta}, \eqref{eq:etaStar} is well-balanced for $U(x)$.
\end{theorem}
\begin{proof}
	Assuming we start from the state $U(x)$, we have $\eta_i^n = \bar{\eta}$, $q_{i+\ha}^n =  0$, and consequently $\mathcal{F}_{i}^n,\mathcal{F}_{i+1}^n$ are also zero.
	Thus, we trivially get
	$$q_{i+\ha}^* = 0 \qquad\mbox{and}\qquad \eta_i^* = \bar{\eta}.$$
	Looking at the coefficient $\xi_{i\pm\frac{1}{2},j}^{L/R}$, we observe that 
	$$
		\begin{array}{ccrcr}      
			\xi^L_{i+\ha,i-1}+\xi^R_{i+\ha,i-1} &=& \left(\frac{1}{12\dx}-\frac{\tau^2}{\dx^3}\right)\left(\alpha^{\zeta,L}_{i+\ha}+\alpha^{\zeta,R}_{i+\ha}\right) &+& \left(\frac{1}{12\dx}-\frac{\zeta^2}{\dx^3}\right)\left(\alpha^{\tau,L}_{i+\ha}+\alpha^{\tau,R}_{i+\ha}\right),\\[2mm]
			\xi^L_{i+\ha,i}+\xi^R_{i+\ha,i} &=& \left(\frac{-5}{4\dx}+\frac{\tau^2}{\dx^3}\right)\left(\alpha^{\zeta,L}_{i+\ha}+\alpha^{\zeta,R}_{i+\ha}\right) &+& \left(\frac{-5}{4\dx}+\frac{\zeta^2}{\dx^3}\right)\left(\alpha^{\tau,L}_{i+\ha}+\alpha^{\tau,R}_{i+\ha}\right),\\[2mm]
			\xi^L_{i+\ha,i+1}+\xi^R_{i+\ha,i+1} &=& \left(\frac{5}{4\dx}-\frac{\tau^2}{\dx^3}\right)\left(\alpha^{\zeta,L}_{i+\ha}+\alpha^{\zeta,R}_{i+\ha}\right) &+& \left(\frac{5}{4\dx}-\frac{\zeta^2}{\dx^3}\right)\left(\alpha^{\tau,L}_{i+\ha}+\alpha^{\tau,R}_{i+\ha}\right),\\[2mm]
			\xi^L_{i+\ha,i+2}+\xi^R_{i+\ha,i+2} &=& \left(-\frac{1}{12\dx}+\frac{\tau^2}{\dx^3}\right)\left(\alpha^{\zeta,L}_{i+\ha}+\alpha^{\zeta,R}_{i+\ha}\right) &+& \left(\frac{-1}{12\dx}+\frac{\zeta^2}{\dx^3}\right)\left(\alpha^{\tau,L}_{i+\ha}+\alpha^{\tau,R}_{i+\ha}\right),
		\end{array}
	$$
    where the coefficients are defined in Section \ref{Sec:space_reconst}.
    From the above relations it follows that 
	\begin{equation}
		\label{eq:coef_proof}
	\displaystyle \sum_{j=i-1}^{i+2} \left(\xi^L_{i+\ha,j}+\xi^R_{i+\ha,j}\right) = 0.
\end{equation}
The elliptic equation \eqref{elliptic_eta} can be written in vector notation as:
\[
\Bigl(I + \mathcal{A} \Bigr)[\eta_1^{n+1}, \dots, \eta_N^{n+1}]^T = [\bar{\eta}, \dots, \bar{\eta}]^T,
\]
where $I$ denotes the identity matrix, and $\mathcal{A} $ is a five-diagonal matrix, with elements $O(g(\Delta t/\Delta x)^2)$. Assuming $ h_i^n > 0 $, this system forms a regular pentadiagonal linear system. Since $\mathcal A$ is a zero row sum matrix, $I+\mathcal A$ is  strictly diagonally dominant. By the Gershgorin-Hadamard theorem \cite{horn2013}, such a matrix has non-zero eigenvalues, ensuring the system has a unique solution, namely $\eta_i^{n+1} = \bar{\eta}$ for $i = 1, \dots, N$. This result follows from the property that $\mathcal{A}[1, \dots, 1]^T = 0$, which is guaranteed by \eqref{eq:coef_proof}.

Finally, using again \eqref{eq:coef_proof} in \eqref{q_new} with $\eta_i^{n+1} = \bar{\eta}$, $q_{i+\ha}^{\star} = 0$, we get $q_{i+\ha}^{n+1}=0$ and the proof is concluded.
\end{proof}
A similar rationale applies to the third-order IMEX scheme, affirming that each stage maintains the equilibrium state, consequently ensuring the numerical solution remains in equilibrium.

\subsubsection{Stability condition} \label{ssec_stability}
For an explicit scheme, the time step $\Delta t$ is determined at stage $t^n$ by the Courant-Friedrichs-Lewy (CFL) condition as $\Delta t = {\rm CFL}\,\Delta x/\rho_{\rm max}$, where the CFL restriction is governed by the maximum spectral radius of the matrix $A(U)$ defining the hyperbolic system 
\begin{equation}\label{eq:CFL_ex}
    \textrm{CFL} = \frac{\rho^n_{\rm max} \Delta t^n}{\Delta x} \leq C_{\textrm{ex}},
\end{equation}
where 
\[
    \rho^n_{\rm max} = \max_i\rho(A(U^n_i) = \max_i(|u^n_i|+\sqrt{g h^n_i})
\]
%being $\lambda_{\rm max} = \max\abs{\lambda_\pm}$, with $\lambda_{\pm} = u \pm \sqrt{gh}$, 
and $C_{\textrm{ex}}$ is a constant close to one.

In the case of our semi-implicit scheme \eqref{eta_1}-\eqref{q_1}, we empirically derive the following stability condition based on the velocity and not the celerity:
\begin{equation}\label{eq:MCFL}
    \textrm{MCFL} = \frac{u^n_{\rm max} \Delta t^n}{\Delta x} \leq C_{\textrm{im}},
\end{equation}
where $u_{\rm max} = \max_j{|u_j|}$ and $C_{\textrm{im}} \approx 0.75$.

This condition is considerably less restrictive than \eqref{eq:CFL_ex}, since from conditions \eqref{eq:CFL_ex} and \eqref{eq:MCFL} one has 
\[
    \frac{\Delta t_E}{\Delta t_I} = \frac{\rm CFL}{\rm MCFL}\frac{u_{\rm max}}{\rho_{\rm max}} \ll 1
\]
%as the condition for the classical CFL can be expressed as
%\[ 
%     \textrm{CFL} \ =\  \frac{\lambda^n_{\rm max}}{u^n_{\rm max}}\, \textrm{MCFL} \ \leq\  \frac{\lambda^n_{\rm max}}{u^n_{\rm max}} C_{\rm im}
%\]
and $\rho^n_{\rm max}/u^n_{\rm max}\gg 1$ for small Froude numbers.

\section{Application to bedload transport: Saint-Venant-Exner model}\label{sec:num_scheme_exner}
In this section we extend the semi-implicit method described in the previous section to the application of bedload transport, which is modelled through the Saint-Venant-Exner (SVE) system. It consist of the coupling of the shallow water equation \eqref{SW} with the bedload transport equation:
\begin{equation}
    \label{sediment_equation}
    \p_t z_b + \p_x q_b = 0
\end{equation}
where $z_b(x,t)$ denotes the height of the sediment layer and $q_b(h,q)(x,t)$ represents the solid transport discharge, see Figure \ref{fig:set-up}.

For the sake of simplicity, we consider here $q_b$ computed by the Grass model \cite{Grass,CastroNieto,QianLi}:
\begin{equation}
    \label{q_b}
    q_b = \xi A_g u|u|^{m_g-1}
\end{equation}
where, empirically, it has been found that $m_g \in [1,4]$, $A_g \in ]0,1[$, and $\xi = 1/(1-\rho_0)$ in which $\rho_0$ is the porosity of the sediment layer, constant in our setting, and $A_g$ symbolizes the strength of the interaction between the fluid and the sediment.  

Consequently, the SVE 1D system is given by:
\begin{equation}
    \label{Ex_sis}
    \begin{cases}
        \partial_t\eta + \partial_x (q + q_{b}) = 0,\\
        \p_t q + \partial_x(q^2/h) + g\partial_x \eta = 0, \\
        \p_t z_b + \p_x q_b = 0,
    \end{cases}
\end{equation}
where the free-surface is defined at $\eta(x,t) = h(x,t) + b(x) + z_b(x,t)$. Notice that here $b(x)$ is a fix (non-erodible) bottom. System \eqref{Ex_sis} can be rewritten in non conservative form as:
\begin{equation}
    \label{non_cons_term_1}
    \frac{\partial W}{\partial t} + A(W)\frac{\partial W}{\partial x} = 0,
\end{equation}
where $W=[h,q,z_b]^T$ and
\[
  %  W = \begin{bmatrix}
   %     h \\ q \\ z_b
   % \end{bmatrix}, \qquad 
   A(W) = \begin{bmatrix} 0 & 1 & 0 \\ gh - \frac{u^2}{h} & 2u & gh \\ \alpha & \beta & 0 \end{bmatrix},\qquad\mbox{with}\quad \alpha = \frac{\partial q_b}{\partial h},\ \beta = \frac{\partial q_b}{\partial q}
\]
We have $\beta = {m_g \xi A_g |u|^{m_g-1}}/{h}$ and $\alpha = -u\beta$. Then, the eigenvalues of the systems are given by the roots of the characteristic polynomial
\[
    p_{\lambda}(\lambda) = -\lambda((u-\lambda)^2 -gh) + gh\beta(\lambda -u).
\]
Some approximations of the eigenvalues have been proposed different authors (see e.g. \cite{deVries1965,Armanini2018}). In the case of Grass model and Froude number ($Fr$) sufficiently different from 1, the approximation proposed by de Vries is
$$
\lambda_1 = u - \sqrt{g h }, 
\quad 
\lambda_2=\beta \frac{u}{1-Fr^2},
\quad 
\lambda_2 = u+ \sqrt{g h},
$$
where $\lambda_3 < \lambda_2$ for $Fr <1$ and $\lambda_3 > \lambda_2$ for $Fr >1$.  Remark that when $A_g =0$ the three eigenvalues become $\lambda = u-\sqrt{gh},\,0,\,u+\sqrt{gh}$. Detailed analysis of Exner model with Grass equation  have been conducted in Macca \emph{et al.} \cite{Macca2024,MaccaRussoBumi}, {\color{blue} where authors proposed the eigenvalues approximations 
$$
\lambda_1 = u - \sqrt{g h }-\beta \frac{\sqrt{g h}}{2(1-Fr)}, 
\quad 
\lambda_2 = \beta \frac{u}{1-Fr^2},
\quad
\lambda_3 = u+ \sqrt{g h}+\beta \frac{\sqrt{g h}}{2(1-Fr)},
$$
with $Fr=|u|/\sqrt{gh}$. A comparison of the exact eigenvalues for the case $h=1$, $\xi A_g=0.001$ and $Fr\in [0,2]$ is presented in Figure~\ref{fig:CompEigenvalues} with de Vries \cite{deVries1965} and Macca et al. \cite{MaccaRussoBumi,Macca2024} approximations. We observe that for $Fr\gtrsim 1.3$ the approximation proposed by Macca et al. is the most accurate for the intermediate eigenvalue. For $Fr\lesssim0.7$ all approximations produce similar results.}

\begin{figure}[!ht]
	\centering	
	\includegraphics[width=0.8\textwidth]{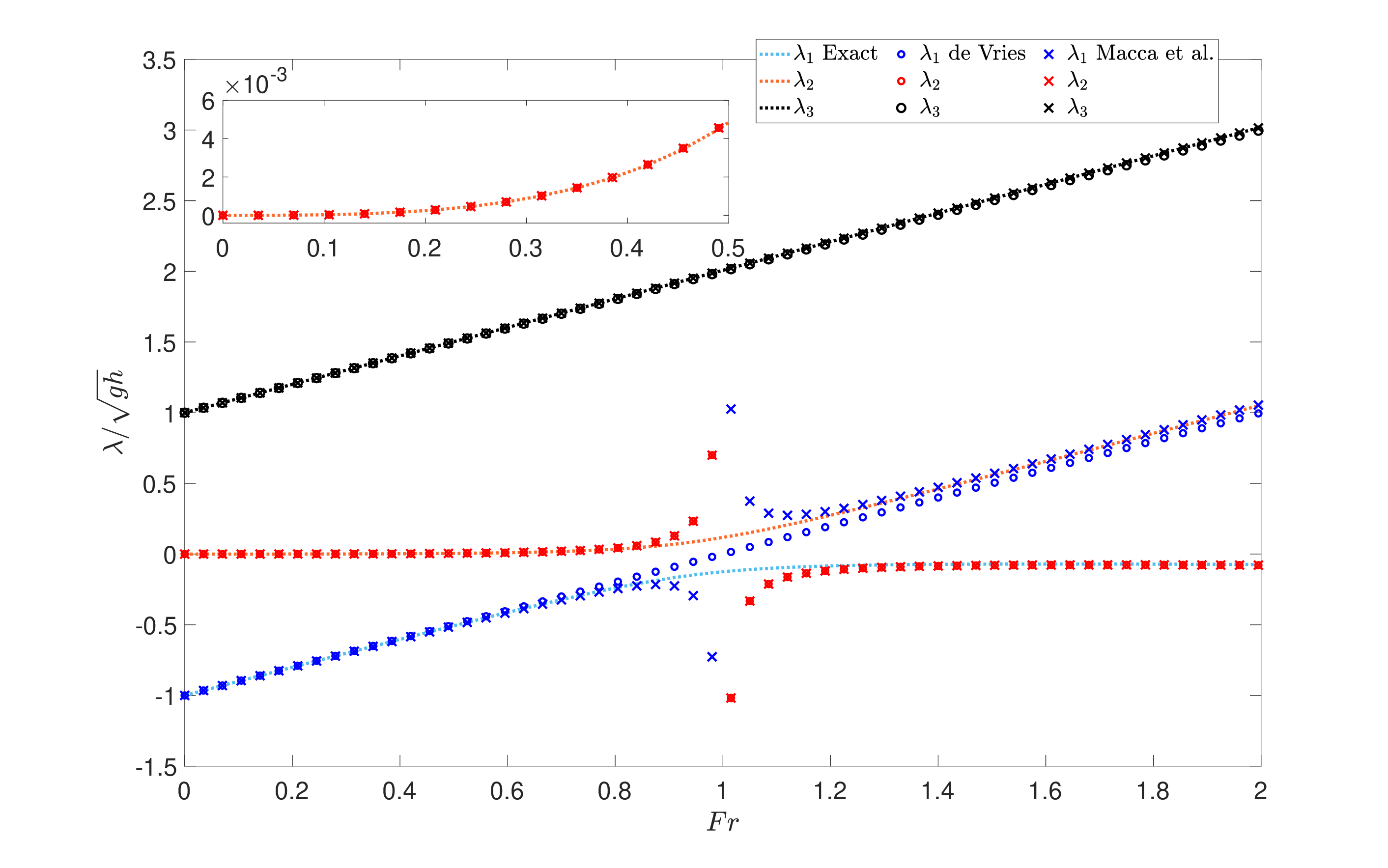}
	\caption{Comparison of the exact eigenvalues and the approximations proposed by de Vries \cite{deVries1965} and Macca et al. \cite{MaccaRussoBumi,Macca2024}. {The Inset figure is a zoom of the intermediate eigenvalue for small Froude numbers.}}
	\label{fig:CompEigenvalues}
\end{figure}

\subsection{Numerical scheme for Exner model} \label{Num_scheme_Ex}
In this section, we focus on incorporating the spatial and temporal discretizations introduced in Section \ref{sec:num_scheme} for the shallow water equations into the SVE model with the Grass equation. For the sake of clarity here we explain these discretizations separately.

\paragraph{Space discretization} Retaining the same framework as proposed in Section \ref{sec:num_scheme}, the discrete system reads 

%\giovanni{The superscripts in $\G$ should be switched. I did it in the first and third equations}

\begin{align}
	\label{eta_1_zb} & \eta_i^{n+1} = \eta_i^n  
    - \frac{\dt}{\dx}\Bigl(q_{i+\frac{1}{2}}^{n+1} - q_{i-\frac{1}{2}}^{n+1}\Bigr) 
    - \frac{\dt}{\dx}\Bigl( \G^{\eta,n}_{i+\ha} - \G^{\eta,n}_{i-\ha} \Bigr), \\[2mm]
	\label{q_zb} & q_{i+\frac{1}{2}}^{n+1} = q_{i+\frac{1}{2}}^{n} - \frac{\Delta t}{\Delta x}\Bigl(\mathcal{F}_{i+1}^n-\mathcal{F}_{i}^n\Bigr) - \frac{\Delta t}{\Delta x}g\int_{I_{i+\frac{1}{2}}} h^n\p_x\eta^{n+1}\, dx ,\\
	\label{zb_zb} & z_{b,i}^{n+1} = z_{b,i}^n 
    -  \frac{\dt}{\dx}\Bigl( \G^{z_b,n}_{i+\ha} - \G^{z_b,n}_{i-\ha} \Bigr),
\end{align}
where the only difference are the terms $ \G^{\eta,n}_{i+\ha}, \G^{z_b,n}_{i+\ha}$ in equations \eqref{eta_1_zb} and \eqref{zb_zb}, correspending to discretizations of $\p_x q_b$ in the mass and the sediment equations, respectively. They are defined by the usual Rusanov flux:
\begin{equation}
	\label{rusanov_eta_Ex}
	\begin{array}{l}
	\G^{\eta,n}_{i+\ha} = \dfrac12\left(q_{b,i+\ha}^+ + q_{b,i+\ha}^- - \alpha_{i+\ha}\left(\eta_{i+\ha}^+-\eta_{i+\ha}^-\right)\right),\\[2mm]
	 \G^{z_b,n}_{i+\ha} = \dfrac12\left(q_{b,i+\ha}^+ + q_{b,i+\ha}^- - \alpha_{i+\ha}\left(z_{b,i+\ha}^+-z_{b,i+\ha}^-\right)\right),
	\end{array}
\end{equation}
where, similarly to what is done in Subsection \ref{Aux_h}, to compute $\F_{i+\ha}$, the values $h_{i+\ha}^\pm,$ $q_{i+\ha}^\pm$ and $z_{b,i+\ha}^\pm$ adopted in \eqref{rusanov_eta_Ex} are the CWENO reconstruction from the values $h_i$, $q_i$ and $z_{b,i}$ respectively.

Finally, an analogous discretization of the pressure term in \eqref{q_zb} as in Section \ref{sec:num_scheme} leads to the first-order in time third order in space semi-implicit scheme for the SVE model \eqref{Ex_sis}.

\paragraph{Time discretization} Similarly to the time reconstruction described in Section \ref{sec:time_recons}, we appropriate define the quantities $W= [\eta,q,z_b]^T$ and $K(W_E,W_I)$ to adapt the iterative process used for the shallow water equations to the Exner model. In particular,
similarly to equation \eqref{ode_form_semi}, the ODE system becomes
\[
    W' = K(W),\qquad \mbox{with}\quad K(W)= \begin{bmatrix}
    -\p_x (q_{b} +q) \\
        -\p_x(qu) - gh\p_x\eta  \\
        -\p_xq_{b}
\end{bmatrix}
\]
%where again $E$ and $I$ represent the explicit and implicit treatment of the corresponding terms. 
where again, 
the semi-implicit approach described above can be adopted to reach high order in time, with
\[
    K(W_E,W_I) = 
    \begin{bmatrix}
              -\p_x (q_{bE} +q_I) \\
              -\p_x(qu)_E - gh_E\p_x\eta_I  \\
              -\p_xq_{bE}
          \end{bmatrix}
\]
which, after space discretization, becomes
\[
    \tilde{K}(W_E,W_I) = 
    \begin{bmatrix}
        -D_x^{4,\eta}(q_{b_E}) & - D_x^1(q_I) \\
        -D_x^3(q_E\, u_E) & - gh_ED_x^2(\eta_I)  \\
        -D_x^{4,z_b}(q_{b_E}) & 0
    \end{bmatrix},
\]

%by doubling the unknowns, the semi-implicit scheme is written as 
%$$ [W_E,W_i]' = \tilde{K}(W_E,W_I) = \begin{bmatrix}
%    -D_x^{4,\eta}(q_{b_E}) & - D_x^1(q_I) \\
%        -D_x^3(q_E\, u_E) & - gh_ED_x^2(\eta_I)  \\
%        -D_x^{4,z_b}(q_{b_E}) & 0
%\end{bmatrix},$$
where the differential operators $D_x^j, j=1,2,3$ are defined as in Section \ref{sec:time_recons}, page \pageref{page:D}, and 
\[
    D_x^{4,w }(q_{b_E}) = \dfrac{\G_{i+\ha}^w-\G_{i-\ha}^w}{\dx},  \quad
    w  \in \{\eta,z_b\},
\]
with $\G_{i+\ha}^{\eta,n}$ and $\G_{i+\ha}^{z_b,n}$ are given by \eqref{rusanov_eta_Ex}, in terms of $W_E$. 

Thus, the semi-discrete in time first order semi-implicit scheme can be written as:
\begin{equation}
\label{First_ord_Exner}
\begin{cases}
	\eta^{n+1} = \eta^n - \dt D_x^1(q^{n+1}) - \dt D_x^{4,\eta}(q_b^n),\\
	q^{n+1} = q^n - \dt D_x^3(q^nu^n) - \dt gh^nD_x^2(\eta^{n+1}),\\
	z_b^{n+1} = z_b^n- \dt D_x^{4,z_b}(q_b^n).
\end{cases}
\end{equation}

%\paragraph{Time discretization} Similarly to the time reconstruction described in Section \ref{sec:time_recons}, we appropriate define the quantities $W= [\eta,q,z_b]^T$ and $K(W_E,W_I)$ to adapt the iterative process used for the shallow water Equations to the Exner model. In particular,
%similarly to equation \eqref{ode_form_semi}, the ODE system becomes
%$$ W' = K(W_E,W_I) = \begin{bmatrix}
%	-\p_x (q_{b_E} +q_I) \\
%	-\p_x(qu)_E - gh_E\p_x\eta_I  \\
%	-\p_xq_{b_E}
%\end{bmatrix}$$
%where $E$ and $I$ represent the explicit and implicit treatment. 
%
%The semi-implicit scheme can be written as 
%$$ W' = \tilde{K}(W_E,W_I) = \begin{bmatrix}
%	-\hat{D}_x(q_{b_E}) & - D_x(q_I)_x \\
%	-\hat{D}_x((qu)_E)& - gh_ED_x(\eta_I)  \\
%	-\hat{D}_x(q_{b_E}) & 
%\end{bmatrix}.$$
%The differential operator $D_x$ and $\hat{D}_x$ are defined as in Section \ref{sec:time_recons}.
%Thus, the semi-discrete in time first order semi-implicit scheme can be written as:
%\begin{equation}
%	\label{First_ord_Exner}
%	\begin{cases}
%		\eta^{n+1} = \eta^n - \dt \hat{D}_x(q_b^n) - \dt D_x(q^{n+1}), \\ 
%		q^{n+1} = q^n - \dt \hat{D}_x(q^nu^n) - \dt g h^nD_x(\eta^{n+1}),\\
%		z_b^{n+1} = z_b^n - \dt \hat{D}_x(q_b^n).
%	\end{cases}
%\end{equation}

To achieve a higher-order solution at time $t^{n+1}$ based on the solution at time $t^n$, an IMEX $s$-stage Runge-Kutta method is employed, as detailed in Section \ref{sec:time_recons}.

%\begin{remark}
%As reported for the shallow water equation Section \ref{Aux_h}, in the Exner model $h_{\i+\ha}$ and $z_{b,i+\ha}$ have to be computed due to the disparity in the point-wise definition of $q$, $h$ and $z_b$.
%\end{remark}

\subsection{Time dependent semi-analytical solution with linear friction}\label{subsec:semi-analitical-sol-SVE}
Deriving analytical solutions for the Exner system is useful to validate the numerical method. Very often, it is not easy to obtain exact solutions, although it is possible to get semi-analytical solutions, obtained by solving a small system of ODEs, adopting a standard method for ODEs. To this aim, it is common to consider a stationary regime for the fluid, and solve the sediment equation under some hypothesis. We remark that this is an approximate solution, since in reality the fluid is not strictly stationary if the sediment changes in time. However, if the sediment is much slower than the fluid, we can assume the fluid motion is almost stationary.

In that case, the semi-analytical solution is obtained by solving (maybe numerically) a scalar PDE. This is the case of some semi-analytical solutions that have been already presented in the literature (see \cite{Hudson2005}, \cite{Macca2024}). With the purpose of completeness, they are presented in Appendix \ref{App_A}, and will be considered in the numerical tests section \ref{sec:numerics}.  

Here, we introduce a novel semi-analytical solution for the Exner system, which depends on time, with the Grass formula with $m_g=3$ and a linear friction term. The main difference with respect to the literature is that, in this case, we do not consider a stationary regime for the fluid but we assume some profiles for the height, velocity and sediment layer at initial time. Then, the hydrodynamic variables $(h,u)$ evolve together with the bottom ($z_b$). It is a time-dependent solution. This solution is obtained by solving (numerically) an ODE system. 

Let us start from the SVE system \eqref{Ex_sis} with $m_g=3$, to which we add a friction term:
\begin{equation} \label{SVEcflinear}
	\left\{ 
	\begin{array}{l}
    \displaystyle
		\partial_t{h}+ \partial_x (h \, u )=0, \\[4mm]
    \displaystyle
		\partial_t( h u) + \partial_x ( h u^2 + g h^2/2) = - g h \partial_x z_b, \\[4mm]
    \displaystyle
		\partial_t(z_b)+\partial_x (\xi A_g u^3)=0.
	\end{array}
	\right.
\end{equation}

We set the function
\[
    h_{\textrm{ini}}(x) = h_0 \in \mathbb{P}_0[x],\qquad u_{\textrm{ini}}(x)= a x+ b\in \mathbb{P}_1[x], \qquad 
    z_{b,\textrm{ini}}(x)= \frac{3 \xi A_g a}{c}(ax+b)^2\in \mathbb{P}_2[x].
\]
for $h_0,a,b,c\in\mathbb{R}$ and $h_0>0$. Let $\tilde{W}(t)=(\, \tilde{h}(t), \, \tilde{u}(t),\, \tilde{z}_b(t))$ be a solution of the ODE system 
$$
\left\{
\begin{array}{l}
	\displaystyle \frac{{\rm d}}{{\rm d}t} \tilde{h}+a \, \tilde{u} \, \tilde{h}=0, \\[2mm]
	\displaystyle \frac{{\rm d}}{{\rm d}t} \tilde{u}+a \, \tilde{u}^2+ \frac{6 g \xi A_g a^2}{c} \tilde{z}_b= 0,\\[2mm]
	\displaystyle \frac{{\rm d}}{{\rm d}t} \tilde{z}_b + c \, \tilde{u}^3=0, \\
%	\tilde{h}(0)=h_0, \qquad 	\tilde{u}(0)=1,  \qquad 	\tilde{z}_b(0)=1.
\end{array} 
\right.
$$
with $(\tilde{h}(0),\tilde{u}(0),\tilde{z}_b(0)) = (h_0,1,1)$. Then, 
$$
W(x,t)= \bigg(  \tilde{h}(t), \, \tilde{h}(t) \,  u_{\textrm{ini}}(x) \, \tilde{u}(t), \, z_{b,\textrm{ini}}(x) \, \tilde{z}_b(t) \bigg)
$$
is a smooth solution of system \eqref{SVEcflinear} with initial condition
$
W(x,0)=(h_0, h_0 u_{\textrm{ini}}(x),z_{b,\textrm{ini}}(x) )
$
and suitable boundary conditions. This semi-analytical solution will be considered in the numerical tests (see Section \ref{sec:numerics}).

\subsection{Absorbing boundary condition (ABC)} \label{ssec:ABC}
%In computational simulations of hyperbolic partial differential equations (PDEs), 

When simulating long time behaviour of the solution, with waves that leave the computational domain, 
accurately imposing boundary conditions is essential to prevent unwanted reflective waves from interfering with the numerical solution. These nonphysical reflections can significantly compromise the accuracy of simulations, particularly in applications like sediment transport in the Exner model and gas dynamics in the Euler equations.

Reflective waves occur at the computational domain boundaries where the discretization does not perfectly align with the continuous boundary conditions. This misalignment causes waves to bounce back into the domain instead of exiting smoothly, leading to inaccuracies. To minimize these reflections and ensure that outgoing waves are effectively absorbed, efficient absorbing boundary conditions (ABCs) are necessary. Techniques such as relaxation methods and far-field operators are commonly used to achieve this. These approaches typically involve a sponge layer, a region within the computational domain where the numerical solution is damped to mitigate reflections. The main challenge is to implement these methods in a way that ensures computational efficiency while accurately absorbing outgoing waves, thereby maintaining the simulation's overall integrity \cite{Mayer1998,CarlosMunoz}.

In this section, we address the imposition of outflow boundary conditions for the one-dimensional Exner model. %, as spurious reflective waves can interact with the numerical solution, thereby actively affecting the accuracy of sediment evolution. Although a highly accurate numerical solution can be obtained within the interior of the domain, boundary discretization can lead to nonphysical reflections.
In order to address this issue, a class of relaxation methods based on far-field operators \cite{Karni,CarlosMunoz} offers a viable alternative to the technique developed in \cite{MaccaRussoBumi}, which involved adding an extra domain to dampen waves and consequently reduce reflections. The new technique, which also requires an additional domain, achieves excellent results without significantly increasing the size of the computational domain.

As discussed in \cite{CHOI2020}, the performance of this ABC is influenced by the selection of the weight function $\Gamma(x)$, which decreases from 1 to 0 within the sponge layer. The relaxation step involves updating the numerical solution $U^{n+1}$ by forming a convex combination of the computed solution $\widetilde{U}^{n+1}$ and a target function, typically the constant far-field state $U^{\rm eq}$. Karni \cite{Karni} emphasized that the weight function must meet the smooth transition condition to ensure the continuity of modal coefficients: $$ \Gamma(x_b) = 1, \quad \Gamma'(x_b) = \Gamma''(x_b) = \Gamma'''(x_b) = \ldots = 0, $$
being $[x_a,x_b]$ is the computational domain and $]x_b,\infty[$ is the absorbing domain. Following this criterion, the weight function proposed by Engsig-Karup et al. \cite{Engsig-Karup2006} is
$$
\Gamma(x) = -2 (1 - \varphi(x))^3 + 3 (1 - \varphi(x))^2,
$$
where 
$$
\varphi(x) = \left( \frac{x - x_b}{x_{c} - x_b} \right) \chi_{]x_b, x_{c}]}(x) + \chi_{]x_{c}, +\infty[}(x).
$$
Nevertheless, from a numerical point of view, it is advised that the length of the sponge layer $ |x_{c} - x_b| $ is greater than $2L$, where $L$ is the longest wavelength in the solution.

From a practical point of view, once the numerical solution $\tilde{U}^{n+1}$ has been obtained on the domain $[x_a,x_c]$, the numerical solution $U^{n+1}_i$ is defined as follows:
\begin{equation}
    U^{n+1}_i = \begin{cases}
        \tilde{U}^{n+1}_i \quad {\rm if} \; x_i\in[x_a,x_b]; \\ 
        \tilde{U}^{n+1}_i\Gamma(x_i) + U^{eq}(1-\Gamma(x_i)) \quad {\rm otherwise}.
    \end{cases}
\end{equation}

\section{Numerical Experiments} \label{sec:numerics}
In this paper, numerical tests related to the one-dimensional shallow water equations and Exner model with Grass equation are considered. 
Our methodology of testing relies on several classical and demanding test cases: 
\begin{enumerate}
    \item \textit{Order of accuracy}. These tests measure the scheme's capability to achieve high-order accuracy for smooth solutions in both the shallow water and Exner models.
    \item \textit{Well-balanced property.} This assessment focuses on the scheme's ability to preserve stationary solutions (e.g., lake at rest) up to machine precision.
    \item \textit{Riemann problems}. Various versions of 1D Riemann problems are simulated to check the shock capturing capabilities of the scheme. Specifically, we consider Riemann problems that satisfy the Rankine-Hugoniot condition. While the schemes are generally non-conservative, the conservation error remains acceptable\footnote{The non-conservative treatment is employed to avoid iterative processes during implicit stages.}.

    Additionally, we examine the schemes' performance with discontinuous problems to assess their capability to handle non-smooth cases. In such scenarios, the schemes could introduce spurious oscillations. This aspect is not the primary focus of this paper.
    \item \textit{Sediment evolution}. The non-conservative semi-implicit treatments aim to accurately capture the evolution of sediment, which typically occurs at a much slower rate compared to water and wave velocities.
    \item \textit{Semi-Analytical ODE-type}. This test focuses on a comparison of the non-conservative third-order semi-implicit scheme with a time-dependent semi-analytical solution, for a large domain and final time.
    \item \textit{Exner waves}. This test has been designed to highlight the effectiveness of the scheme in modelling sediment evolution when a wave train has been imposed on the surface. By varying the wave frequency and adopting different spatial and temporal discretizations, it is possible to achieve excellent results, even when the surface waves are not finely discretized in either space or time.
 
\end{enumerate}

In an explicit scheme, the time step $ \Delta t^n $ is determined by the CFL condition, expressed as  
\begin{equation}  
    \label{timestep_exp}  
    \Delta t^n = \mathrm{CFL_{\rm EXP}}\frac{\Delta x}{\lambda_{\max}^n},  
\end{equation}  
where $ \lambda_{\max}^n $ represents the maximum spectral radius of the Jacobian matrices $ \partial F / \partial U $ across all computational cells. This condition arises from the stability constraints of explicit schemes, which require that information does not travel more than one grid cell per time step. Consequently, explicit methods are often subject to stringent time step restrictions.  

In contrast, for a semi-implicit scheme, the time step $ \Delta t^n $ must satisfy the material CFL (MCFL) condition (see \eqref{eq:MCFL}), which is generally less restrictive than the standard explicit CFL condition\footnote{Empirical studies suggest that the material CFL condition should not exceed 0.5 \cite{Bosca2024}.}. This leads to the definition  
\begin{equation}  
    \label{timestep_mcfl}  
    \Delta t^n = \mathrm{MCFL}\frac{\Delta x}{u_{\max}^n},  
\end{equation}  
where $ u_{\max}^n $ denotes the maximum material speed.  

However, a more general formulation of the time step is required to fully leverage the advantages of the semi-implicit approach. Unlike fully explicit methods, which impose strict stability constraints due to the presence of fast waves, a semi-implicit scheme allows for a more relaxed condition. In this scheme, certain terms — typically the stiffest components of the system — are treated implicitly, thereby mitigating the stability limitations imposed by explicit time stepping. As a result, the time step is not necessarily restricted by the largest eigenvalue of the Jacobian but can instead be increased beyond the explicit CFL limit. 

Ideally the time step for the semi-implicit scheme should be dictated by accuracy requirements as well as by the CFL stability imposed by the explicit term. A detailed analysis of the most suitable time step controller is beyond the scope of the present paper, and will be subject of future investigation.

In the present paper we adopt an alternative CFL condition for the semi-implicit scheme, where the time step is given by  
\begin{equation}  
    \label{timestep_semi_implicit}  
    \Delta t^n = \min\Bigl(\mathrm{MCFL}\frac{\Delta x}{u_{\max}^n},\mathrm{CFL_{\rm IMEX}}\frac{\Delta x}{\lambda_{\max}^n}\Bigr).  
\end{equation}  
Here, the parameter $ \mathrm{CFL_{\rm IMEX}} $ is typically greater than 1, and is selected empirically based on accuraccy, rather than stability requirements. For this reason, the final time step is thus chosen as the minimum between the material CFL condition and the relaxed stability constraint of the IMEX scheme, ensuring a balance between computational efficiency and numerical accuracy.
Theoretically, when $u \approx 0$, the  time step is not limited by stability,  and in this case it depends only on the CFL$_{\rm IMEX}$\footnote{Since the CFL$_{\rm IMEX}$ could be greater than 1 without upper limitation, we impose CFL$_{\rm IMEX} \le \tilde{C}$, with $\tilde{C}$ a parameter that changes from test to test.}.

In all numerical experiments, the gravitational constant is set to $ g = 9.81\ \text{m/s}^2 $, while the numerical relative errors are evaluated using the $ L^1 $ norm, computed as  
\begin{equation}
    \textrm{Error} = \frac{||w_{\rm num} - w_{\rm ref}||_1}{||\bar{w} - w_{\rm ref}||_1},  
\end{equation}
where $ w_{\rm ref} $ denotes the reference solution obtained with the third-order semi-implicit scheme on a finer mashes; while $\bar{w}$ is the exact equilibrium.

\subsection{Order of accuracy} \label{ssec:accuracy}
This section is dedicated to verifying that the semi‐implicit schemes applied to the shallow water equations and the Exner model, as introduced earlier, accurately reproduce solutions that conform to the theoretical order of the scheme when the solutions are smooth. In addition to considering both the shallow water case—with and without bottom topography—and the Exner model, we also perform a rigorous numerical test to validate the accuracy of the fully third-order scheme and the simplified one. These numerical experiments are designed to confirm that the scheme achieves the expected third-order convergence, demonstrating the robustness of implicit discretization, as well as the effectiveness of explicit correction in improving spatial accuracy.

Given the small difference between the two methods, the rest of the tests are performed just with the simplified scheme.

\subsubsection{Shallow water with flat bottom} \label{sssec:acc:SW_flat} 
The common settings of this experiment are: $[x_a,x_b]=[-4,6]$ the interval; the CFL at initial condition is $15.4$, although it is adapted step by step such that MCFL$\leq 0.4$%, \footnote{\label{noteCFL_MCFL} The CFL condition is adapted step by step such that MCFL$=0.4$.}; 
, flat bottom topography $b(x) \equiv 0$ and  
\begin{comment}
\begin{equation}
    \label{Ord_acc_SW_1_eta_0} 
    U_0(x) =     \begin{bmatrix}
    	\eta_0(x) \\ q_0(x) 
     \end{bmatrix}   
     =    	
    \begin{bmatrix}
        \displaystyle 0.7 + 0.2 e^{-5(x-1)^2} \\ 
        0        
    \end{bmatrix}    
\end{equation}
\end{comment}

\begin{equation}
    \label{Ord_acc_SW_1_eta_0_1} 
    U_0(x) =     \begin{bmatrix}
    	\eta_0(x) \\ q_0(x) 
     \end{bmatrix}   
     =    	
    \begin{bmatrix}
        \displaystyle 0.7 + 0.2 e^{-3(x-1)^2} \\ 
        0        
    \end{bmatrix}    
\end{equation}

\begin{table}[!ht]
\numerikNine
\centering
\begin{tabular}{|c||cc|cc|cc|cc|}
\hline \multicolumn{9}{|c|}{\textbf{Shallow water - Rate of convergence}} \\
\hline  
\hline   
 & \multicolumn{4}{|c|}{\textbf{Simplified Scheme  \eqref{eta_1}-\eqref{q_1}}} & \multicolumn{4}{|c|}{\textbf{Fully Third Order Scheme \eqref{real_third_ord_eta}-\eqref{real_third_ord_q}}}  \\
 \hline  
 \hline   
& \multicolumn{2}{c|}{\textbf{$\eta$}} & \multicolumn{2}{c|}{\textbf{$q$}} & \multicolumn{2}{c|}{\textbf{$\eta$}} & \multicolumn{2}{c|}{\textbf{$q$}} \\ 
$N$ &  $L^1$ error &  order  & $L^1$ error &  order & $L^1$ error &  order & $L^1$ error &  order \\ \hline
& & & & \\[-3mm] 
100 & 5.54$\times$10$^{-1}$  &  --  & 5.37$\times$10$^{-1}$ & --   & 5.47$\times$10$^{-1}$  &  --  & 5.25$\times$10$^{-1}$ & --   \\ 
200 & 7.95$\times$10$^{-2}$  & 2.80 & 7.74$\times$10$^{-2}$ & 2.79 & 7.65$\times$10$^{-2}$  & 2.84 & 7.51$\times$10$^{-2}$ & 2.81 \\
400 & 1.42$\times$10$^{-2}$  & 2.49 & 1.41$\times$10$^{-2}$ & 2.46 & 1.39$\times$10$^{-2}$  & 2.46 & 1.41$\times$10$^{-2}$ & 2.42 \\
800 & 1.75$\times$10$^{-3}$  & 3.02 & 1.71$\times$10$^{-3}$ & 3.04 & 1.52$\times$10$^{-3}$  & 3.20 & 1.54$\times$10$^{-3}$ & 3.19 \\
1600& 2.86$\times$10$^{-4}$  & 2.61 & 2.86$\times$10$^{-4}$ & 2.58 & 2.01$\times$10$^{-4}$  & 2.92 & 2.09$\times$10$^{-4}$ & 2.87 \\
\hline
\end{tabular}
\caption{Test  \ref{sssec:acc:SW_flat}: Smooth perturbation of the lake at rest for shallow water on flat bottom with initial condition \eqref{Ord_acc_SW_1_eta_0_1}.  Convergence rates and $L^1$-norm errors for free-surface $\eta$ and $q$ between the numerical solution and the reference solution at $t_{\text{final}} = 1\,s$ on uniform mesh and CFL$ = 15.4$ such that MCFL$\leq0.4$. The reference solution has been computed on a $6400$ uniform mesh.}
\label{tab:Ord_acc_SW_flat_2}
\end{table}

% -------------
%
% --- Fig ---
\begin{figure}[!ht]
         \centering
         \includegraphics[width=0.49\textwidth]{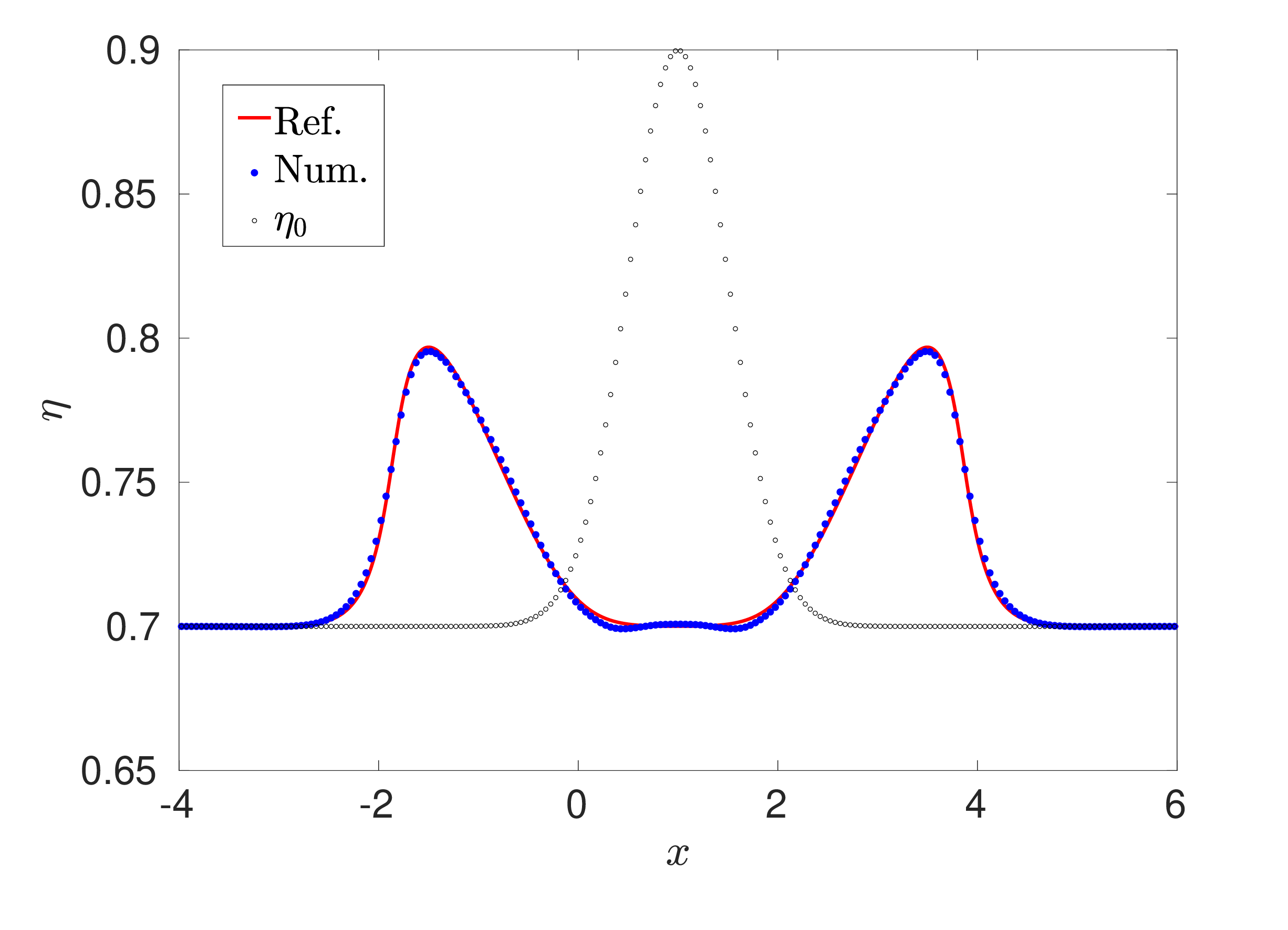}
         \includegraphics[width=0.49\textwidth]{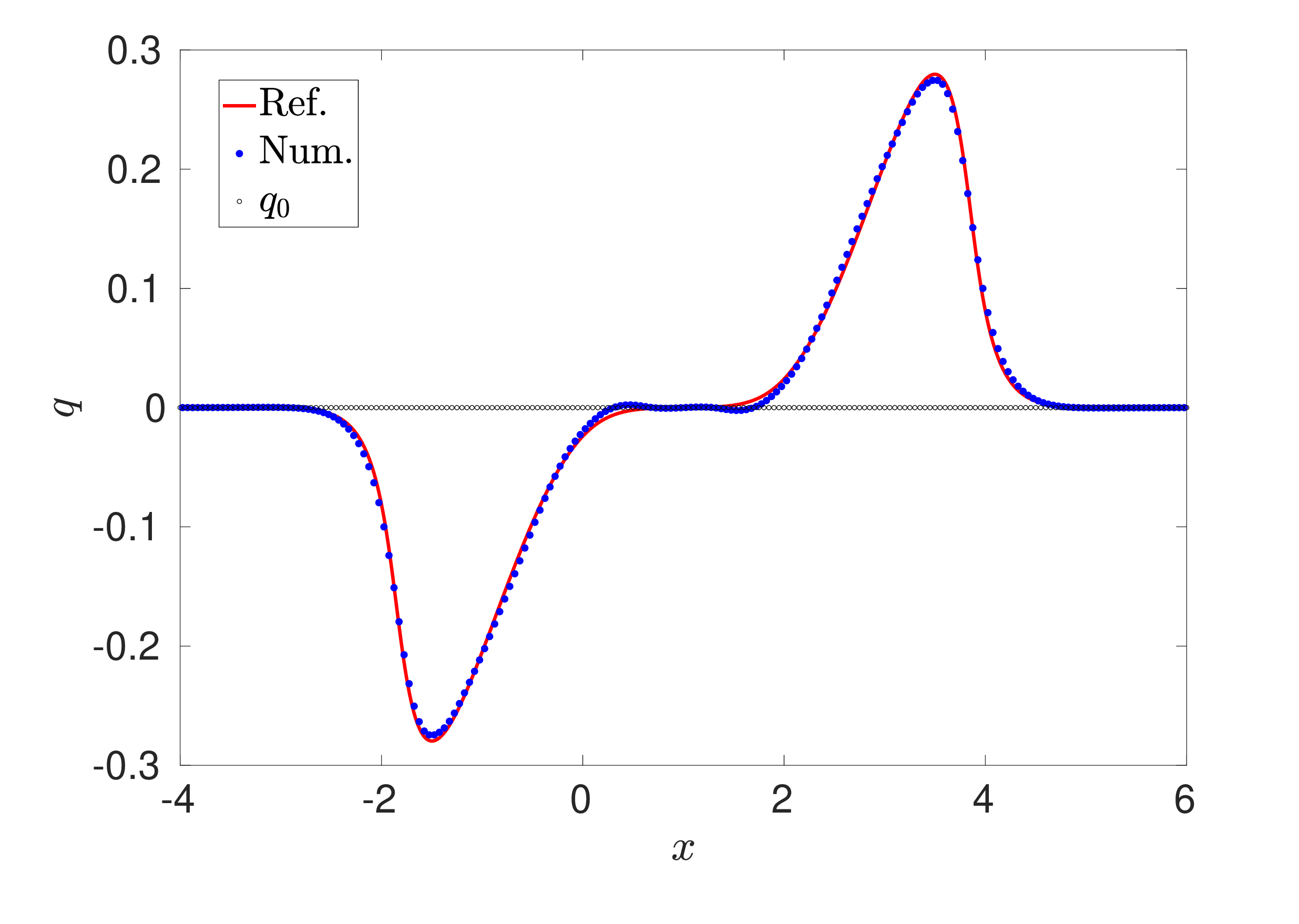}
     \caption{Test \ref{sssec:acc:SW_flat}: Order accuracy for shallow water with flat bottom topography. Initial conditions $(\eta_0,q_0)$ and numerical approximation obtained with a third-order semi-implicit staggered scheme adopting a $200$ uniform mesh and MCFL$=0.4$ at time $t=1\ s$. Free surface (top) and discharge (bottom). The reference solution has been obtained on a $6400$ uniform mesh.}
     \label{Exp:Acc:SW:Test_1_flat}
\end{figure}
% ---------
\begin{figure}[!ht]
     \centering
     \begin{subfigure}[b]{0.33\textwidth}
         \centering
         \includegraphics[width=\textwidth]{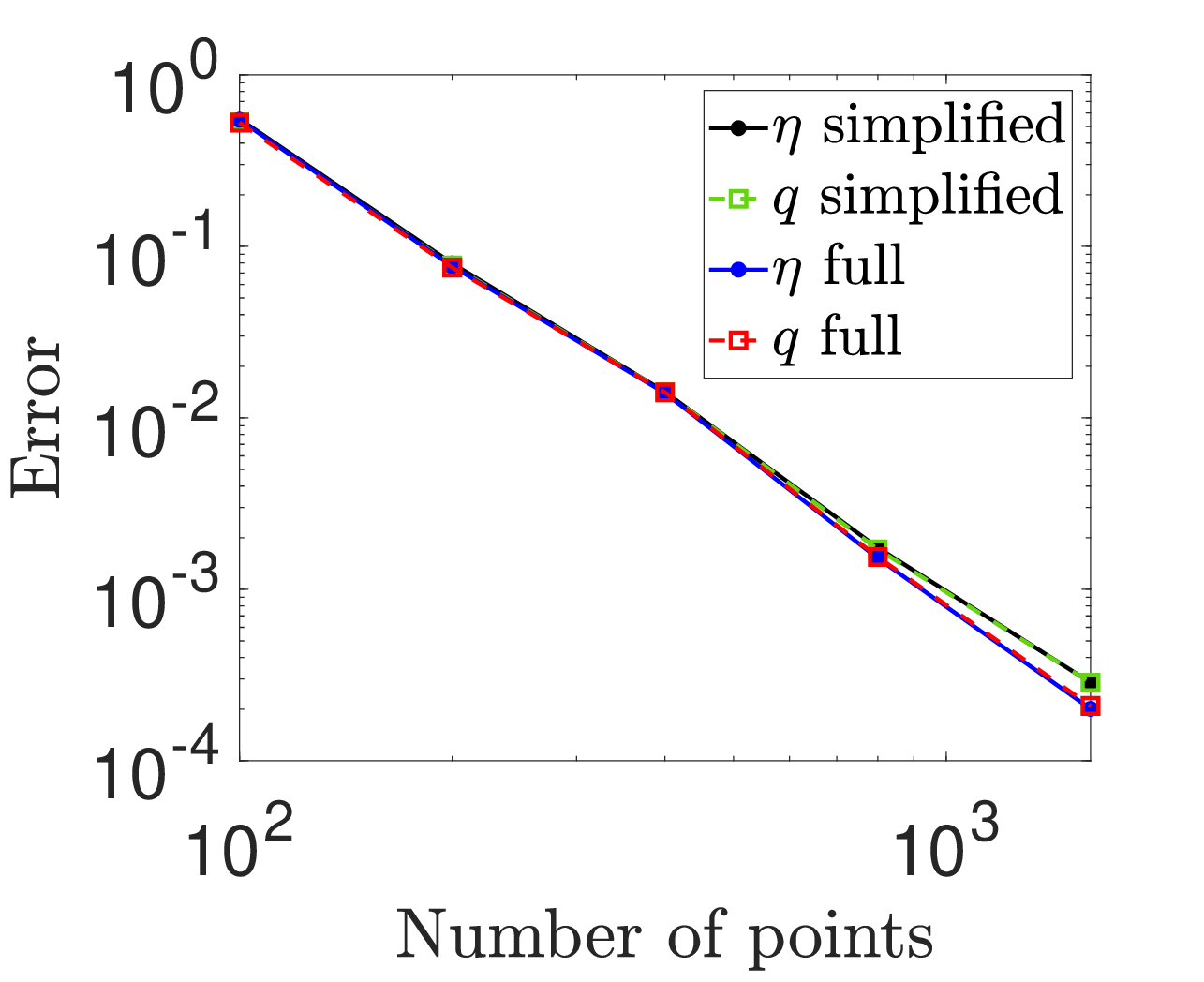}
         \caption{Points vs Errors in log scale.}
         \label{Exp:Acc:SW:P_E_flat}
     \end{subfigure}
     \hfill
     \begin{subfigure}[b]{0.33\textwidth}
         \centering
         \includegraphics[width=\textwidth]{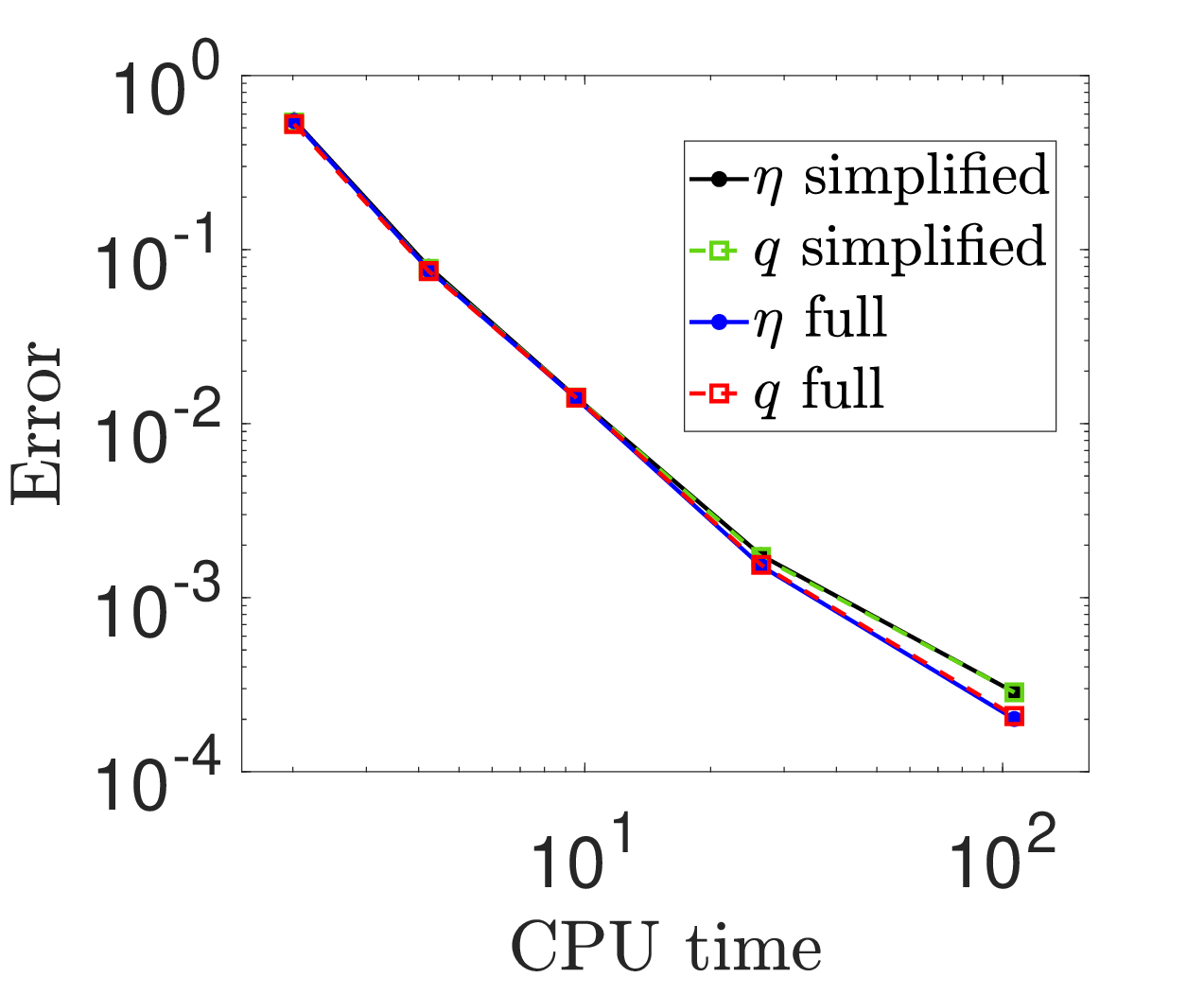}
         \caption{CPU times vs Errors in log scale.}
         \label{Exp:Acc:SW:CPU_E_flat}
     \end{subfigure}
     \hfill
     \begin{subfigure}[b]{0.29\textwidth}
         \centering
         \includegraphics[width=\textwidth]{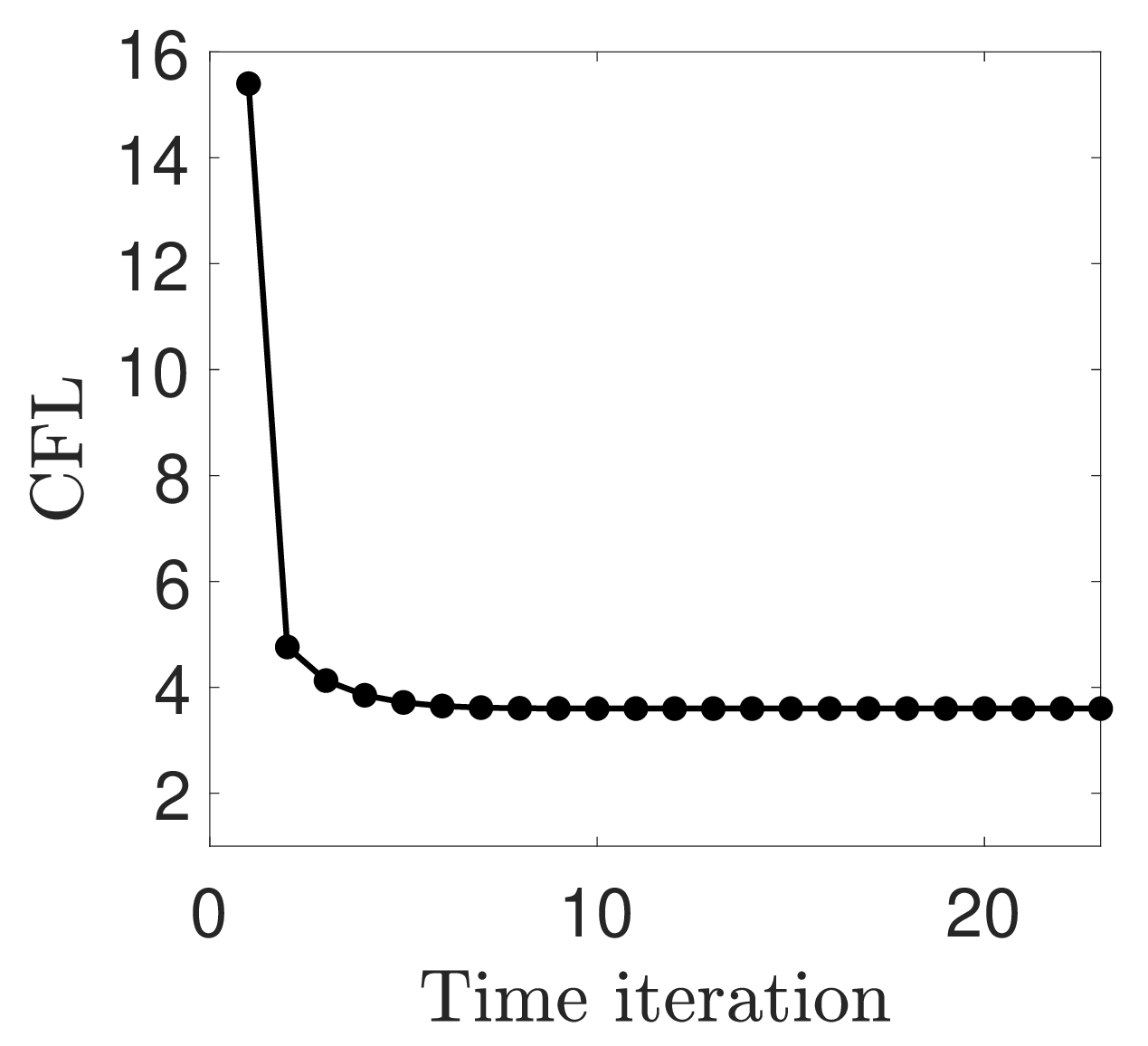}
         \caption{CFL for each time iteration.}
         \label{Exp:Acc:SW:CFL_flat}
     \end{subfigure}
     \caption{Test \ref{sssec:acc:SW_flat}: Order accuracy for shallow water with flat bottom topography. L$^1$ errors vs number of points and L$^1$ errors vs cpu time in seconds in log scale for $\eta$ and $q$, respectively \ref{Exp:Acc:SW:P_E_flat}-\ref{Exp:Acc:SW:CPU_E_flat}. CFL condition at each time step for the numerical solution computed on a $400$ uniform mesh \ref{Exp:Acc:SW:CFL_flat}. The final time is $0.8\ s$ and the reference solution has been obtained on a $6400$ uniform mesh.} 
     \label{Exp:Acc:SW:CPU_CFL_flat}
\end{figure}
% --------------
% 

Figure~\ref{Exp:Acc:SW:Test_1_flat} depicts the initial condition, the reference solution obtained on a uniform grid of 6400 points, and the numerical approximation for $\eta$ and $q$ at the final time $t = 0.8\ s$ computed on a uniform grid with 200 points and a MCFL value of $0.4$. 

In Table~\ref{tab:Ord_acc_SW_flat_2} the convergence rates and L$^1$-norm errors are presented, including those computed for the fully third-order scheme and the simplified one. It is evident from the data that the errors and convergence rates of the two schemes are quite similar, with a noticeable difference only for fine grids, which justifies neglecting the explicit term $\Delta_3q_i$ in practical applications. Moreover, in Figure~\ref{Exp:Acc:SW:CPU_CFL_flat}, the relationship between the number of points and the error is shown, together with the computational cost and the error on the log scale \ref{Exp:Acc:SW:P_E_flat}-\ref{Exp:Acc:SW:CPU_E_flat}. The progression of the CFL iteration through iterations is also shown in figure \ref{Exp:Acc:SW:CFL_flat}. We see that the CFL decreases to 3.6 approximately after a few iterations due to MCFL $\leq 0.4$.

\subsubsection{Shallow water with non-flat bottom} \label{sssec:acc:SW_no_flat} 
We consider a similar test configuration, except for the bottom and free surface. In this case, we define
\begin{equation}
	\label{Ord_acc_SW_2_eta_0} 
	b(x) = 0.01 + 0.2 e^{-10(x-1)^2}; \qquad	U_0(x) = 
	\begin{bmatrix}
		0.7 - 0.2 \, e^{-\ha(x-1)^2} \\ 0.01                
	\end{bmatrix}  ,  
	\end{equation}

% ----- Tab ----------
\begin{table}[!ht]
\numerikNine
\centering
\begin{tabular}{|c||cc|cc|}
\hline \multicolumn{5}{|c|}{\textbf{Shallow water - Rate of convergence}} \\
\hline  
\hline   
& \multicolumn{2}{c|}{\textbf{$\eta$}} & \multicolumn{2}{c|}{\textbf{$q$}} \\ 
$N$ &  $L^1$ error &  order & $L^1$ error &  order\\ \hline 
& & & & \\[-3mm]
100 & 1.42$\times$10$^{-1}$  &  --  & 3.03$\times$10$^{-0}$ & --   \\ 
200 & 1.37$\times$10$^{-2}$  & 3.37 & 3.46$\times$10$^{-1}$ & 3.13 \\
400 & 1.92$\times$10$^{-3}$  & 2.83 & 4.54$\times$10$^{-2}$ & 2.93 \\
800 & 1.77$\times$10$^{-4}$  & 3.44 & 4.14$\times$10$^{-3}$ & 3.45 \\
1600& 2.34$\times$10$^{-5}$  & 2.91 & 5.52$\times$10$^{-4}$ & 2.91 \\
\hline
\end{tabular}
\caption{Test \ref{sssec:acc:SW_no_flat}: Smooth initial condition for shallow water with non-flat bottom. Convergence rates and $L^1-$norm errors for free-surface $\eta$ and $q$ between the numerical solution and the reference solution at $t_{\text{final}} = 0.8\ s$ on uniform mesh and CFL$ = 15.4$ such that MCFL$\leq0.4.$ The reference solution has been computed on a $6400$ uniform mesh.}
\label{tab:Ord_acc_SW_no_flat}
\end{table}
% -------------
%
% -------------
%\begin{figure}[!ht]
%	\centering
%		\includegraphics[width=0.5\textwidth]{Figures/Experiments/Accuracy/SW/m_Test_SW_2_IC_eta.eps}
%\caption{Test \ref{sssec:acc:SW_no_flat}: Order accuracy for shallow water with non-flat bottom topography. Initial condition.}
%\label{Exp:Acc:SW:Test_2_no_flat_ic}
%\end{figure}
			
\begin{figure}[!ht]
         \centering
         \includegraphics[width=0.49\textwidth]{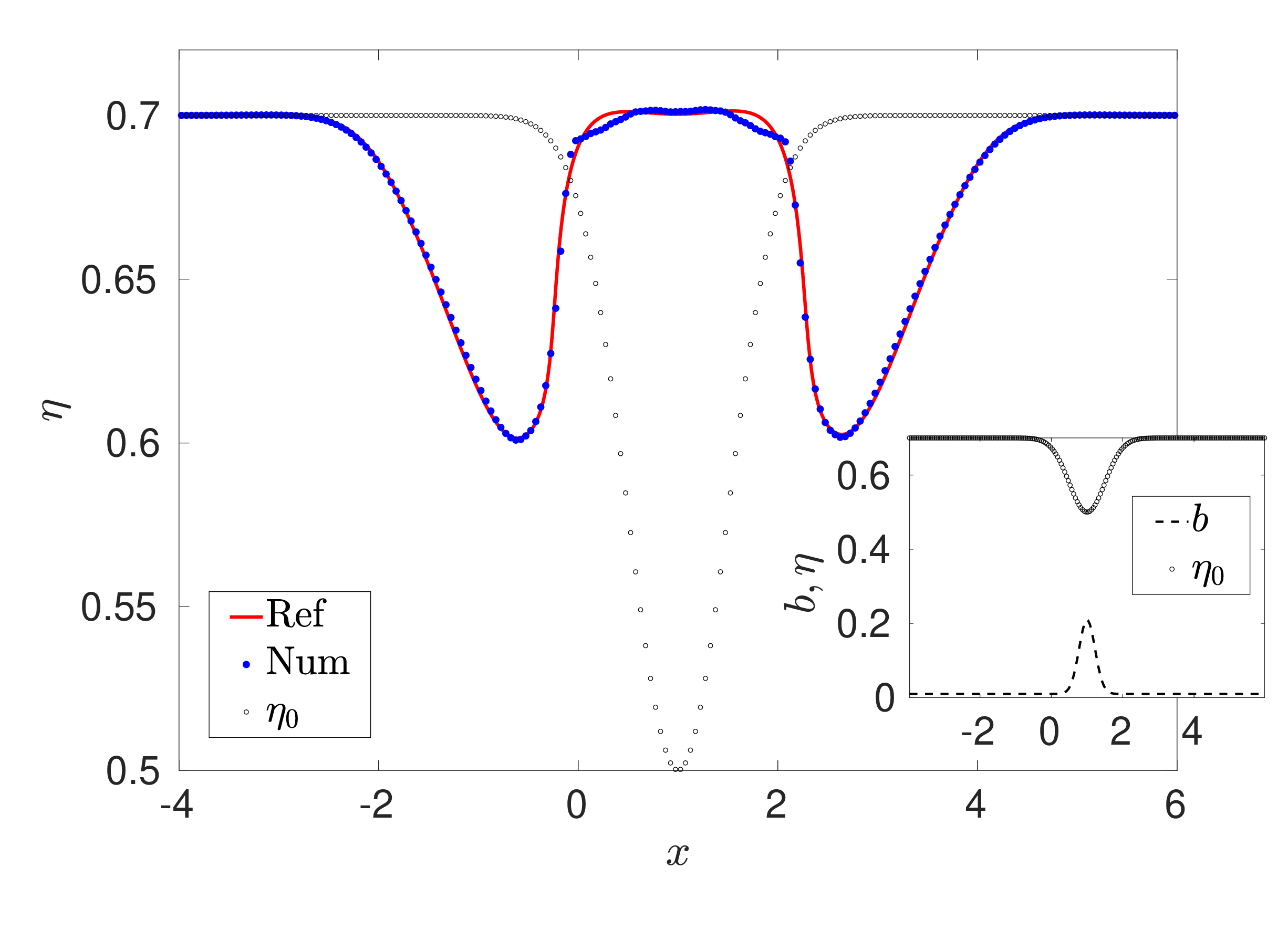}
        \includegraphics[width=0.49\textwidth]{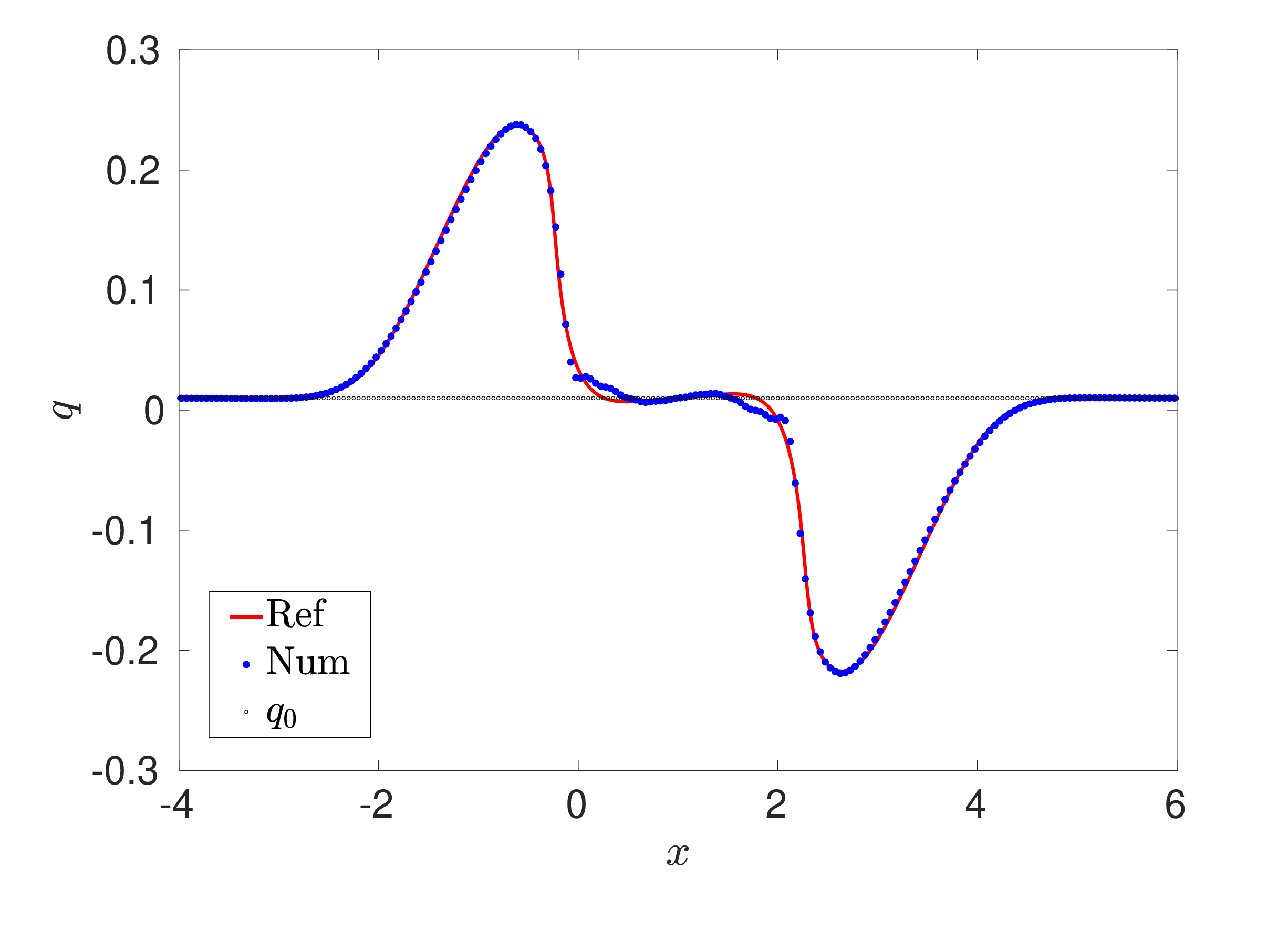}
     \caption{Test \ref{sssec:acc:SW_no_flat}: Order accuracy for shallow water with non-flat bottom topography. Comparison of the numerical approximation obtained with the third-order semi-implicit staggered scheme adopting a $200$ uniform mesh and MCFL$\leq0.4$ at time $t=0.8\ s$ with a reference solution. Free surface (right) and discharge (left). The reference solution has been obtained on a $6400$ uniform mesh. Inner figure in left figure is the initial condition for the bottom and free surface.}
     \label{Exp:Acc:SW:Test_2_no_flat}
\end{figure}
% ---------
\begin{figure}[!ht]
     \centering
     \begin{subfigure}[b]{0.32\textwidth}
         \centering
         \includegraphics[width=\textwidth]{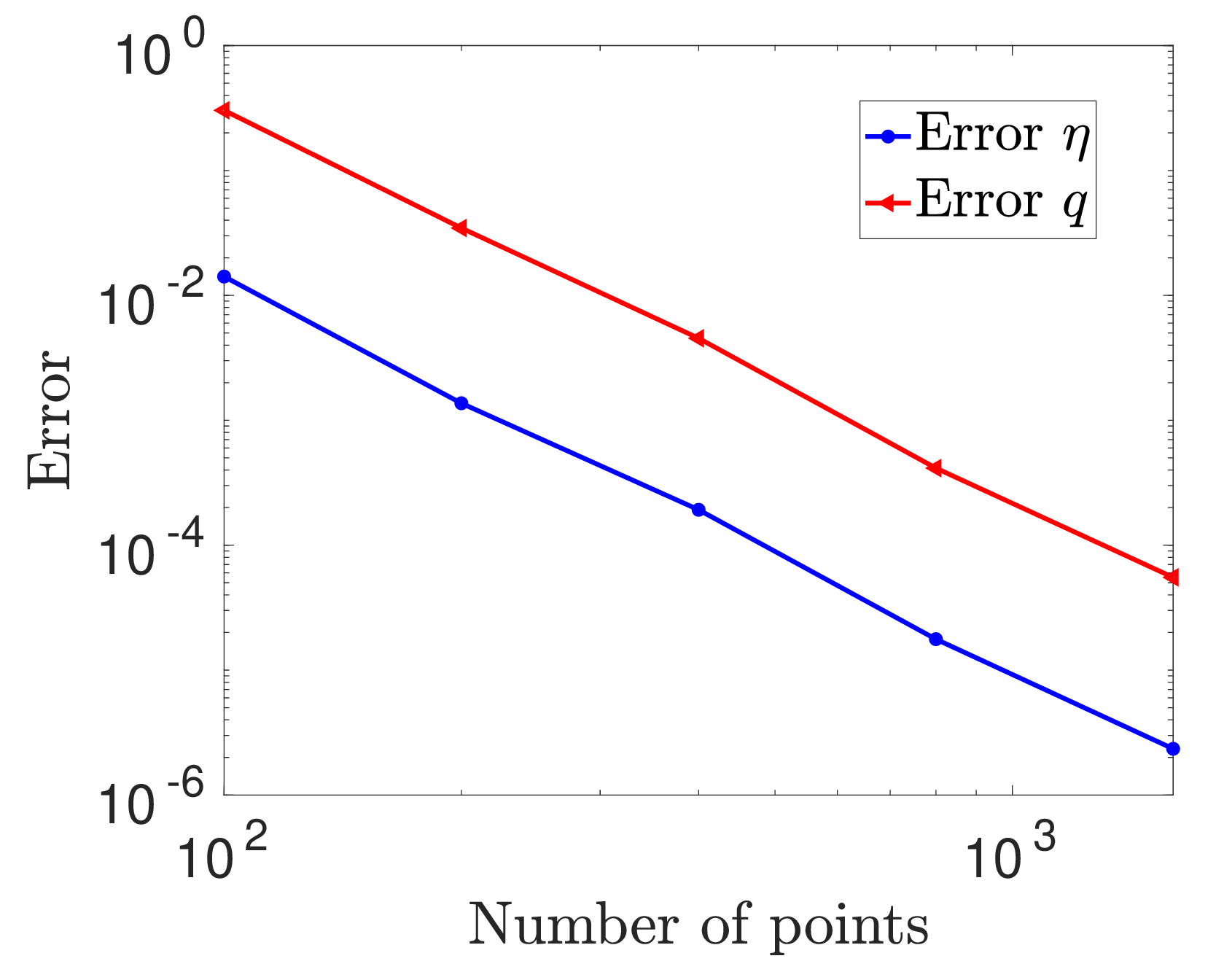}
         \caption{Points vs Errors in log scale.}
         \label{Exp:Acc:SW:P_E_no_flat}
     \end{subfigure}
     \hfill
     \begin{subfigure}[b]{0.32\textwidth}
         \centering
         \includegraphics[width=\textwidth]{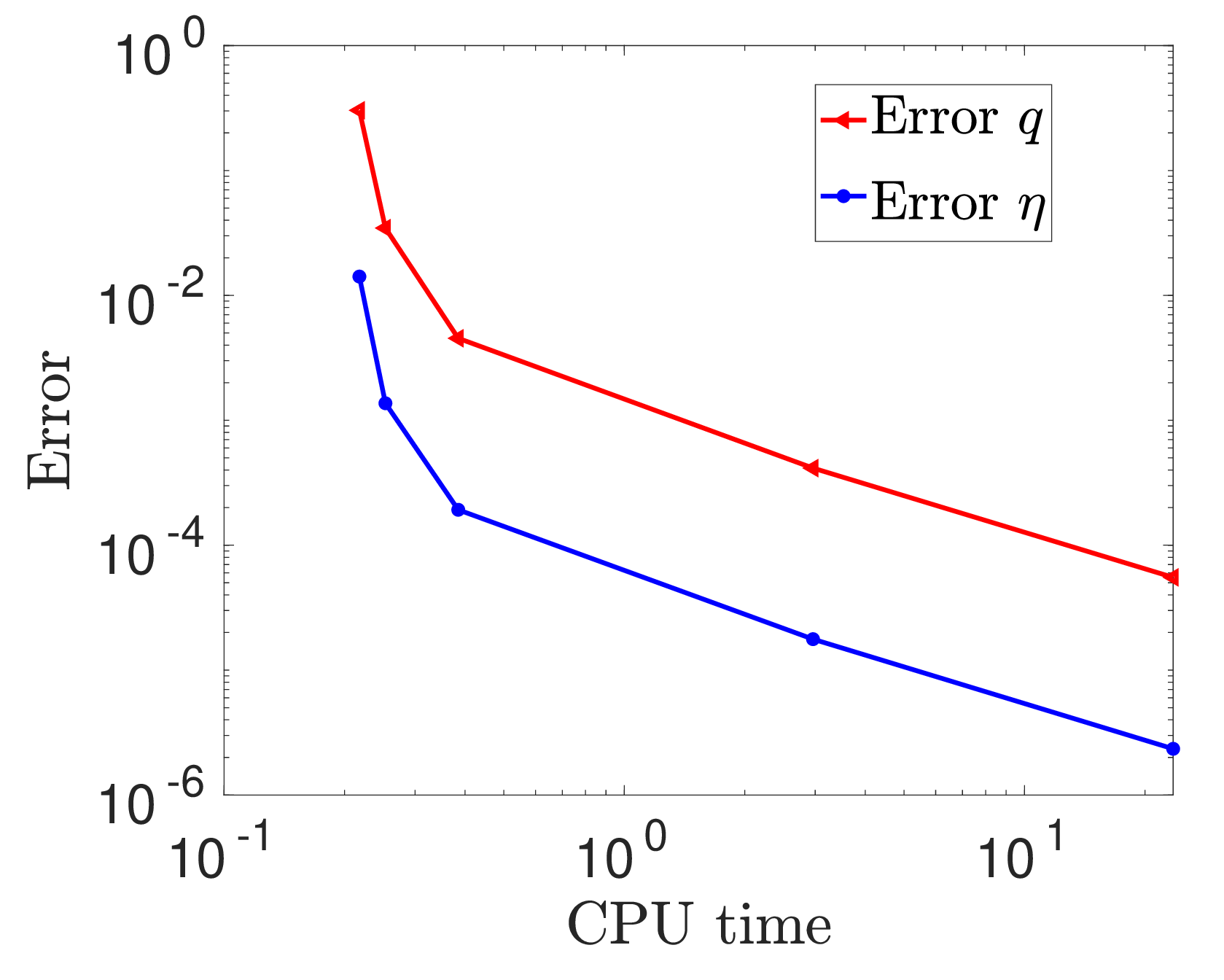}
         \caption{CPU times vs Errors in log scale.}
         \label{Exp:Acc:SW:CPU_E_no_flat}
     \end{subfigure}
     \hfill
     \begin{subfigure}[b]{0.32\textwidth}
         \centering
         \includegraphics[width=\textwidth]{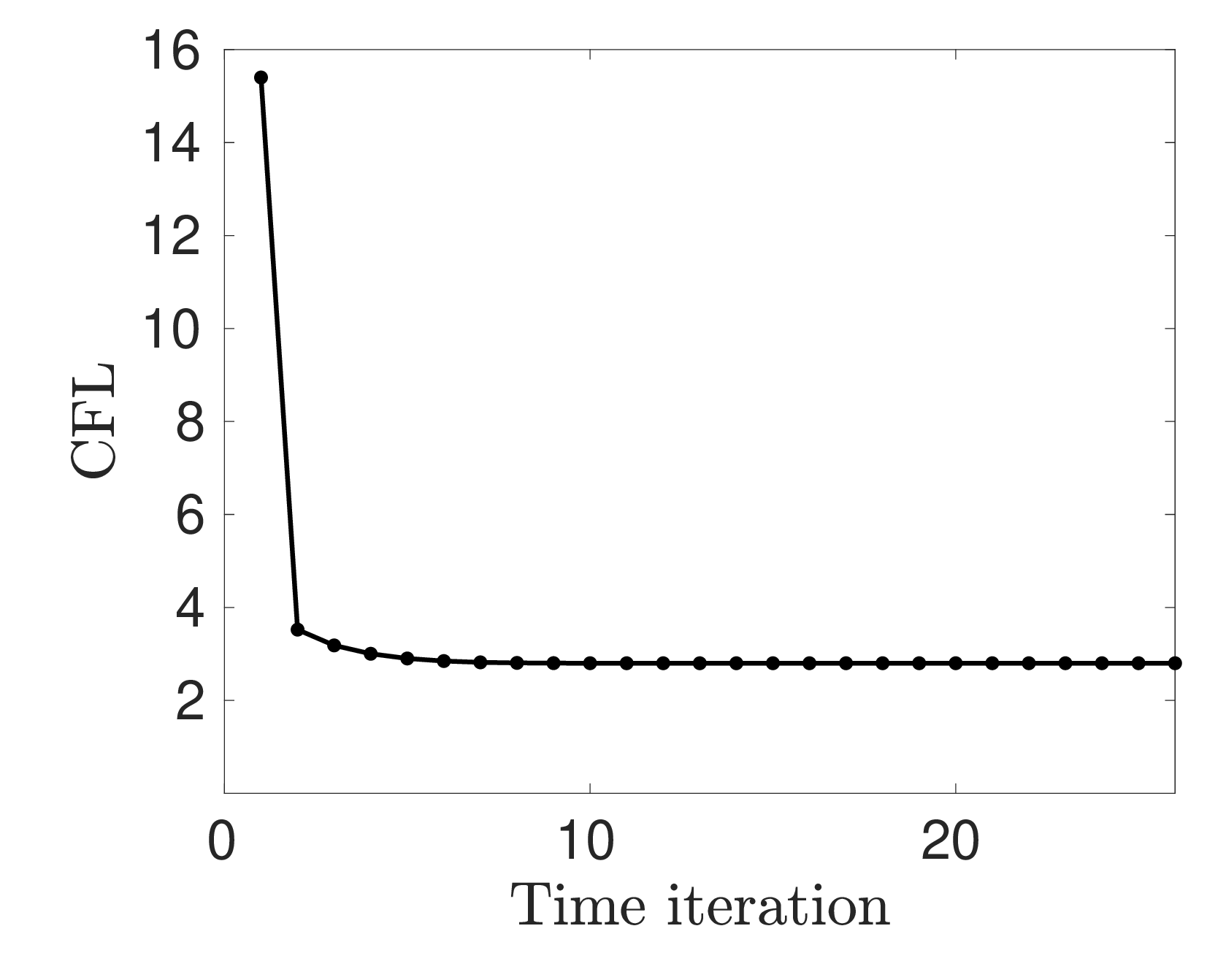}
         \caption{CFL for each time iteration.}
         \label{Exp:Acc:SW:CFL_no_flat}
     \end{subfigure}
     \caption{Test \ref{sssec:acc:SW_no_flat}: Order accuracy for shallow water with non-flat bottom topography \ref{sssec:acc:SW_no_flat}. L$^1$ errors vs number of points and L$^1$ errors vs cpu time in seconds in log scale for $\eta$ and $q$, respectively \ref{Exp:Acc:SW:P_E_no_flat}-\ref{Exp:Acc:SW:CPU_E_no_flat}. CFL condition at each time step for the numerical solution computed on a $400$ uniform mesh \ref{Exp:Acc:SW:CFL_no_flat}. The final time is $0.8\ s$ and the reference solution has been obtained on a $6400$ uniform mesh. }
     \label{Exp:Acc:SW:CPU_CFL_no_flat}
\end{figure}
% --------------
% 

%{\color{red}{\bf Modificar si se deja como inner figure} Figure~\ref{Exp:Acc:SW:Test_2_no_flat_ic} depicts the initial condition for the free surface and the bottom defined by equation (\ref{Ord_acc_SW_2_eta_0}). In Figure \ref{Exp:Acc:SW:Test_2_no_flat}, the reference solution obtained on a uniform grid of 6400 points, and the numerical solution for $\eta$ and $q$ at the final time $t = 0.8\ s$ computed on a uniform grid with 200 points and a MCFL value of $0.4$ are presented.  }
Figure ~\ref{Exp:Acc:SW:Test_2_no_flat} depicts the numerical and reference solutions obtained in this case, as well as the initial condition for the free surface and bottom (see inner figure).

Table~\ref{tab:Ord_acc_SW_no_flat} presents the convergence rates and errors in the L$^1$-norm. Moreover, in Figure~\ref{Exp:Acc:SW:CPU_CFL_no_flat}, the relationship between the number of points and the error is displayed, along with the computational cost and error in log scale \ref{Exp:Acc:SW:P_E_no_flat}-\ref{Exp:Acc:SW:CPU_E_no_flat}. The progression of the CFL number with the iterations is also shown in this figure \ref{Exp:Acc:SW:CFL_no_flat}. In this case, the CFL decreases to 2.9 after some iterations.

\subsubsection{Exner model with non-flat sedimentation} \label{sssec:acc:Ex_flat}
We consider now a similar accuracy test for the Exner case.
The common settings of this experiment are: $[x_a,x_b]=[-200,200]$ the interval; the CFL at initial condition is $3.43$; flat bottom topography $b(x) \equiv 0.01$; $A_g = 0.1;$ $ m_g = 3$ and $\xi = 1/(1-\rho_0)$ with $\rho_0 = 0.2.$  The initial condition is 
\begin{equation}
    \label{Ord_acc_Ex_1_eta_0} 
    W_0(x) = \begin{bmatrix}
    	\eta_0(x) \\ q_0(x) \\ z_{b_0}(x) 
    	\end{bmatrix}
    =
    \begin{bmatrix}
        10 \\ 
        10    \\ 
        0.1 + 2e^{-x^8/10^{14}}
    \end{bmatrix}    .
\end{equation}
% 
% FIG--------
\begin{figure}[!ht]
	\centering
	\includegraphics[width=0.49\textwidth]{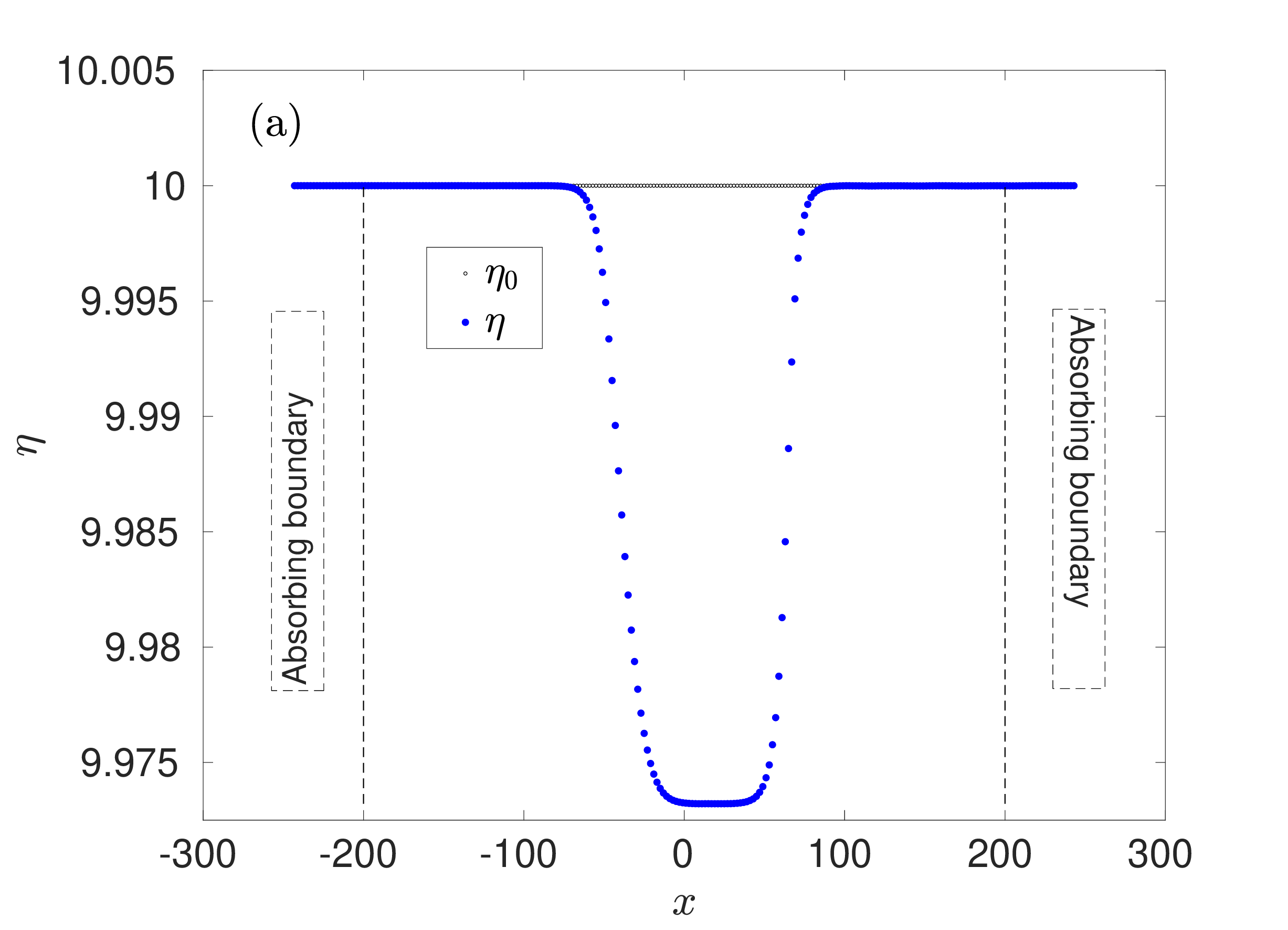}
	\includegraphics[width=0.49\textwidth]{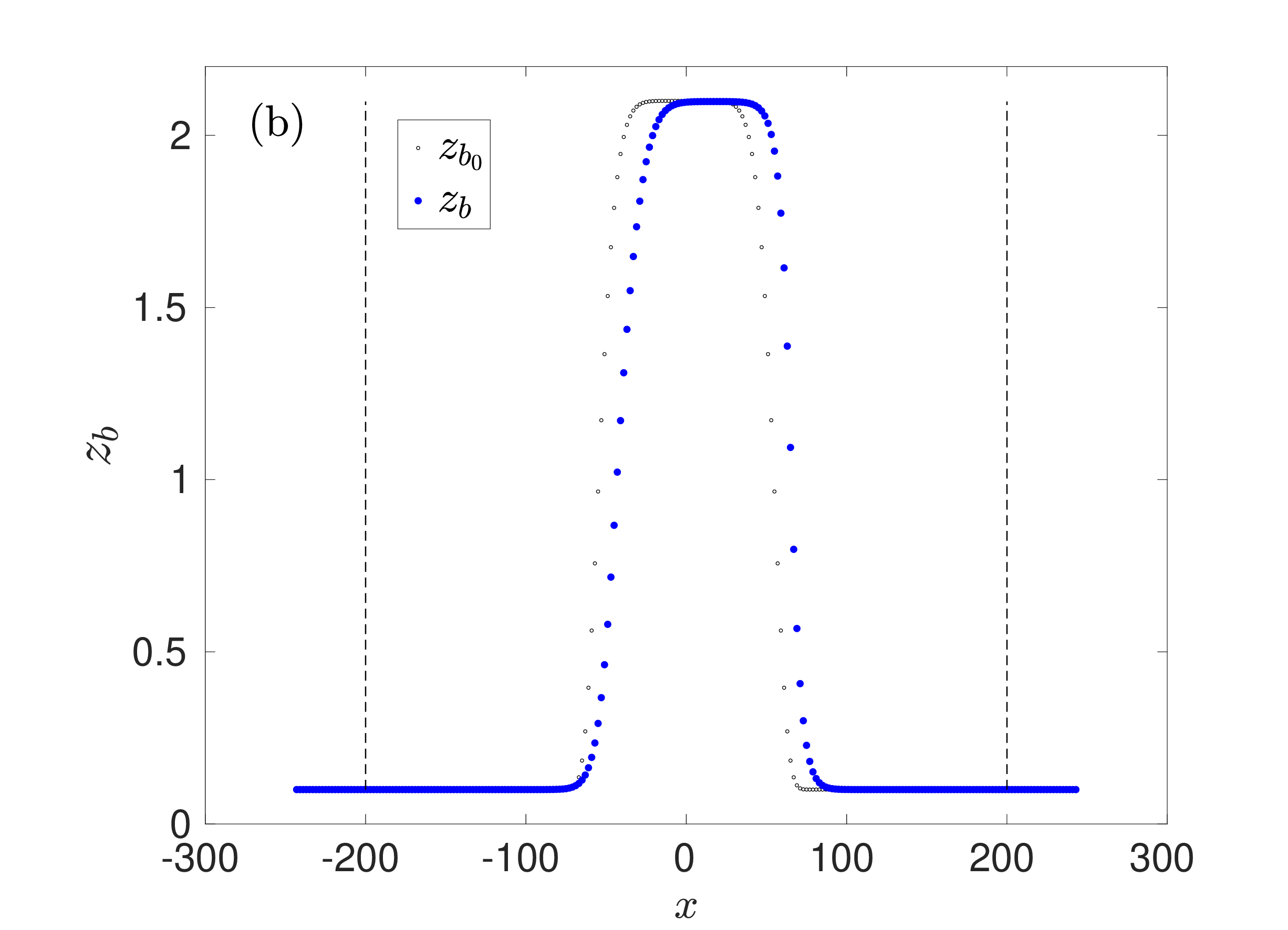}
	\includegraphics[width=0.49\textwidth]{Figures/Experiments/Accuracy/EX/m_zb.eps}
	\includegraphics[width=0.49\textwidth]{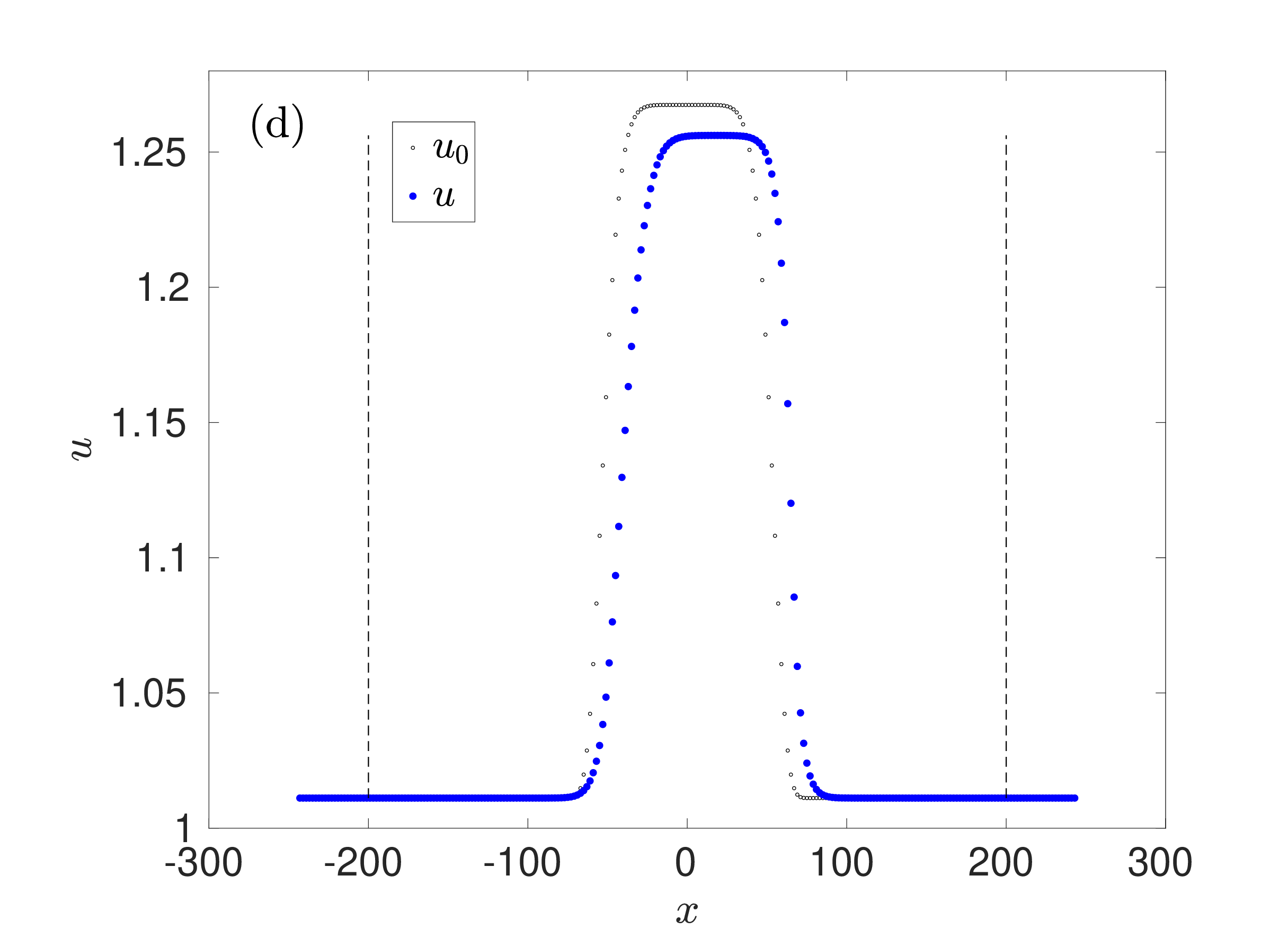}
	\caption{Test \ref{sssec:acc:Ex_flat}: Accuracy test for Exner model. Numerical solutions for (a) $\eta$; (b) $z_b$; (c) $q$; (d) $u$ obtained with the third-order semi-implicit staggered scheme adopting a $200$ uniform mesh and CFL$\approx 3.4$ at time $t=200\ s$. $A_g$ is equals to $0.1$ and $m_g = 3.$}
	\label{Exp:AC:Ex:Jose}
\end{figure}
% ---------
% 
% ---------
\begin{figure}[!ht]
     \centering
     \begin{subfigure}[b]{0.32\textwidth}
         \centering
         \includegraphics[width=\textwidth]{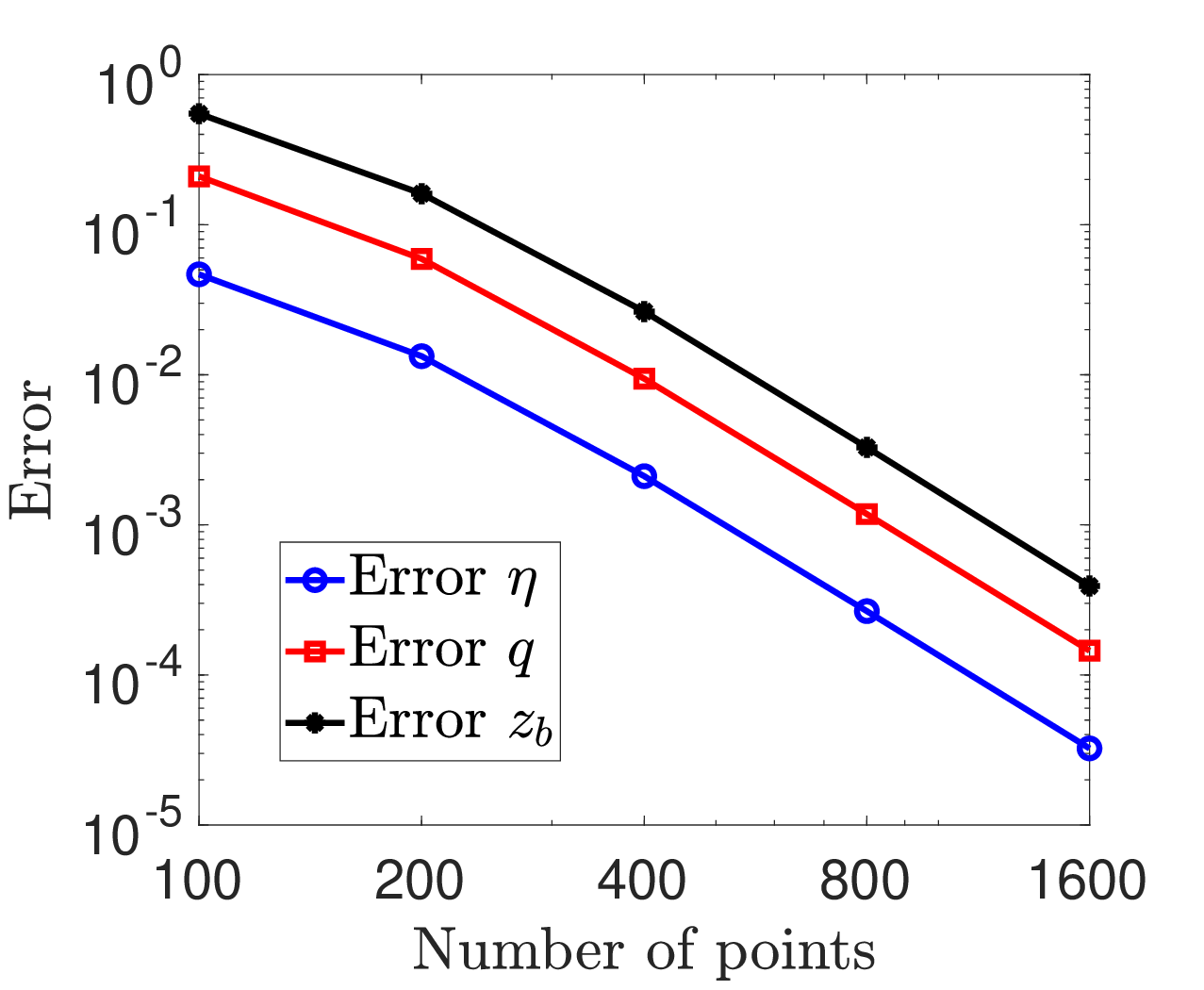}
         \caption{Points vs Errors in log scale.}
         \label{Exp:Acc:Ex:P_E_flat}
     \end{subfigure}
     \hfill
     \begin{subfigure}[b]{0.32\textwidth}
         \centering
         \includegraphics[width=\textwidth]{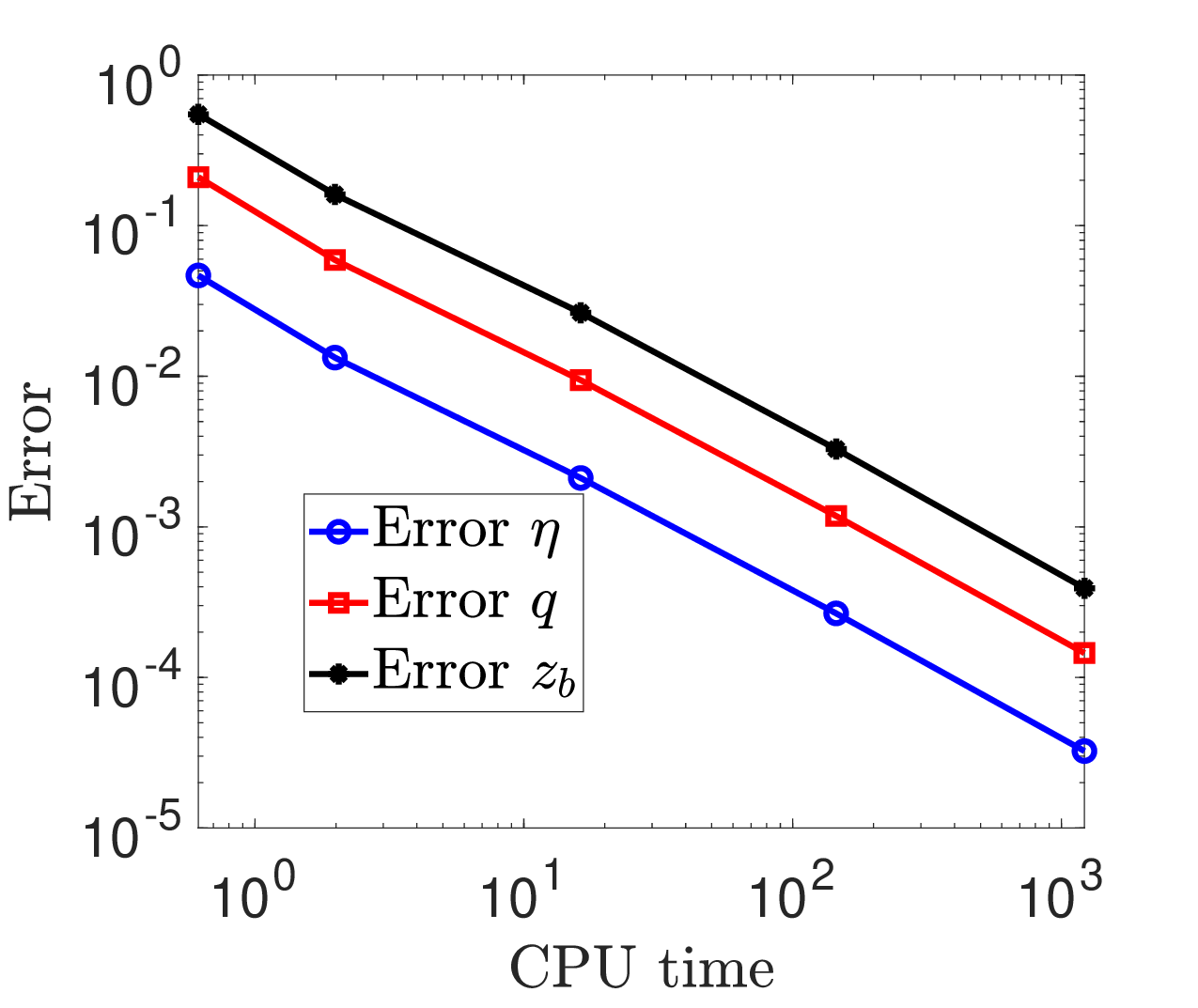}
         \caption{CPU times vs Errors in log scale.}
         \label{Exp:Acc:Ex:CPU_E_flat}
     \end{subfigure}
     \hfill
     \begin{subfigure}[b]{0.32\textwidth}
         \centering
         \includegraphics[width=\textwidth]{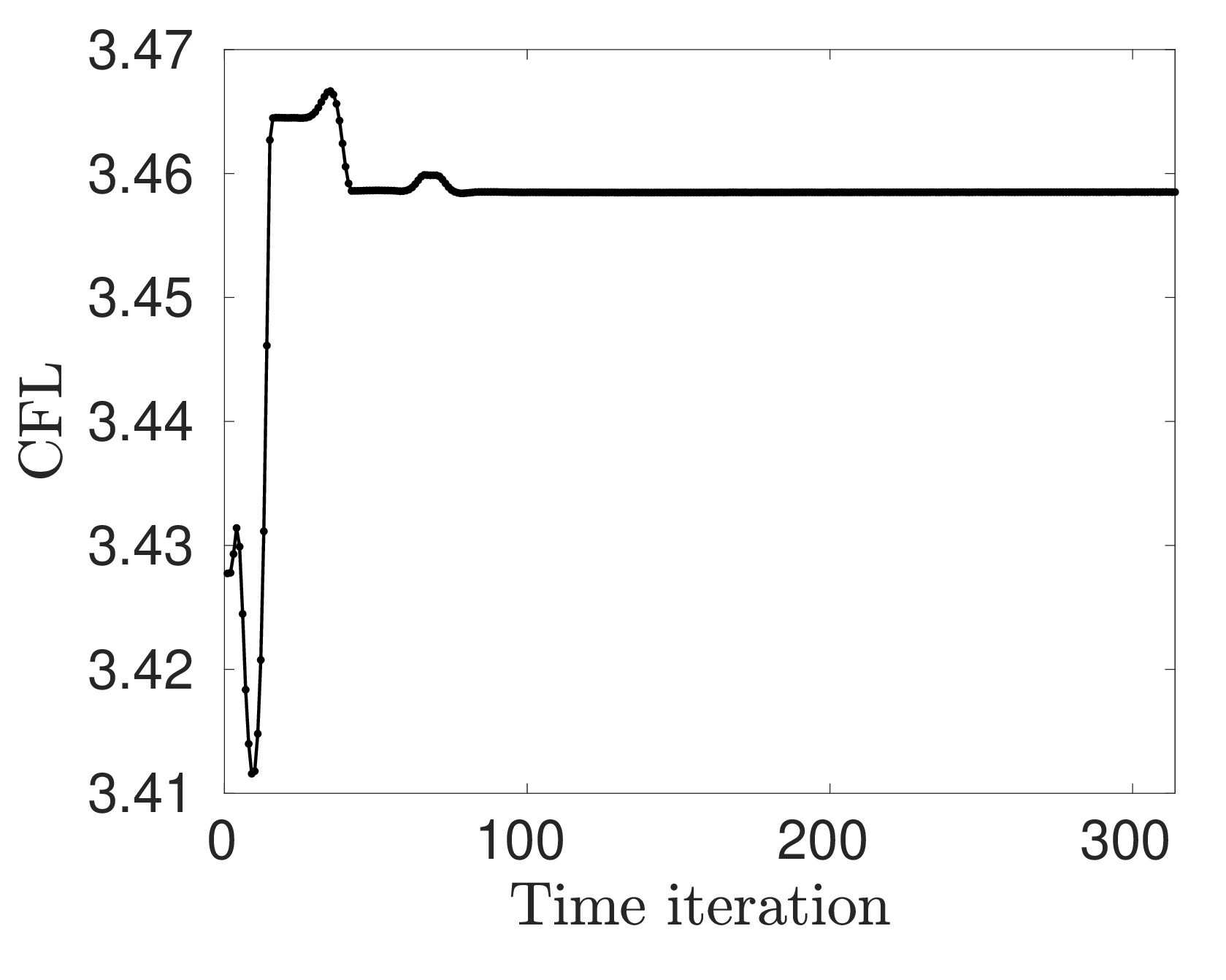}
         \caption{CFL for each time iteration.}
         \label{Exp:Acc:Ex:CFL_flat}
     \end{subfigure}
     \caption{Test \ref{sssec:acc:Ex_flat}: Order of accuracy for Exner model with flat bottom topography . L$^1$ errors vs number of points and L$^1$ errors vs cpu time in seconds in log scale for $\eta$, $q$ and $z_b$, respectively \ref{Exp:Acc:Ex:P_E_flat}-\ref{Exp:Acc:Ex:CPU_E_flat}. CFL condition at each time step for the numerical solution computed on a $200$ uniform mesh \ref{Exp:Acc:Ex:CFL_flat}. The final time is $200\ s$ and the reference solutions have been computed on a $6400$ uniform mesh. The errors have been computated in the computational domain $[-200,200]$. }
     \label{Exp:Acc:Ex:CPU_CFL_flat}
\end{figure}

% ---------
% 
% ----- Tab ----------
\begin{table}[!ht]
\numerikNine
\centering
\begin{tabular}{|c||cc|cc|cc|}
\hline \multicolumn{7}{|c|}{\textbf{Exner model - Rate of convergence}} \\
\hline  
\hline   
& \multicolumn{2}{c|}{\textbf{$\eta$}} & \multicolumn{2}{c|}{\textbf{$q$}} & \multicolumn{2}{c|}{\textbf{$z_b$}} \\ 
$N$ &  $L^1$ error &  order & $L^1$ error &  order & $L^1$ error &  order\\ \hline 
& & & & & &  \\[-3mm]
100  & 4.67$\times$10$^{-2}$  &  --  & 2.10$\times$10$^{-1}$ & --   &  5.49$\times$10$^{-1}$ & --   \\ 
200  & 1.33$\times$10$^{-2}$  & 1.81 & 5.92$\times$10$^{-2}$ & 1.83 &  1.61$\times$10$^{-1}$ & 1.77 \\
400  & 2.11$\times$10$^{-3}$  & 2.64 & 9.42$\times$10$^{-3}$ & 2.65 &  2.64$\times$10$^{-2}$ & 2.61 \\
800  & 2.66$\times$10$^{-4}$  & 2.99 & 1.18$\times$10$^{-3}$ & 2.99 &  3.28$\times$10$^{-3}$ & 3.01 \\
1600 & 3.24$\times$10$^{-5}$  & 3.04 & 1.45$\times$10$^{-4}$ & 3.03 &  3.91$\times$10$^{-4}$ & 3.07 \\
\hline
\end{tabular}
\caption{Test \ref{sssec:acc:Ex_flat}: Smooth initial condition for Exner model with flat bottom.  Convergence rates and $L^1-$norm errors for free-surface $\eta$, $q$ and $z_b$ computed with the Richardson extrapolation at $t_{\text{final}} = 200\ s$ on uniform mesh and CFL$ \approx 3.4$ such that MCFL$=0.4.$}
\label{tab:Ord_acc_Ex_flat}
\end{table}
%---------
%
Figure~\ref{Exp:AC:Ex:Jose} shows the initial and final solutions for $\eta$ (top-left); $z_b$ (top-right); $q$ (bottom-left) and $u$ (bottom-right) obtained with the third-order semi-implicit scheme developed above computed on a $200$ uniform mesh, CFL $\approx$ 3.4, see Figure~\ref{Exp:Acc:Ex:CFL_flat}, at final time $t = 200\ s$.  In Figure~\ref{Exp:Acc:Ex:CPU_CFL_flat} the relationships between the number of points and the errors are displayed, along with the computational costs and errors in log scale (see Figure \ref{Exp:Acc:Ex:P_E_flat}-\ref{Exp:Acc:Ex:CPU_E_flat}). The progression of the CFL across iterations is also shown in Figure \ref{Exp:Acc:Ex:CFL_flat}. As can be observed, at the initial time step, the CFL number is set to $3.45$ since the restriction due to the material waves, $\lambda_{\rm max}^n/u_{\rm max}^n$, impose a CFL lower than $8$. However, from the first iteration onwards, the surface waves generated by the sedimentation $z_b$ automatically reduce the CFL number such that the MCFL condition is satisfied. Once these waves leave the computational domain $[x_a, x_b]$, the CFL number returns to $3.45$. To prevent wave reflection due to the boundary conditions, two absorption domains, namely $[x_d, x_a]$ and $[x_b, x_c]$, each with a width of $45\ m$, are used at the boundaries of the computational domain (see \cite{Karni}). In figure~\ref{Exp:AC:Ex:Jose} the absorbing boundary domain are identified by dashed grey lines. These domains, with the absorbing technique presented above, ensure the passage of waves through $x_a$ and $x_b$ while automatically damping reflected ones. The final time is $200\ s$, while the errors are computed with a reference solution obtained with a $6400$ uniform mesh. Additionally, Table~\ref{tab:Ord_acc_Ex_flat} presents the convergence rates and errors in the $L^1$-norm.

\subsection{Well-balanced}
\label{ssec:well_balanced}
This section aims to demonstrate that the semi-implicit schemes employed for the shallow water equations and the Exner model, as previously introduced, possess the capability to maintain a static water surface (referred to as "lake at rest") with precision extending to machine-level accuracy. 
\subsubsection{Shallow water} \label{sssec:WB:SW}
Let us consider a similar configuration of tests \ref{sssec:acc:SW_flat} and \ref{sssec:acc:SW_no_flat}, thus flat and non-flat bottom topography, defined by $b(x) = 0.1 + \epsilon e^{-10(x-1)^2}$, with $\epsilon=0$ and $\epsilon = 0.2$, respectively.

\begin{figure}[!ht]
	\centering
		\centering
		\includegraphics[width=0.45\textwidth]{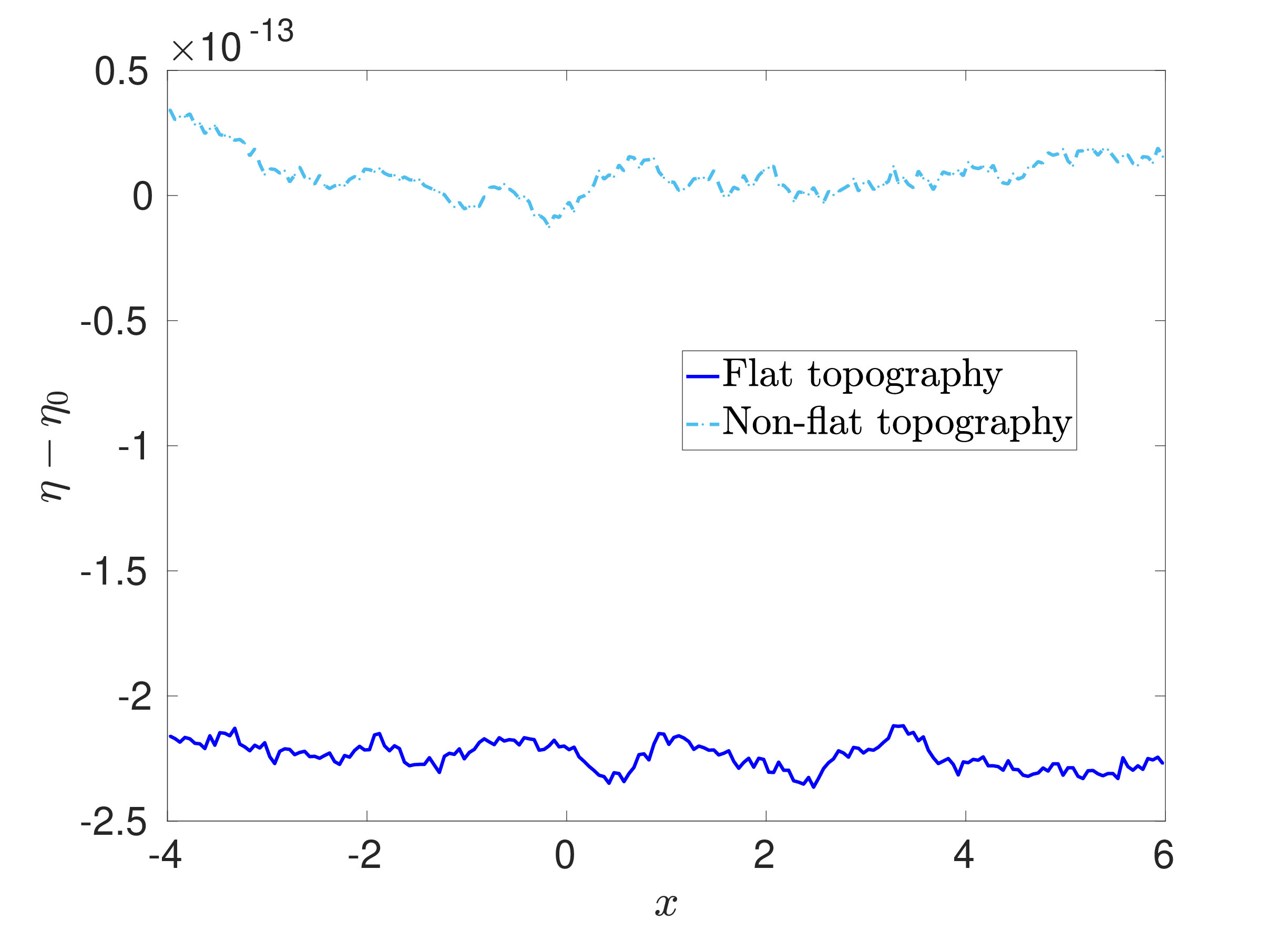}
		 \includegraphics[width=0.45\textwidth]{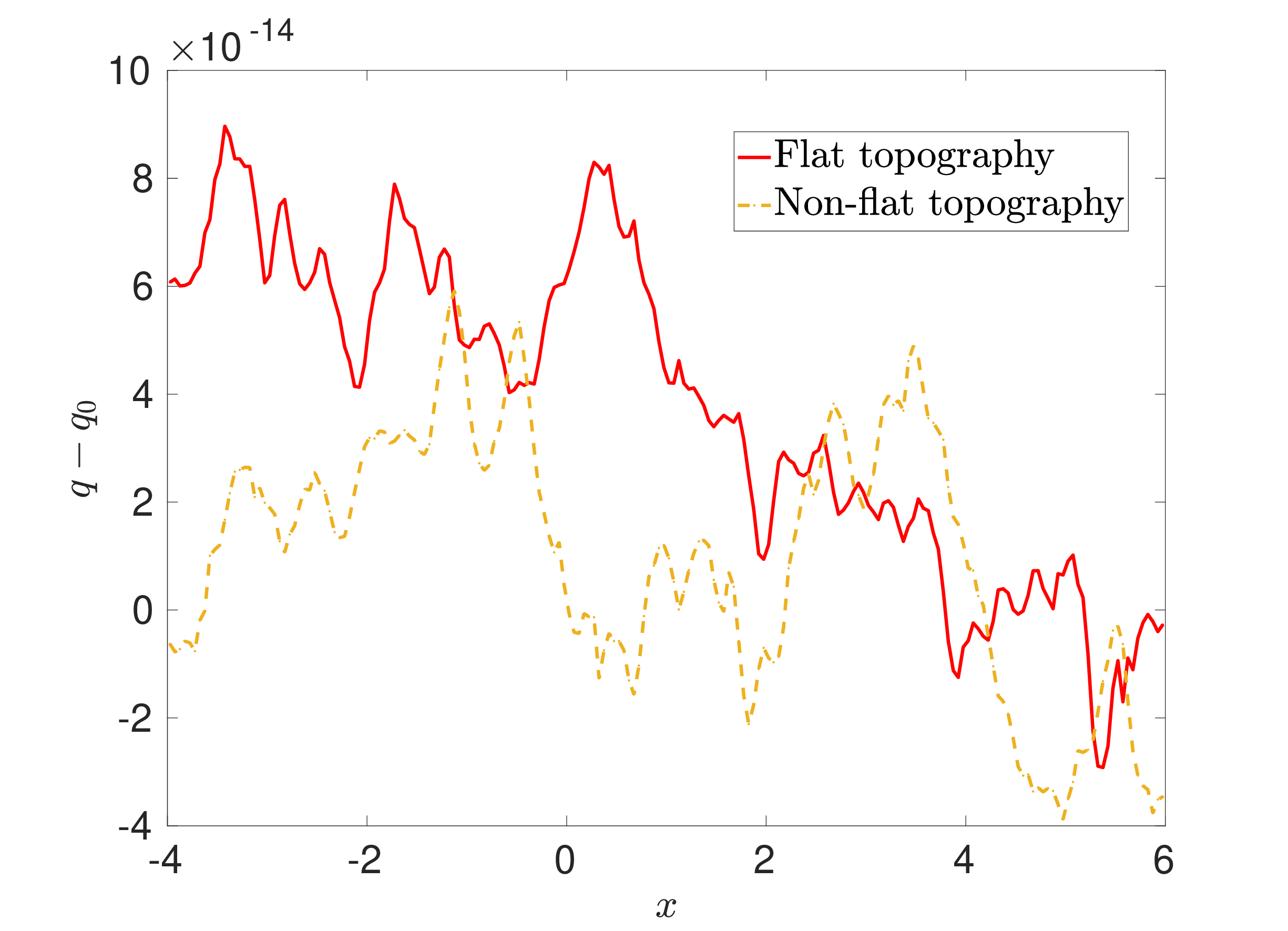}  
	\caption{Test \ref{sssec:WB:SW}: Well-balanced for shallow water with flat  and no-flat  bottom topography . Difference between numerical solution and stationary solution for $\eta$ (left) and $q$ (right) obtained with a third-order semi-implicit staggered scheme adopting a $200$ uniform mesh and CFL$=40$ at time $t=80\ s$.}
	\label{Exp:WB:SW:Test_1_flat}
\end{figure}

% -----------------
% 
% -----------------
%\begin{figure}[!ht]
%     \centering
%     \begin{subfigure}[b]{0.49\textwidth}
%         \centering
%         \includegraphics[width=0.9\textwidth]{Figures/Experiments/WB/SW/m_SW_WB_1_eta.eps} \includegraphics[width=0.9\textwidth]{Figures/Experiments/WB/SW/m_SW_WB_1_q.eps}  
%         \caption{Difference between numerical solution and stationary solution for $\eta$ (left) and $q$ (right) with flat topography.}
%         \label{Exp:WB:SW:flat}
%     \end{subfigure}
%     \hfill
%     \begin{subfigure}[b]{0.49\textwidth}
%         \centering
%         \includegraphics[width=0.9\textwidth]{Figures/Experiments/WB/SW/m_SW_WB_2_eta.eps}
%         \\ \includegraphics[width=0.9\textwidth]{Figures/Experiments/WB/SW/m_SW_WB_2_q.eps}
%         \caption{Difference between numerical solution and stationary solution for $\eta$ (left) and $q$ (right) with non-flat topography.}
%         \label{Exp:WB:SW:no_flat}
%     \end{subfigure}
%     \caption{Test \ref{sssec:WB:SW}: Well-balanced for shallow water with flat  and no-flat  bottom topography . Different between numerical solution and stationary solution for $\eta$ (left) and $q$ (right) obtained with a third-order semi-implicit staggered scheme adopting a $200$ uniform mesh and CFL$=40$ at time $t=80$.}
%     \label{Exp:WB:SW:Test_1_flat}
%\end{figure}
% -------------
% 
\begin{equation}
    \label{WB_SW_1_eta_0} 
    U_0(x) = 
    \begin{bmatrix}
        0.7  \\
        0        
    \end{bmatrix}   . 
\end{equation}

Figure~\ref{Exp:WB:SW:Test_1_flat} show the difference between the numerical solutions and the stationary solutions for $\eta$ and $q$, considering both flat  and non-flat  bottom topography at the final time $t = 80\ s$. These computations were performed on a uniform grid with 200 points and a CFL value of $40$. In both scenarios, the semi-implicit scheme effectively preserves the stationary solution up to machine precision, irrespective of the shape of the fixed bottom.

\subsubsection{Exner model} \label{sssec:WB:EX}
Following a similar configuration of Test \ref{sssec:WB:SW}, $[x_a,x_b]=[-4,6]$ is the interval; $g$ is the gravitational constant $g=9.81$ $m/s^2$; the CFL at initial condition is $40$. In this test it is shown the well-balanced properties of the method when the sediment layer is defined as the sum of a flat bottom topography $b(x)  \equiv 0 $ plus a sediment layer because there is no movement of the fluid and Grass model is considered. We consider $A_g=0.01$  the interaction between fluid and sediment; $m_g = 3$ and $\xi = 1/(1-\psi_0)$ with uniform porosity $\psi_0 = 0.2$. The initial condition is  
% -----------------
% 
\begin{figure}[!ht]
	\centering
	\includegraphics[width=0.5\textwidth]{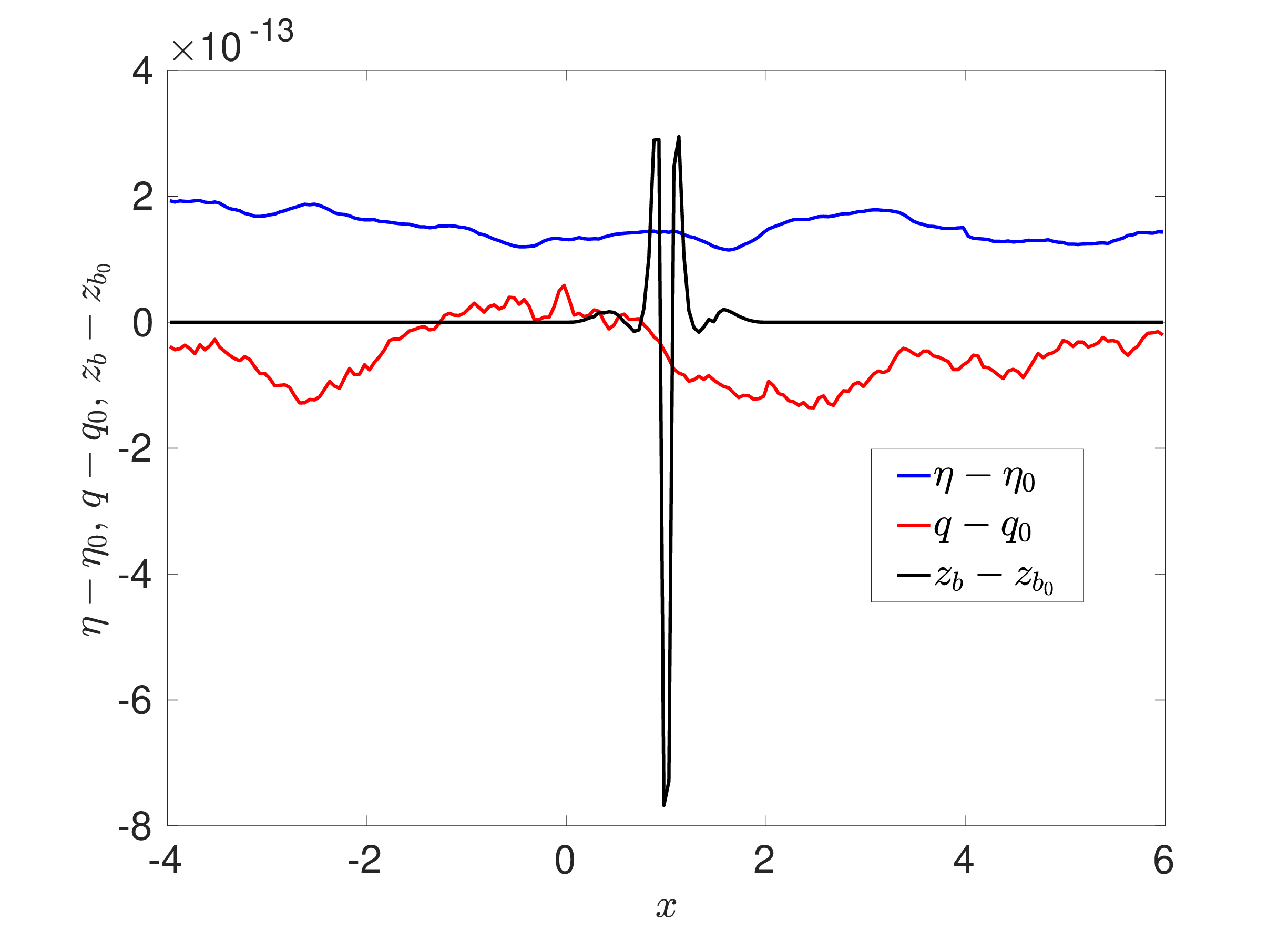}
	\caption{Test \ref{sssec:acc:Ex_flat}: Well-balanced for Saint-Venant-Exner model with flat bottom topography and sediment layer $z_b$ . Difference between numerical solution and stationary solution for $\eta$ (blue line), $z_b$ (black line) and $q$ (red line),  obtained with a third-order semi-implicit staggered scheme adopting a $200$ uniform mesh and CFL$=40$ at time $t=80\ s$.}
	\label{Exp:WB:Ex:Test_1_flat}
\end{figure}
% -----------------
%\begin{figure}[!ht]
%     \centering
%     \begin{subfigure}[b]{0.49\textwidth}
%         \centering
%         \includegraphics[width=0.9\textwidth]{Figures/Experiments/WB/EX/m_EX_WB_eta.eps} \includegraphics[width=0.9\textwidth]{Figures/Experiments/WB/EX/m_EX_WB_zb.eps}  
%         \caption{Different between numerical solution and stationary solution for $\eta$ (top) and $z_b$ (bottom).}
%         \label{Exp:WB:EX:flat}
%     \end{subfigure}
%     \hfill
%     \begin{subfigure}[b]{0.49\textwidth}
%         \centering
%         \includegraphics[width=0.9\textwidth]{Figures/Experiments/WB/EX/m_EX_WB_q.eps}
%         \\ \includegraphics[width=0.9\textwidth]{Figures/Experiments/WB/EX/m_EX_WB_h.eps}
%         \caption{Different between numerical solution and stationary solution for $q$ (top) and $h$ (bottom).}
%         \label{Exp:WB:EX:flat_1}
%     \end{subfigure}
%     \caption{Test \ref{sssec:acc:Ex_flat}: Well-balanced for Saint-Venant-Exner model with flat bottom topography and sediment layer $z_b$ . Difference between numerical solution and stationary solution for $\eta$ and $z_b$ \ref{Exp:WB:EX:flat} and $q$ and $h$ \ref{Exp:WB:EX:flat_1}, respectively (left) and (right),  obtained with a third-order semi-implicit staggered scheme adopting a $200$ uniform mesh and CFL$=40$ at time $t=80$.}
%     \label{Exp:WB:Ex:Test_1_flat}
%\end{figure}
% -------------
% 
\begin{equation}
    \label{WB_EX_1_eta_0} 
    W_0(x) = 
    \begin{bmatrix}
        0.7  \\
        0  \\
        0.1 + 0.2e^{-10(x-1)^2}
    \end{bmatrix}    
\end{equation}

Figure~\ref{Exp:WB:Ex:Test_1_flat} shows the discrepancy between the numerical solutions and the stationary solutions for $\eta$, $q$, $z_b$ at the final time $t = 80\ s$. These computations were performed on a uniform grid with 200 points and a CFL value of $40$. For all the variables, the semi-implicit scheme effectively preserves the stationary solution up to machine precision, also when the the bottom is computed as un unknown of the Saint-Venant-Exner model, instead to defined it as a given fixed bottom topography for the shallow-water system.

\subsection{Riemann problem}\label{ssec:Riemann}
This section focuses on evaluating the performance of the semi-implicit scheme under discontinuous initial conditions. Specifically, a Riemann problem that satisfies the Rankine-Hugoniot conditions is presented to assess the accurate positioning of the shock, and mass conservation 
and to measure conservation error introduced in the momentum (discharge) equation.

\subsubsection{Exact Riemann problem for shallow water}\label{sssec:SW:RP:exact}
Following \cite{Toro2009}, the exact Rankine-Hugoniot conditions satisfy
\begin{align}    \label{RH}
    -v & \jump{h} + \jump{q} =0 \\ 
    -v & \jump{q} + \Jump{ \frac{q^2}{h} + \ha g h^2 } = 0,
\end{align}
where $\jump{w} \equiv w_R - w_L$ and $v\in ]u_R + \sqrt{gh_R},u_L + \sqrt{gh_L}[.$  
% -----------------
\begin{figure}[!ht]
     \centering
         \centering
         \includegraphics[width=0.49\textwidth]{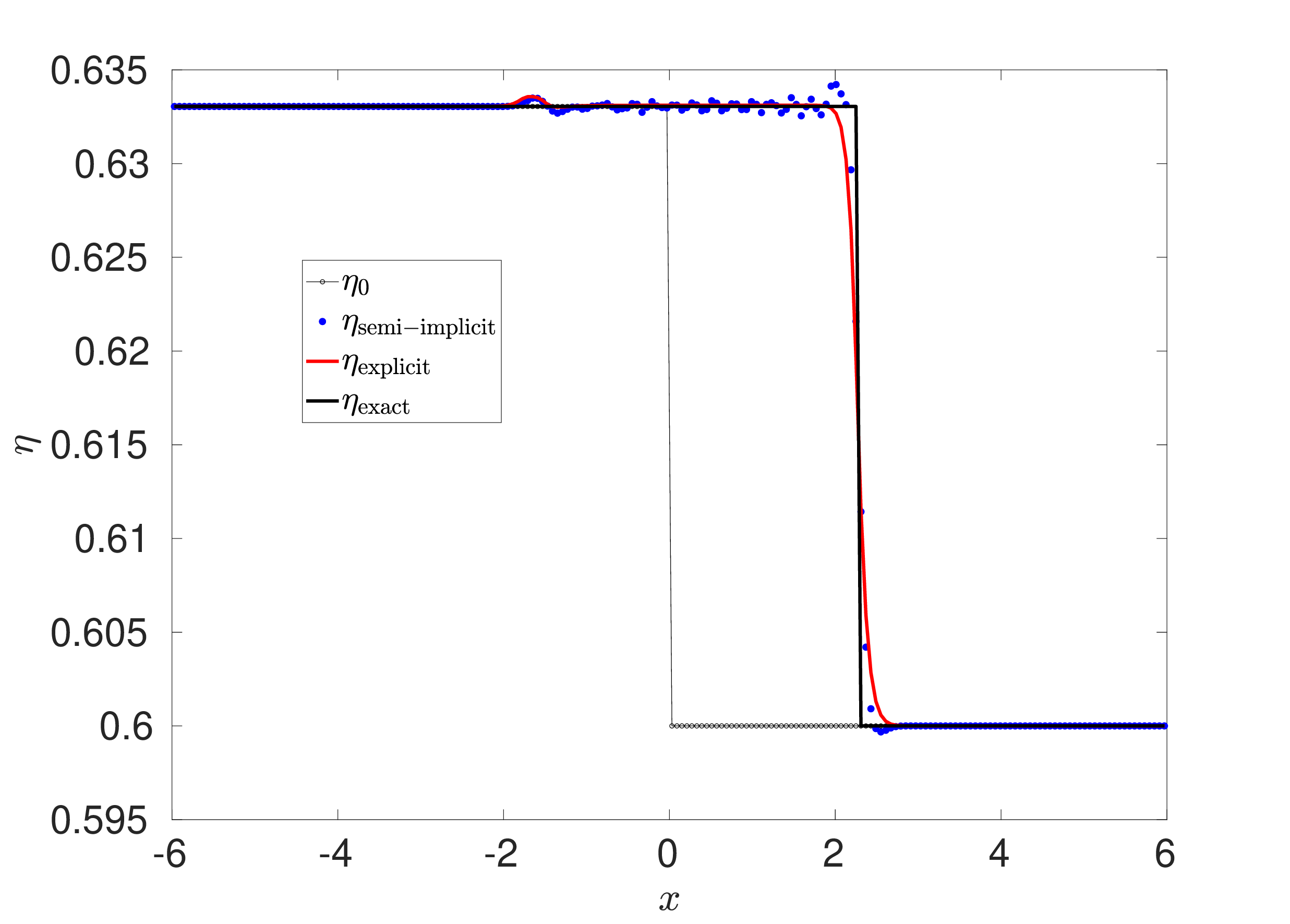} 
         \includegraphics[width=0.49\textwidth]{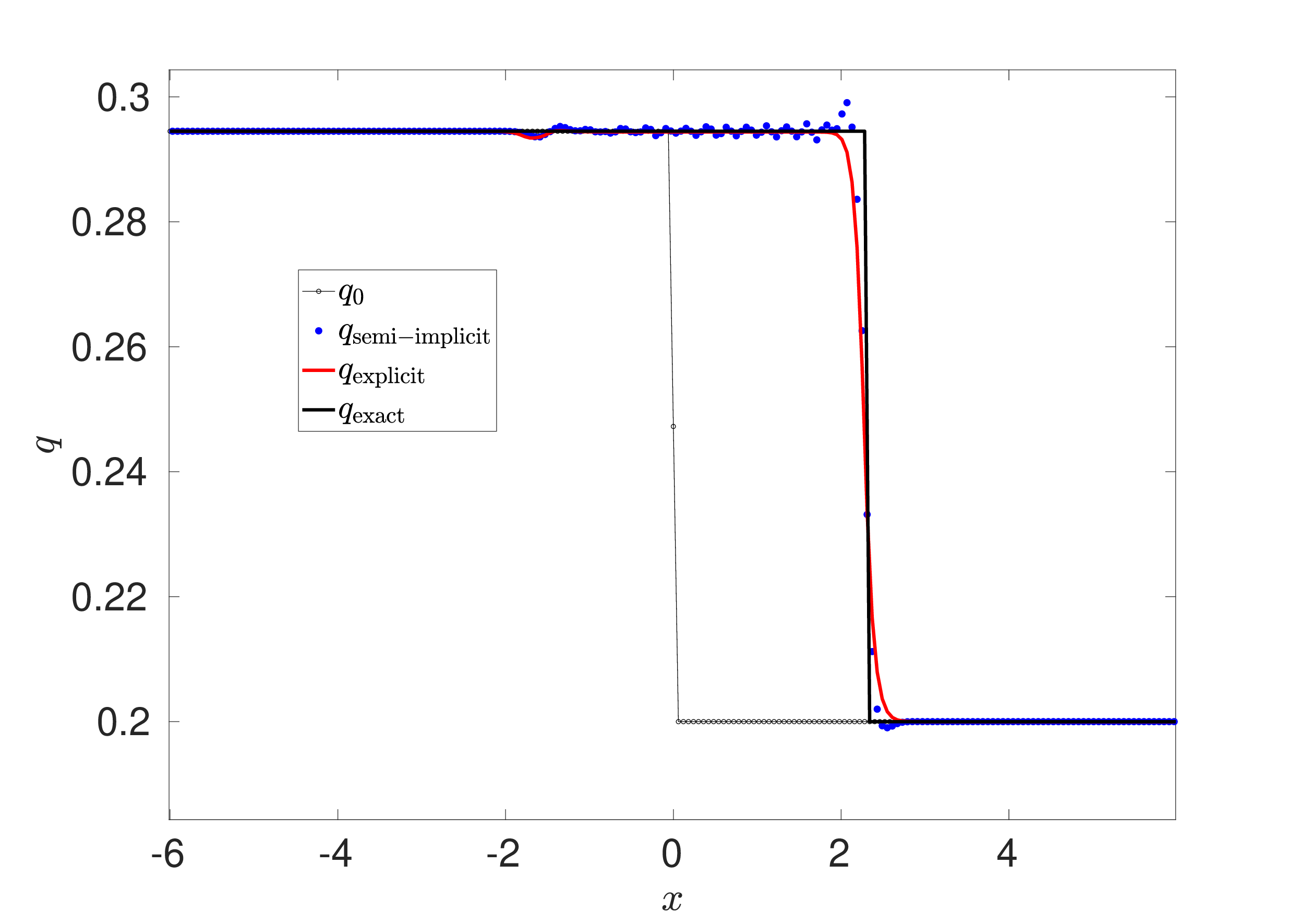}
     \caption{Test \ref{sssec:SW:RP:exact}: Exact Riemann problem for shallow water equations. Numerical solutions for $\eta$ and $q$, respectively,  obtained with a third-order semi-implicit staggered scheme, a first order explicit method and the analytical exact solution adopting a $200$ uniform mesh and CFL$=2.54$ at time $t=0.8\ s$. The explicit CFL is $0.4$ as the MCFL.}
     \label{Exp:SW:RP:exact}
\end{figure}
% ----------
%
% ----------
\begin{figure}[!ht]
     \centering
         \centering
         \includegraphics[width=0.4\textwidth]{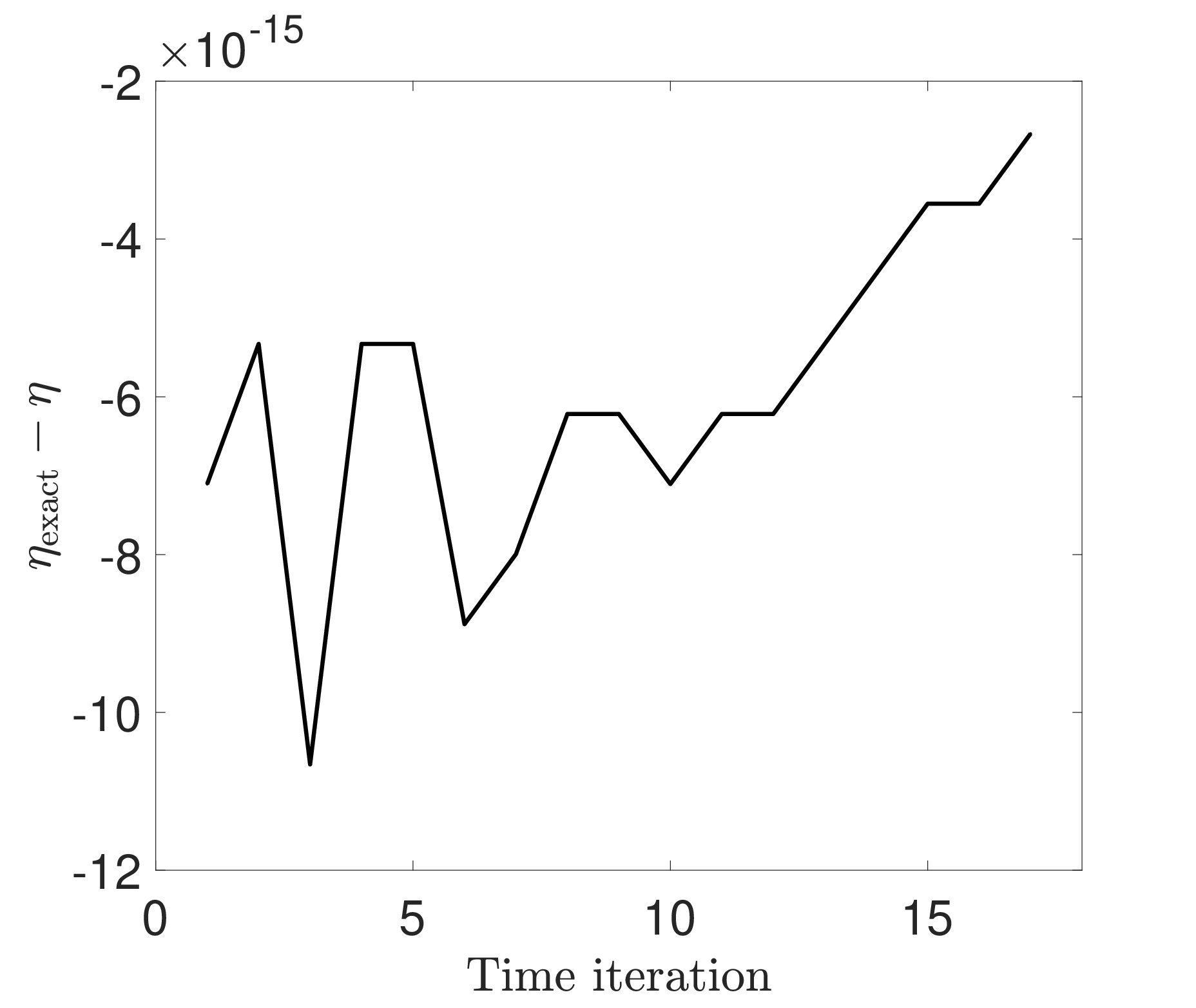} 
         \qquad 
         \includegraphics[width=0.4\textwidth]{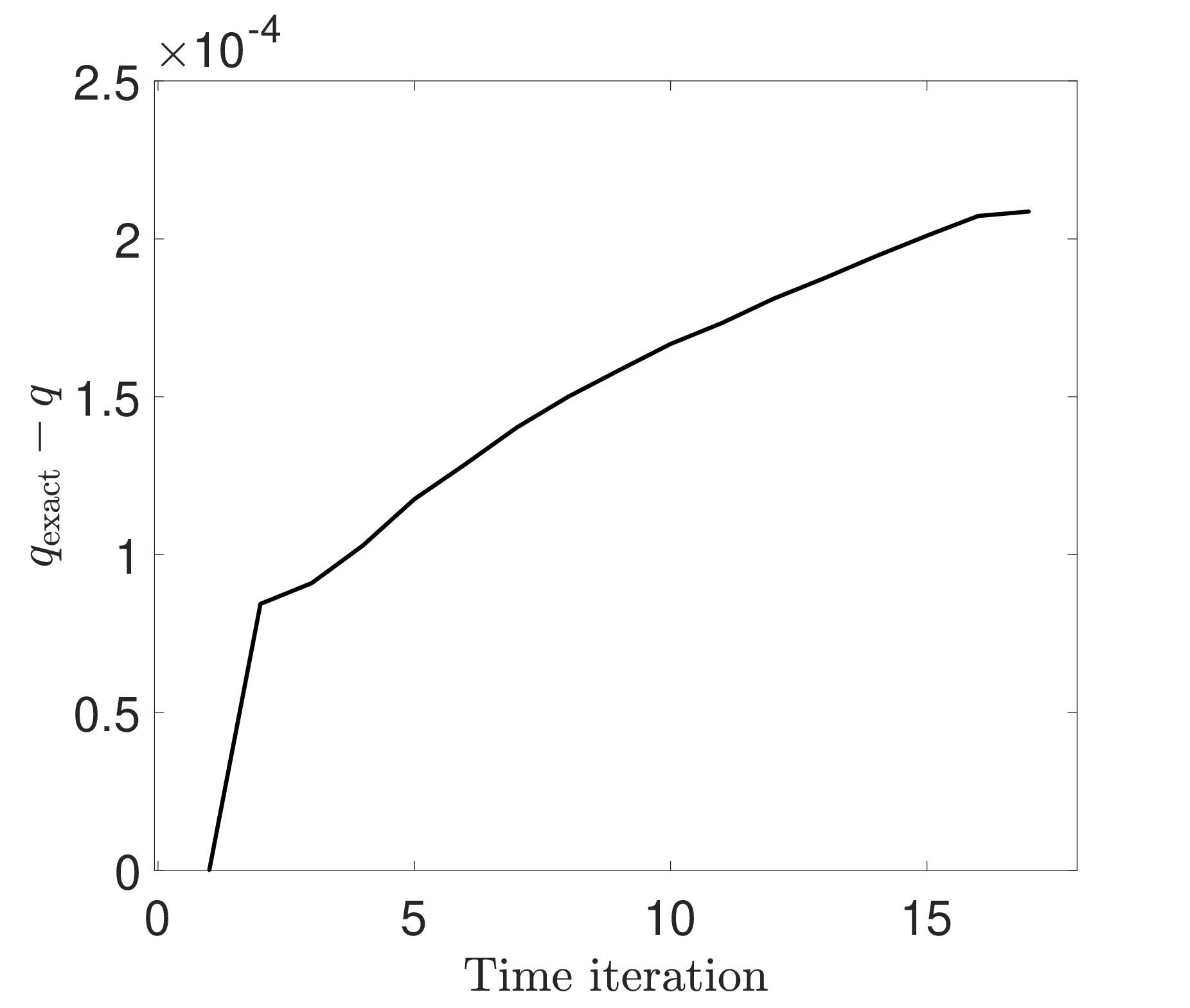}
     \caption{Test \ref{sssec:SW:RP:exact}: Exact Riemann problem for shallow water equations. Difference between integral of exact and numerical solutions for $\eta$ (left) and $q$ (right)  obtained with a third-order semi-implicit staggered scheme adopting a $200$ uniform mesh and CFL$=2.54$ at time $t=0.8\ s$.}
     \label{Exp:SW:RP:exact_con}
\end{figure}
% ----------

The common parameters for this experiment are as follows: the interval $[x_a, x_b] = [-4, 6]$, a CFL number of $2.54$, $h_R = 0.6$, $q_R = 0.2$, and $v = \lambda_R + 0.1$ where $\lambda_R = u_R + \sqrt{gh_R}$.

Figure~\ref{Exp:SW:RP:exact} illustrates the shock evolution for a Riemann problem that satisfies the Rankine-Hugoniot conditions \eqref{RH}. It is evident that the third-order semi-implicit scheme successfully captures the shock evolution, despite exhibiting some numerical oscillations near the discontinuity. The numerical solutions were computed on a uniform mesh with 200 points at a final time $t = 0.8\ s$, using a CFL number of $2.54$ for the semi-implicit scheme and $0.4$ for the explicit first-order method. Both schemes effectively preserve the shock evolution. Specifically, Figure~\ref{Exp:SW:RP:exact_con} shows that the semi-implicit scheme maintains the thickness up to machine precision, as expected. However, it introduces some error in the discharge due to the inherently non-conservative nature of the system.

This experiment reveals that, in the presence of shocks, the semi-implicit scheme can accurately capture the evolution of the shock. However, the approach used for the pressure term $h\partial_x\eta$ through high-order interpolation with the optimal CWENO polynomial is insufficient by itself to prevent and eliminate oscillations. Therefore, advanced semi-implicit shock-capturing techniques, such as \emph{Quinpi} scheme \cite{Quinpi}, should be employed. Nonetheless, this issue is beyond the scope of the present work and will be addressed in future studies.

\begin{comment}
\subsubsection{Non exact Riemann problem for shallow water}\label{sssec:SW:RP}
Following \cite{Toro2009} \ema{Fix it}, the Riemann problem...

\begin{figure}[!ht]
     \centering
     \begin{subfigure}[b]{0.49\textwidth}
         \centering
         \includegraphics[width=\textwidth]{Figures/Experiments/Shock/SW/RP/RP_eta.eps} \includegraphics[width=\textwidth]{Figures/Experiments/Shock/SW/RP/RP_q.eps}  
         \caption{Numerical solutions for $\eta$ (left) and $q$ (right).}
         \label{Exp:SW:RP:shock_sol}
     \end{subfigure}
     \hfill
     \begin{subfigure}[b]{0.49\textwidth}
         \centering
         \includegraphics[width=\textwidth]{Figures/Experiments/Shock/SW/RP/RP_eta_zoom.eps}
         \\ \includegraphics[width=\textwidth]{Figures/Experiments/Shock/SW/RP/RP_q_zoom.eps}
         \caption{Zoom of the numerical solutions for $\eta$ (left) and $q$ (right).}
         \label{Exp:SW:RP:shock_zoom}
     \end{subfigure}
     \caption{Riemann problem for shallow water with discontinuous flat bottom topography \ref{sssec:SW:RP}. Numerical solutions and zoom for $\eta$ and $q$ \ref{Exp:SW:RP:shock_sol} and  \ref{Exp:SW:RP:shock_zoom}, respectively (left) and (right),  obtained with a third-order semi-implicit staggered scheme adopting a $200$ uniform mesh and CFL$=1.22$ at time $t=0.1$. The explicit solutions have been obtained on the same setting with CFL$_{\rm ex} = 0.4$.}
     \label{Exp:SW:RP:Test_shock_RP}
\end{figure}
\end{comment}

\subsection{Sediment evolution} \label{sssec:SA:long}
Following \cite{Macca2024}, in this section the long time sediment simulation is provided. In particular, the third-order semi-implicit solution, presented below, has been compared with the semi-analytical solution computed by the second order CAT2 \cite{Macca2024CATMOOD} applied to \eqref{svesemianalytical} and the quasi-stationary solution obtained by the second order approximation of quasi-stationary partial differential equation \eqref{Scal_eq}. With this purpose in mind, we have to consider different CFL conditions, and consequently different time steps, one for each scheme that we adopt. For this reason, we define CFL$_{\rm q-stat}$ the CFL condition used for the scalar equation scheme \eqref{Scal_eq}; CFL$_{\rm q-scal}$ the CFL condition adopted for the explicit scheme applied to scalar equation  \eqref{svesemianalytical}; and CFL$_{I\rm MEX}$ the CFL condition adopted for the semi-implicit scheme applied to full system  \eqref{First_ord_Exner}. In particular we set CFL$_{\rm q-scal}=0.9$ for the explicit scalar scheme \eqref{svesemianalytical}; CFL$_{\rm q-stat}=0.45$ for the second order explicit scheme applied to \eqref{Scal_eq}; and  for the semi-implicit methods a larger CFL condition could be used, however, since the term $qu$ \eqref{First order} is treated explicitly, the semi-implicit CFL condition could not be arbitrary larger and a material CFL condition must be satisfied. In our case CFL$_{\rm IMEX} = 9.2$ is adopted. These CFL conditions come out from the natural stability conditions of the different methods.
%
% -------------
\begin{figure}[!ht]
         \centering
         \includegraphics[width=0.32\textwidth]{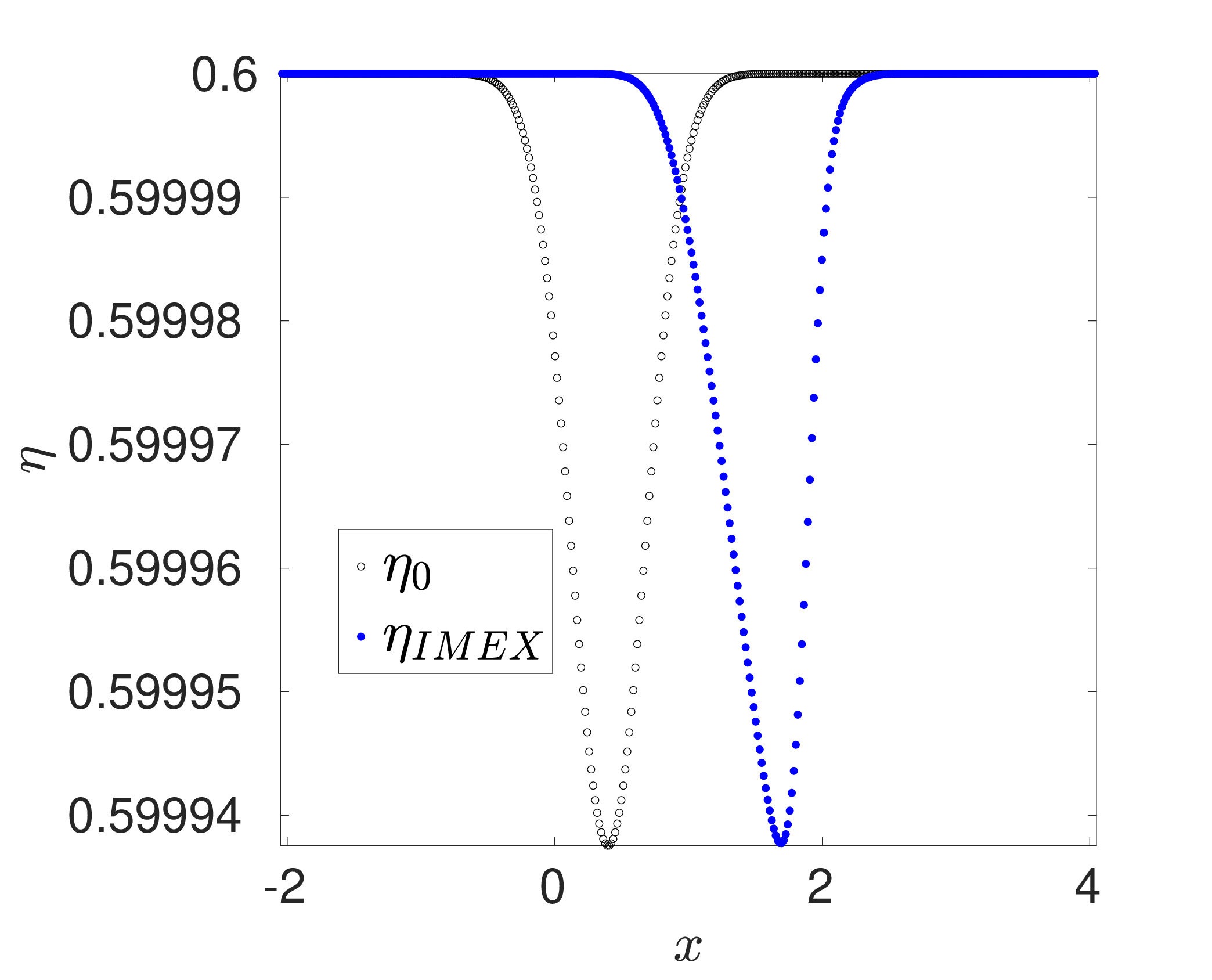}
         \includegraphics[width=0.32\textwidth]{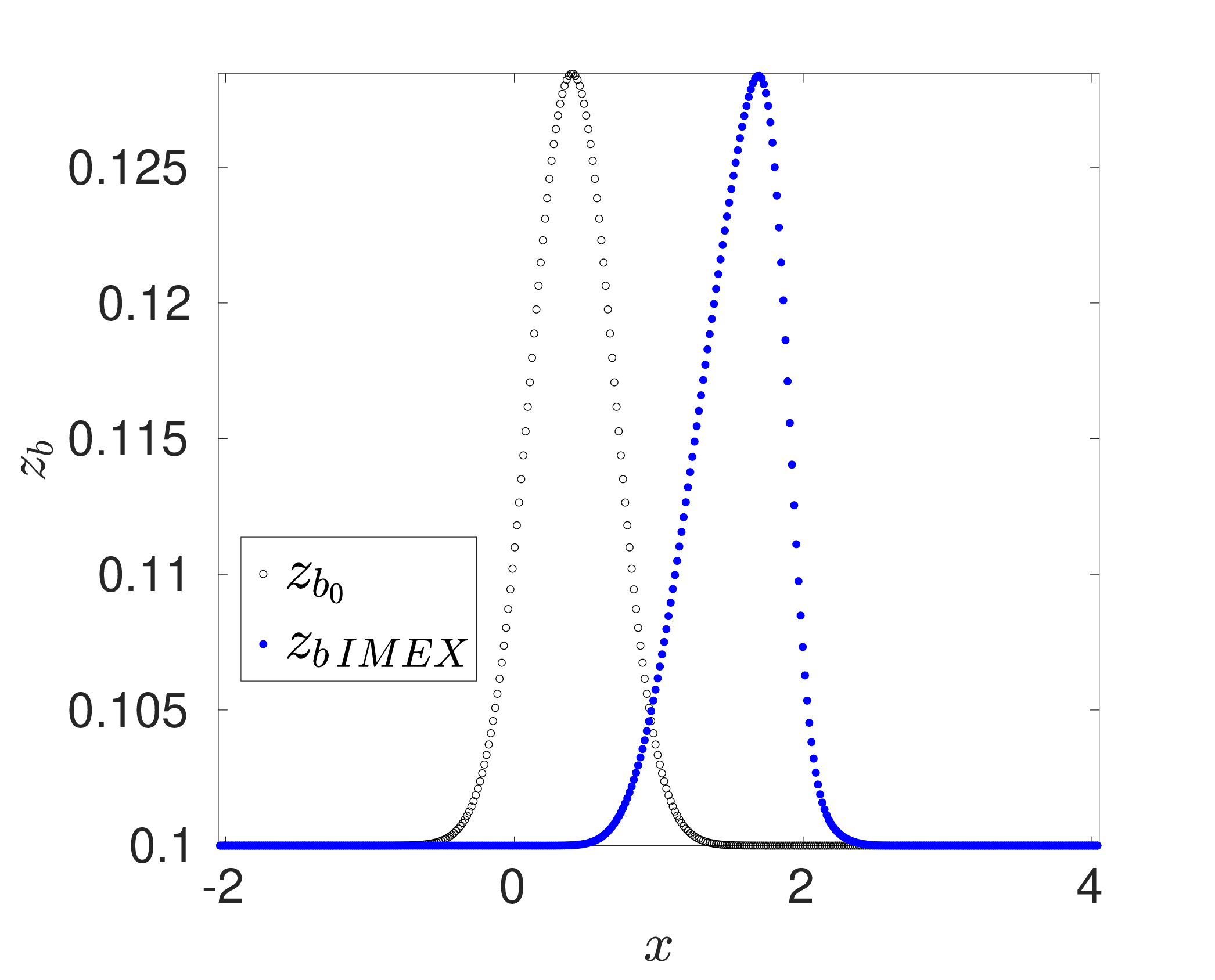}
         \includegraphics[width=0.32\textwidth]{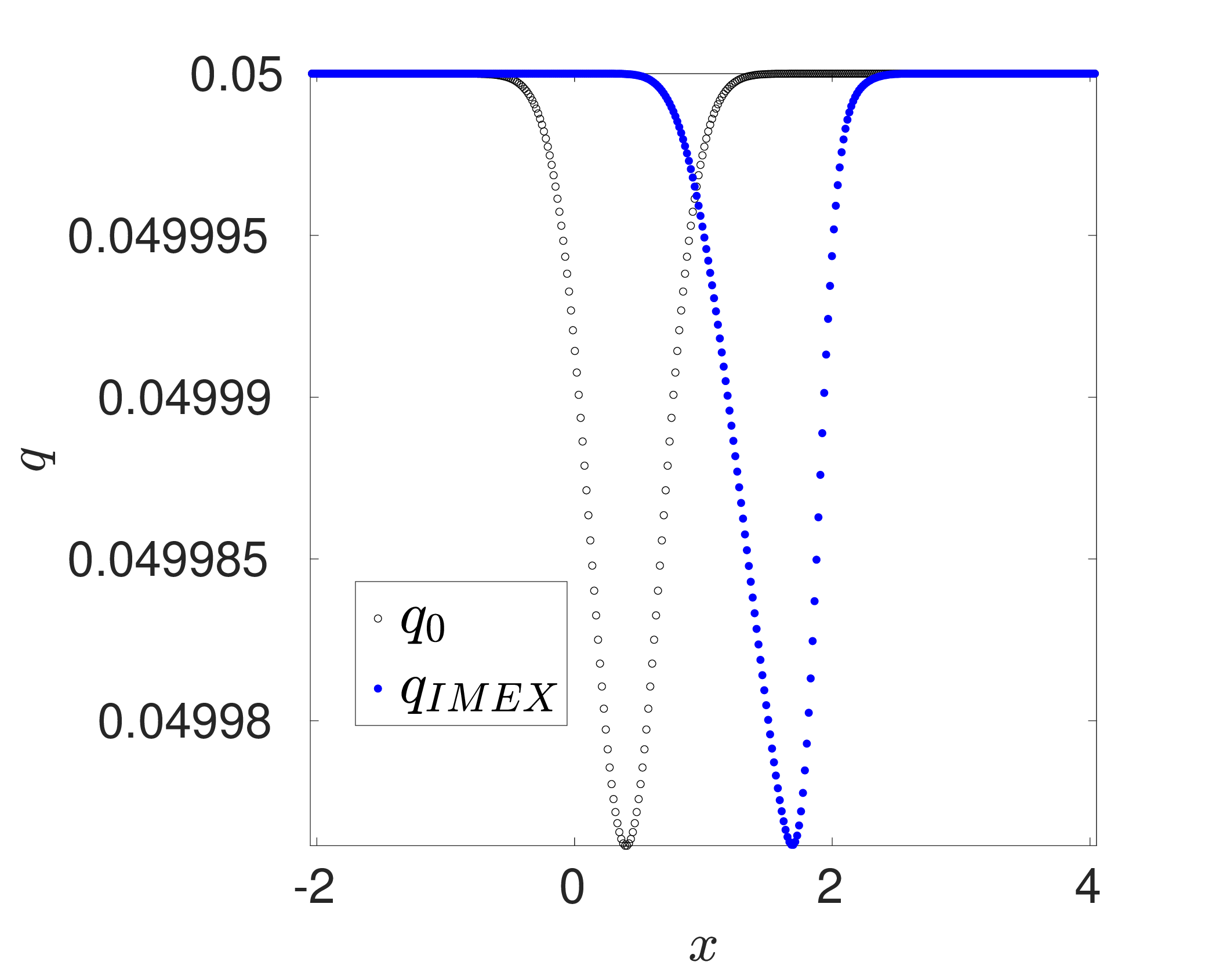}         
     \caption{Test  \ref{sssec:SA:long}: Sediment evolution for Exner model with flat bottom topography and non-flat sediment layer. Numerical solutions for $\eta$ (left), $z_b$ (center) and $q$ (right),  obtained with a third-order semi-implicit staggered scheme adopting a $400$ uniform mesh and CFL$_{\rm IMEX}=9.2$ at time $t=1400\ s$.}
     \label{Exp:SA:test_long_eta_q}
\end{figure}
% -----------
\begin{figure}[!ht]\centering
     \includegraphics[width=0.8\textwidth]{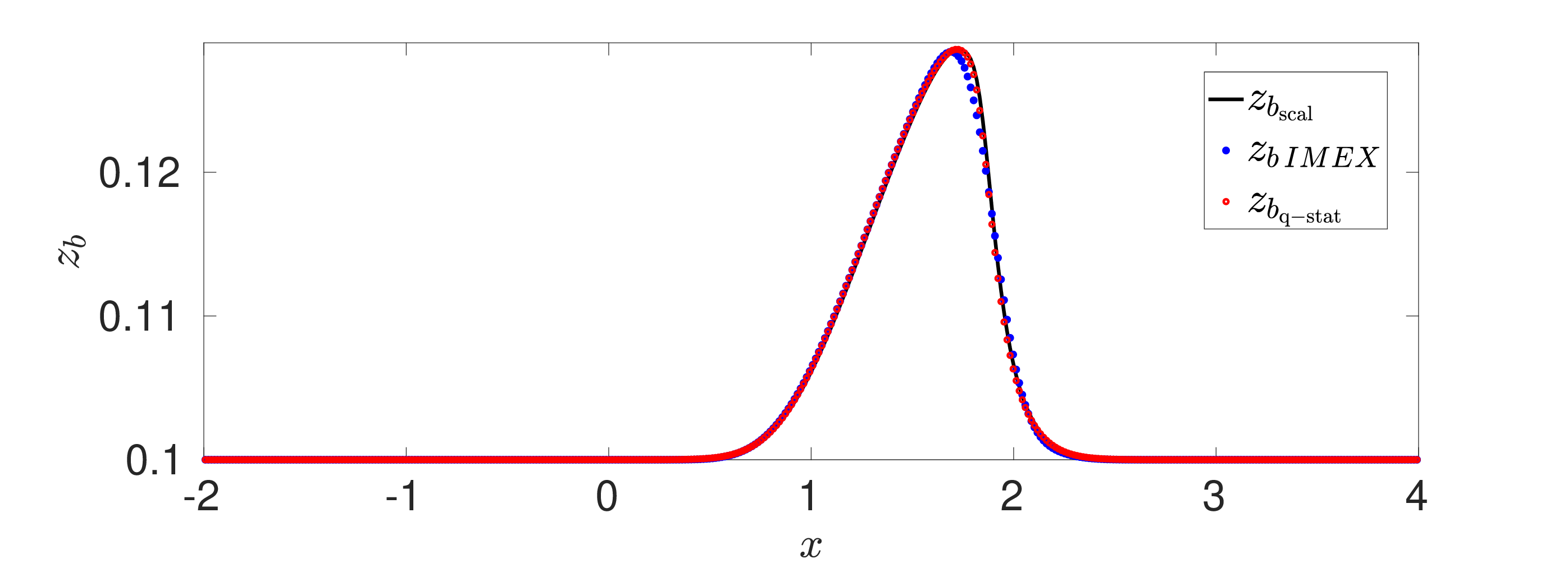} %\\ 
     \caption{Test \ref{sssec:SA:long}: Sediment evolution for Exner model with flat bottom topography and non-flat sediment layer, and initial condition \eqref{u0}. Comparison between $z_b$ computed with the third-order semi-implicit scheme, CAT2 applied to \eqref{Scal_eq} and second order method applied to \eqref{svesemianalytical}  adopting a $400$ uniform mesh and CFL$_{\rm IMEX}=8.7$, CFL$_{\rm q-stat}=0.45$ and CFL$_{\rm q-scal}=0.9$ at time $t=1400\ s$.}
     \label{Exp:SA:z_b}
\end{figure}
% -------------
\begin{figure}[!ht]\centering
     %\includegraphics[width=0.8\textwidth]{Figures/Experiments/Semi_Analyt/Scalar/m_Exn_zb_long_comparison_2.eps}
     %\\ 
     \includegraphics[width = 0.8\textwidth]{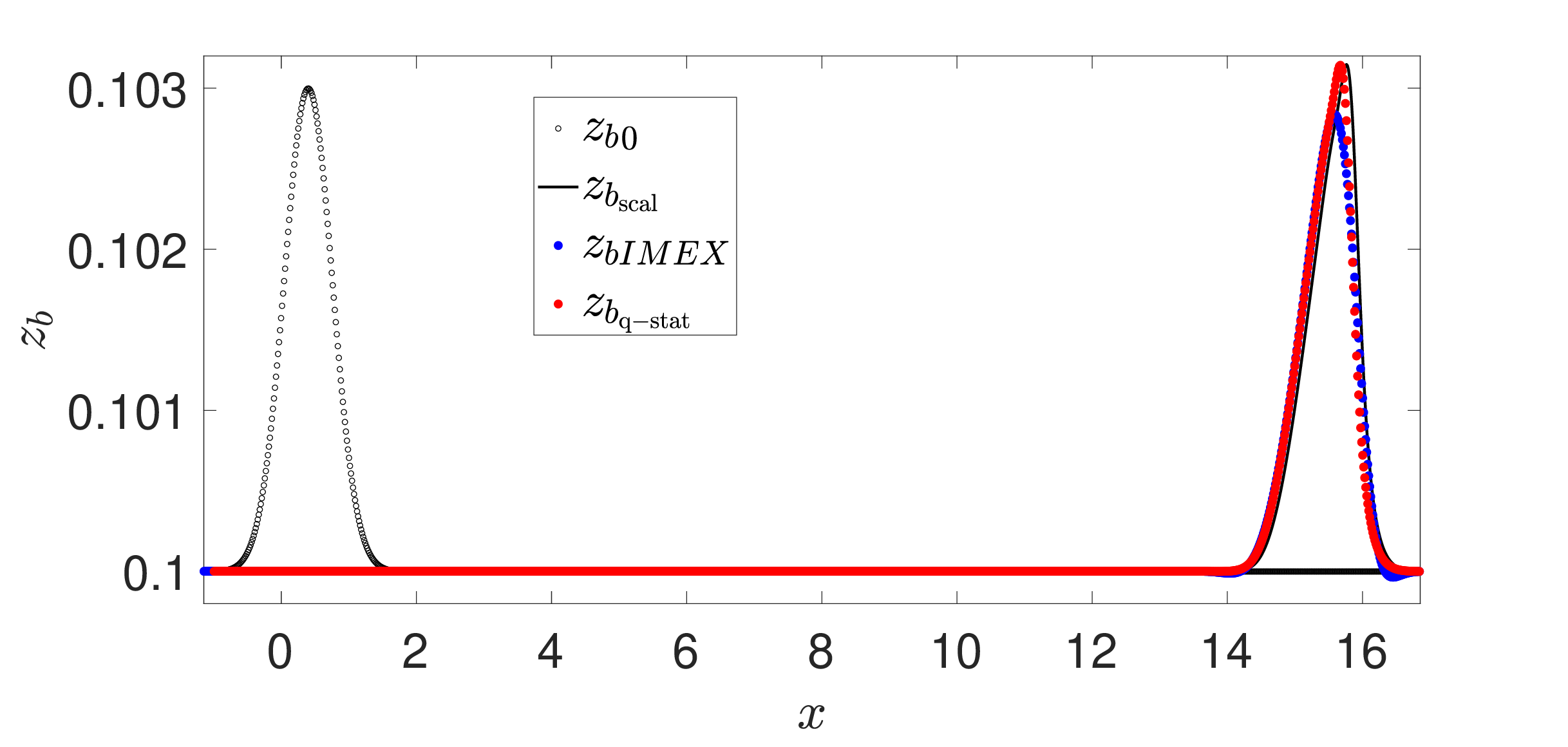}
     \caption{Test \ref{sssec:SA:long}: Sediment evolution for Exner model with flat bottom topography and non-flat sediment layer, and initial condition \eqref{u0_1}. Comparison between $z_b$ computed with the third-order semi-implicit scheme, CAT2 applied to \eqref{Scal_eq} and second order method applied to \eqref{svesemianalytical}  adopting a $1200$ uniform mesh and CFL$_{\rm IMEX}=9.2$, CFL$_{q-stat}=0.45$ and CFL$_{q-scal}=0.9$ at time $t=20000\ s$.}
     \label{Exp:SA:z_b_2}
\end{figure}
% -----------
\begin{figure}[!ht]
         \centering
         \includegraphics[width=0.49\textwidth]{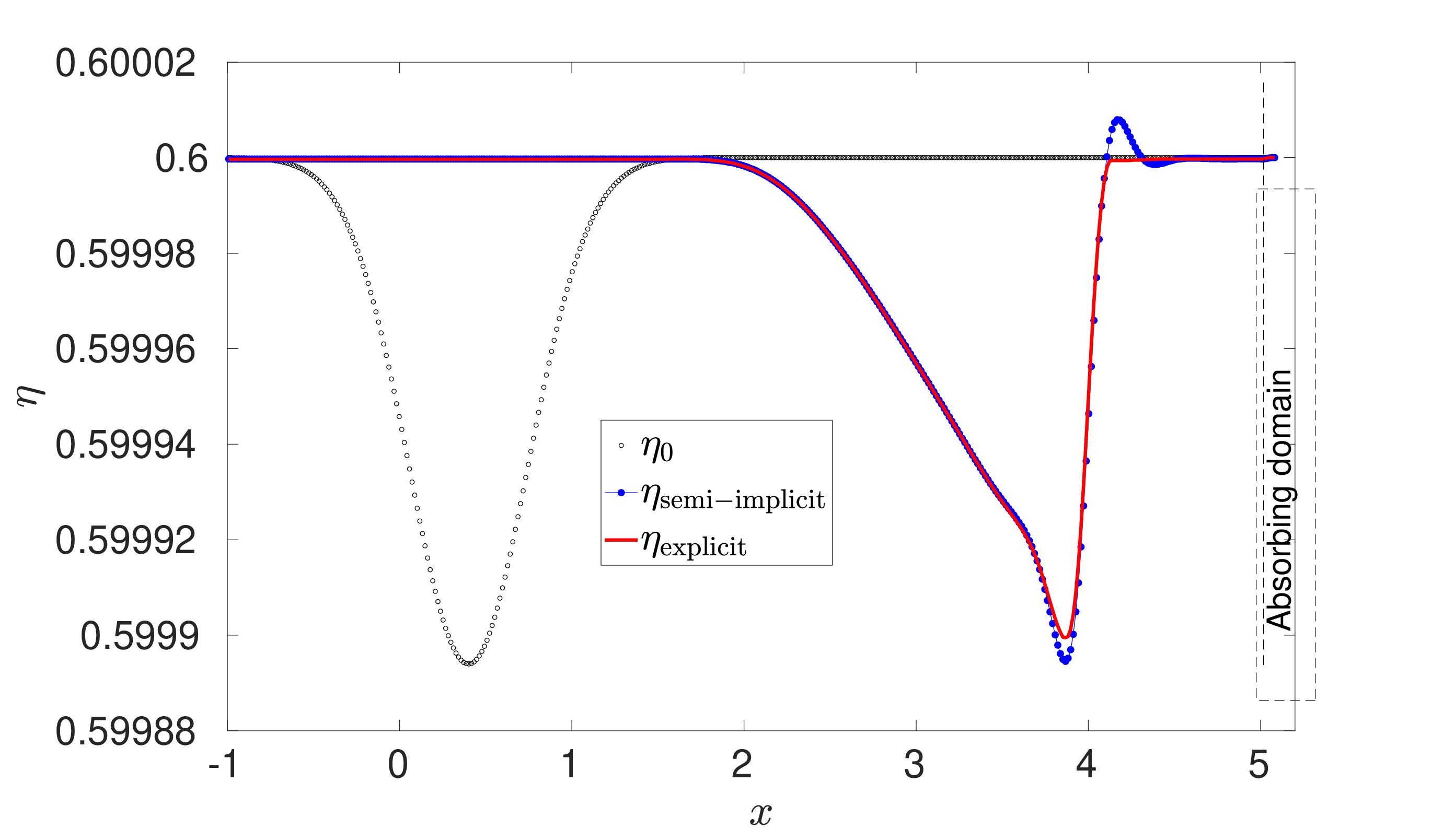}
         \includegraphics[width=0.49\textwidth]{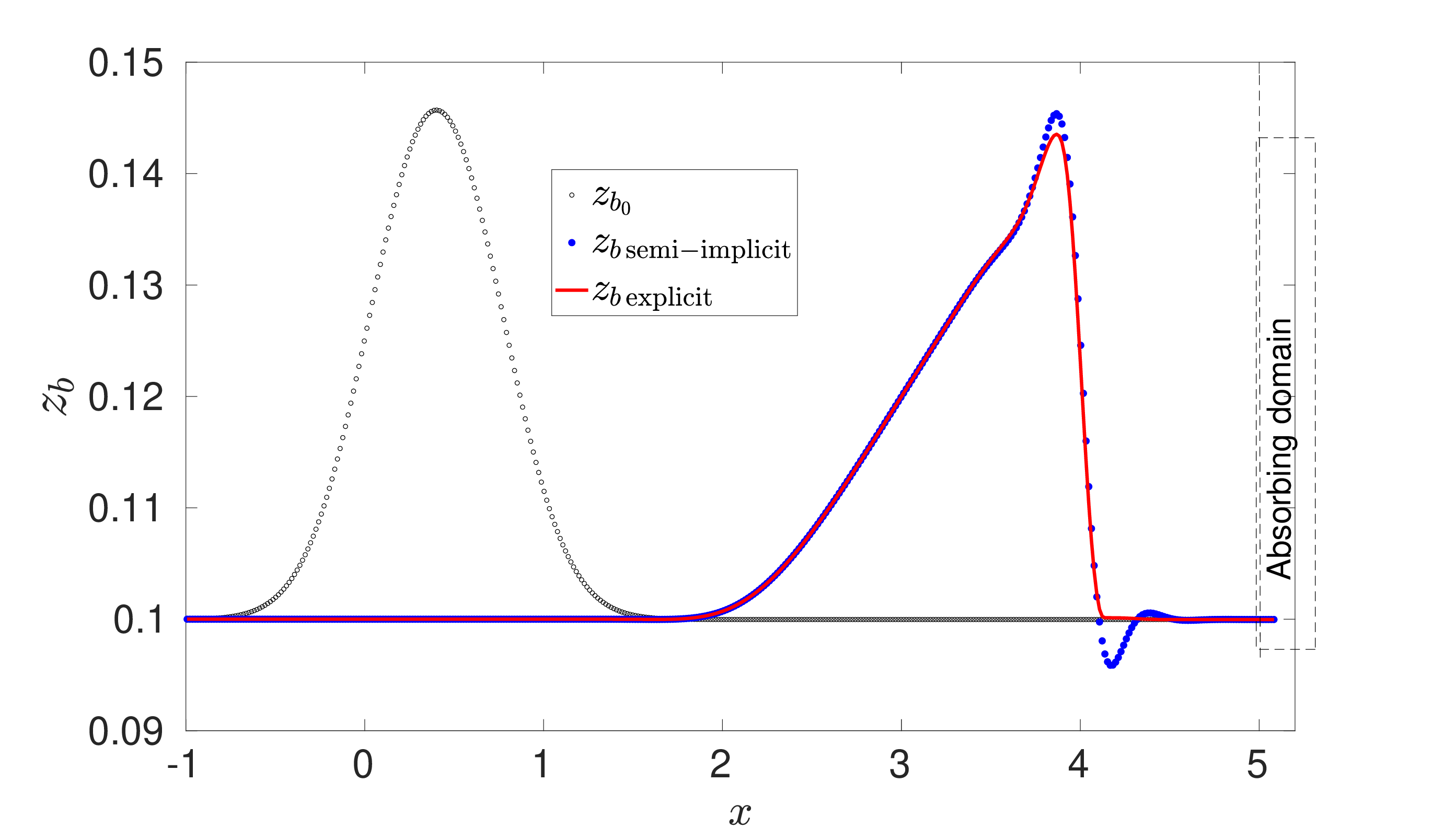}    
         \caption{Test  \ref{sssec:SA:long}: Sediment evolution for Exner model with flat bottom topography and non-flat sediment layer. Numerical solutions for $\eta$ (left), $z_b$ (right),  obtained with a third-order semi-implicit staggered scheme adopting a $400$ uniform mesh and CFL$_{\rm IMEX}=9.49$ at time $t=3500\ s$.
         }
     \label{Exp:SA:test_long_shock}
\end{figure}
% --------------
%

The common settings of this experiment are: $[x_a,x_b]=[-1,4]$ the interval; $A_g = 0.1;$ $\xi = \frac{1}{1-\rho_0},$ with $\rho_0=0.2;$ $m_g=3;$ $t_{end}= 1400\ s;$ and, initial conditions are so set: $b(x) \equiv 0,$ $h_0(x) = 0.5,$    
\begin{equation}
    \label{u0} 
    u_0(x) = 0.1 + 6\times10^{-3}e^{-\frac{(x-0.4)^2}{0.5^2}}
\end{equation}
and $z_b(x_a) = 0.1.$

Given $u_0(x)$, all variables can be derived as in Section~\ref{App_A_2}. In particular, the constant $C$ is computed as $C = G(u_0(a)) + g(h_0(a) + z_b(a) + b(a))$ where $G$ is a solution of \eqref{G'} and $Q$ is obtained through $Q = q_0(a) + q_b(a).$ Then, free boundary conditions are imposed at boundaries.

Figure~\ref{Exp:SA:test_long_eta_q} shows the numerical results for $\eta$ and $q$ obtained with the third-order semi-implicit staggered scheme at final time $t = 1400\ s$ on a uniform mesh with 400 points and CFL$_{\rm IMEX}=9.2.$ Figure~\ref{Exp:SA:z_b} exhibits the sediment evolution obtained with the semi-implicit scheme, the second order CAT2 method applied to reduced equation \eqref{svesemianalytical} and second order scheme applied to quasi-simple wave equation \eqref{Scal_eq} with different CFL conditions CFL$_{\rm IMEX}=9.2$, CFL$_{\rm q-stat}=0.45$ and CFL$_{\rm q-scal}=0.9$ at time $t=1400\ s$. It particular, the semi-analytical numerical solutions of the approximate partial differential equations are in agreement with the solution of the full Exner system \eqref{Ex_sis}. 

Nonetheless, it is worth mentioning that similar results can be obtained when considering long-time simulations, specifically with $t_{\rm end} = 20000\ s$ seconds on the interval $[x_a,x_b]=[-1,17]$, under similarly settings except for $ u_0(x)$ that in this case is given by 
\begin{equation}
    \label{u0_1} 
    u_0(x) = 0.1 + 6\times10^{-4}e^{-\frac{(x-0.4)^2}{0.5^2}}.
\end{equation}
These results are illustrated in Figure~\ref{Exp:SA:z_b_2}.

Finally, given that the semi-implicit scheme is not inherently shock-capturing, it is important to highlight its oscillatory behavior near discontinuities, particularly after the dune motion generates a significant jump. This phenomenon is evident when considering a similar test in the interval $[x_a, x_b] = [-1, 5]$, with $t_{\rm end} = 3500\ s$, CFL$_{\rm IMEX} = 9.49$ and CFL$_{\textrm{EXP}} = 0.45$ \cite{Loubere2024}.
%{\color{red}the following settings: the interval $[x_a, x_b] = [-1, 5]$; $A_g = 0.1$; $\xi = \frac{1}{1-\rho_0}$ with $\rho_0 = 0.2$; $m_g = 3$; and $t_{\rm end} = 3500\ s$ with CFL$_{\rm IMEX} = 9.49$ and CFL$_{\textrm{EXP}} = 0.45$ \cite{Loubere2024}.} 
The initial conditions are set as $b(x) \equiv 0$, $h_0(x) = 0.5$,  
\begin{equation}
    \label{u0_2} 
    u_0(x) = 0.1 + 10^{-2}e^{-\frac{(x-0.4)^2}{0.5^2}},
\end{equation}
and $z_b(x_a) = 0.1$. The obtained results for the free surface and sediment layer are in Figure~\ref{Exp:SA:test_long_shock}, where we see small oscillations near the shock.
The oscillation decreases when using smaller CFL numbers. As already mentioned, this problem would be probably overcome by adopting the implicit approach in \cite{Quinpi}.

\subsection{Semi-Analytical ODE-type}\label{sssec:SA:Ode}
In this section, the numerical solution computed with the third-order semi-implicit scheme has been compared with the semi-analytical solution ODE-type introduced in Subsection~\ref{subsec:semi-analitical-sol-SVE}. In particular, fixing $A_g = 0.001,$ $\xi = 1,$ and initial condition 
\begin{equation}
    \label{semi_analy_IC}
    W_0(x) = \begin{bmatrix}
        h_0 \\ h_0u_0(x) \\ z_{b_0}(x)
    \end{bmatrix} 
\end{equation}
in which $h_0\equiv 10;$ $u_0(x) = 0.01(x+1)$ and $z_{b_0} = 0.03A_gu_0(x)^2.$ The computational domain $[x_a,x_b] = [0,400]$ while the absorbing domain $[x_b,x_c] = [400,450]$ has been adopted to avoid reflective phenomena to the right boundary condition. 
%\begin{figure}[!ht]
%     \centering
%         \includegraphics[width=0.49\textwidth]{Figures/Experiments/Semi_Analyt/ODE/m_Test_1_h.eps}
%         \includegraphics[width=0.49\textwidth]{Figures/Experiments/Semi_Analyt/ODE/m_Test_1_u.eps}
%     \caption{Test \ref{sssec:SA:Ode}: Semi-Analytical ODE-type. Exact ODE-type solution and numerical one for $h$ (left) and $u$ (right  obtained with a third-order semi-implicit staggered scheme  adopting a $200$ uniform mesh and CFL$_{\rm IMEX}=1.87$ at time $t=2000$ $s$.}
%     \label{Exp:SA:test_ode_h_u}
%\end{figure}
\begin{figure}[!ht]
	\centering
	\includegraphics[width=0.45\textwidth]{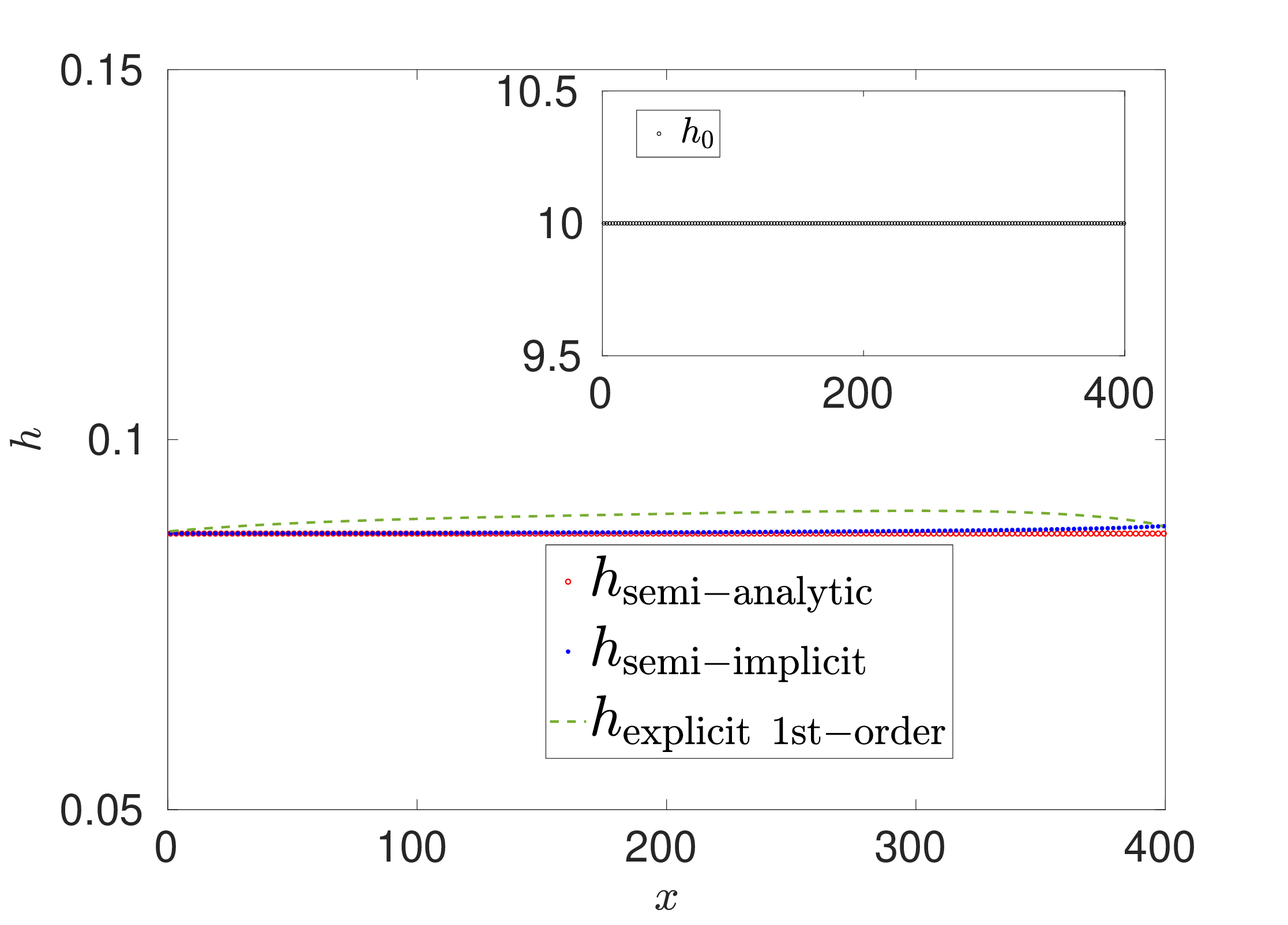}
	\qquad 
	\includegraphics[width=0.45\textwidth]{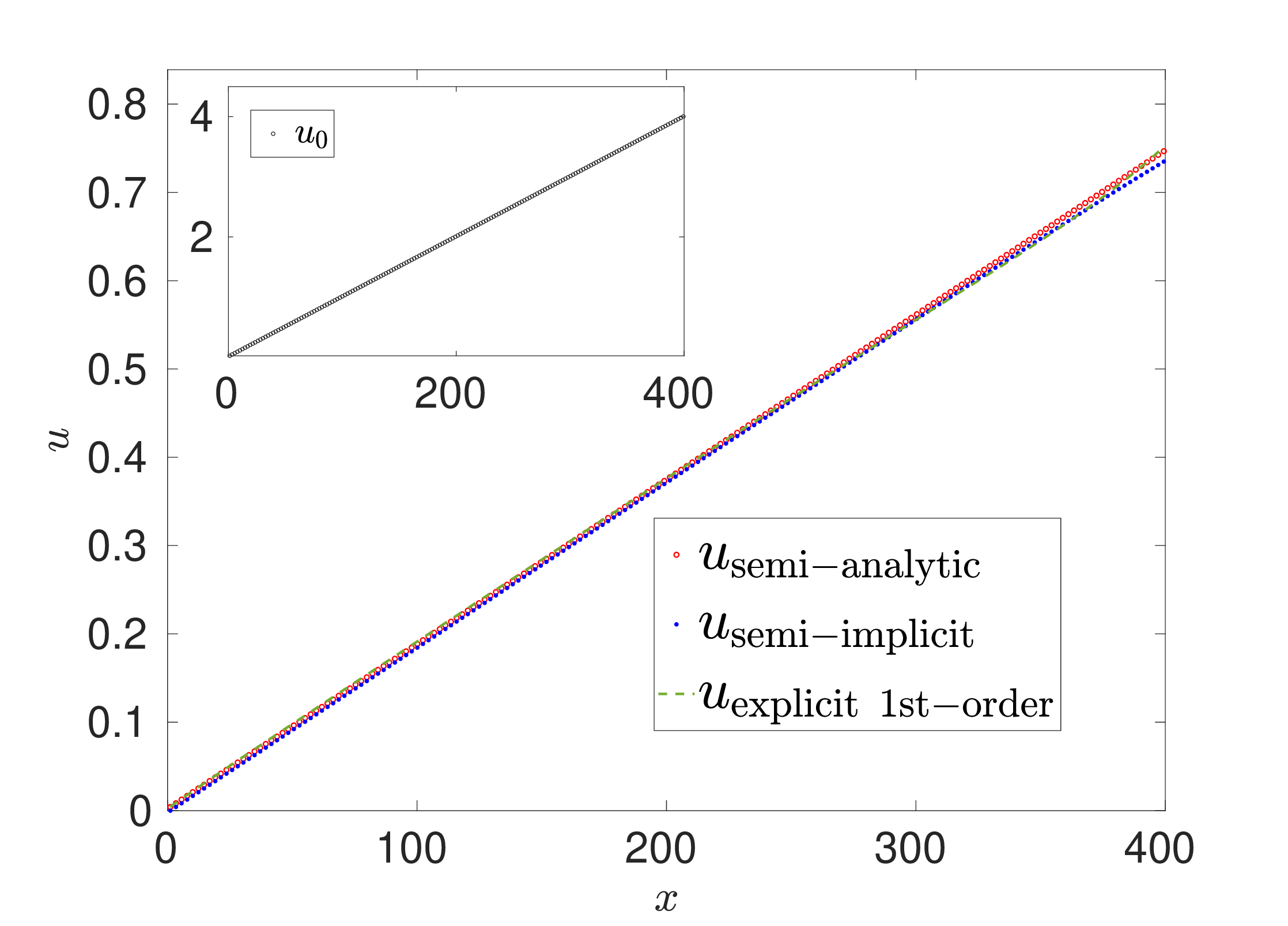}
	\caption{Test \ref{sssec:SA:Ode}: Semi-Analytical ODE-type. Exact ODE-type solution and numerical one for $h$ (left) and $u$ (right)  obtained with a third-order semi-implicit staggered scheme adopting a $200$ uniform mesh and CFL$_{\rm IMEX}=1.87$ at time $t=2000\ s$.}
	\label{Exp:SA:test_ode_h_u}
\end{figure}
\begin{figure}[!ht]\centering
     \includegraphics[width=0.8\textwidth]{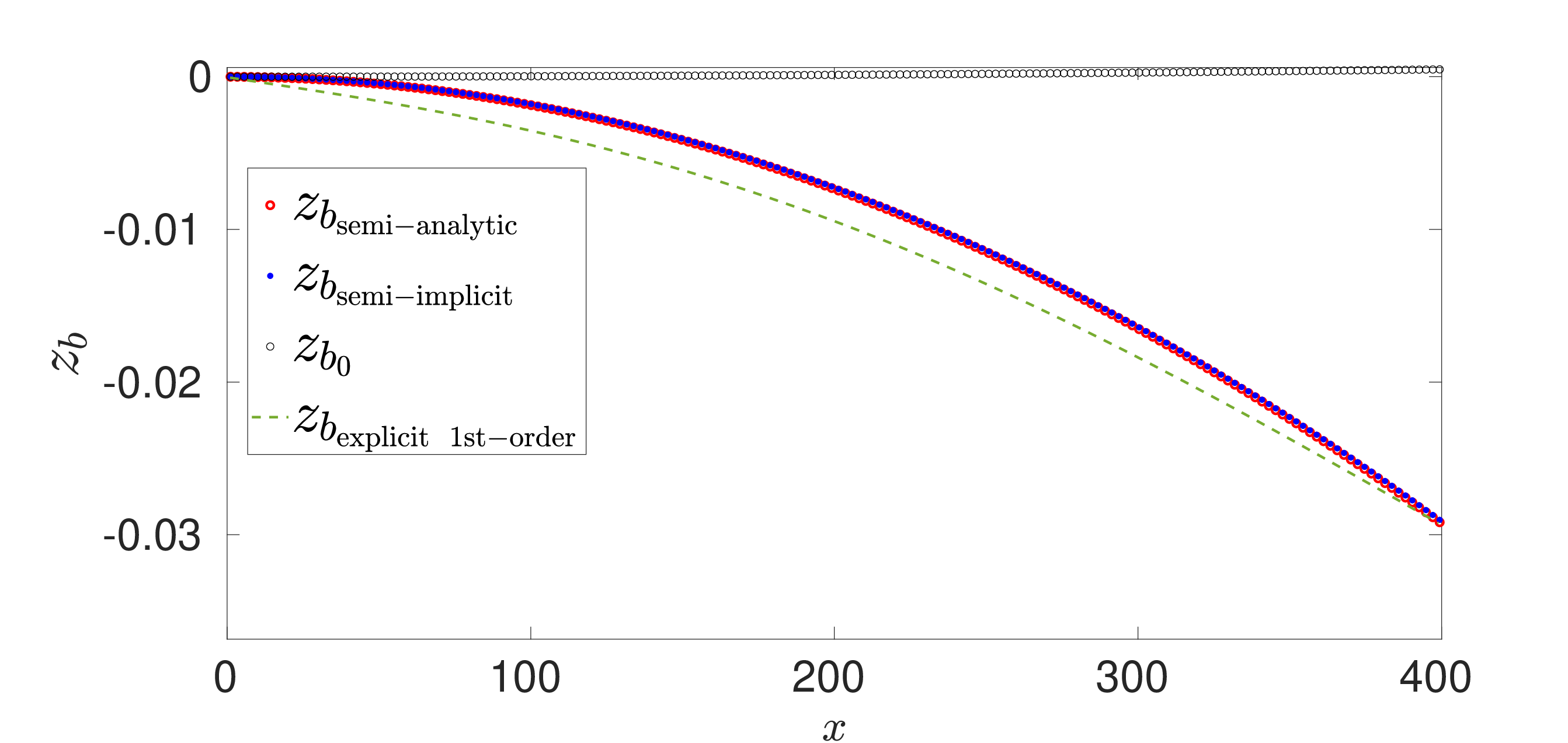}
     \caption{Test \ref{sssec:SA:Ode}. Semi-Analytical ODE-type. Exact ODE-type solution and numerical one for $z_b$  obtained with a third-order semi-implicit staggered scheme  adopting a $200$ uniform mesh and CFL$_{\rm IMEX}=1.87$ at time $t=2000\ s$.}
     \label{Exp:SA:ODE_z_b}
\end{figure}
% -------------
% 

Figures~\ref{Exp:SA:test_ode_h_u}-\ref{Exp:SA:ODE_z_b} show the initial conditions, the numerical solutions of the semi-analytic ODE-type system and the solution of the third-order semi-implicit scheme for $h,$ $u$ and $z_b$, which are computed on the domain $[x_a,x_b]=[0,400]$ with 200 uniform cells at final time $t = 2000\ s$. In this case, we use CFL$=1.87$ such that the MCFL$=0.4$. We can observe how the numerical solutions give a good approximation of the semi-analytical solution, where only some small differences can be appreciated near the end of the domain after $2000\ s$ of simulation. In this case the reference solution has been computed with Matlab \emph{ode45} solver. Let us remark that this is a hard test since it is a long domain and a very slow process, so we need a very long time of simulation. Actually, the erosion process at the end of the domain is 3 $cm$ approximately while the longitudinal domain is 400 $m$. In this figures, we also show the results obtained using an first-order (time-space) explicit method with the Rusanov flux also with 200 cells. We see how the third-order method notably improves the results for this test.

\subsection{Exner waves}\label{sssec:Ex:Waves}
Following a quantitative and qualitative analysis of the third-order semi-implicit numerical scheme, it remains to assess its ability to accurately capture sediment evolution, despite introducing errors by failing to resolve surface waves, which require sufficiently refined spatial and temporal discretization to be well-captured. In this regard, building on the approach of \cite{MaccaRussoBumi}, surface waves are generated to evaluate the scheme’s accuracy. The procedure followed is as outlined below:
%-------------
\begin{figure}[!ht]
     \centering
     \centering
     \includegraphics[width=0.49\textwidth]{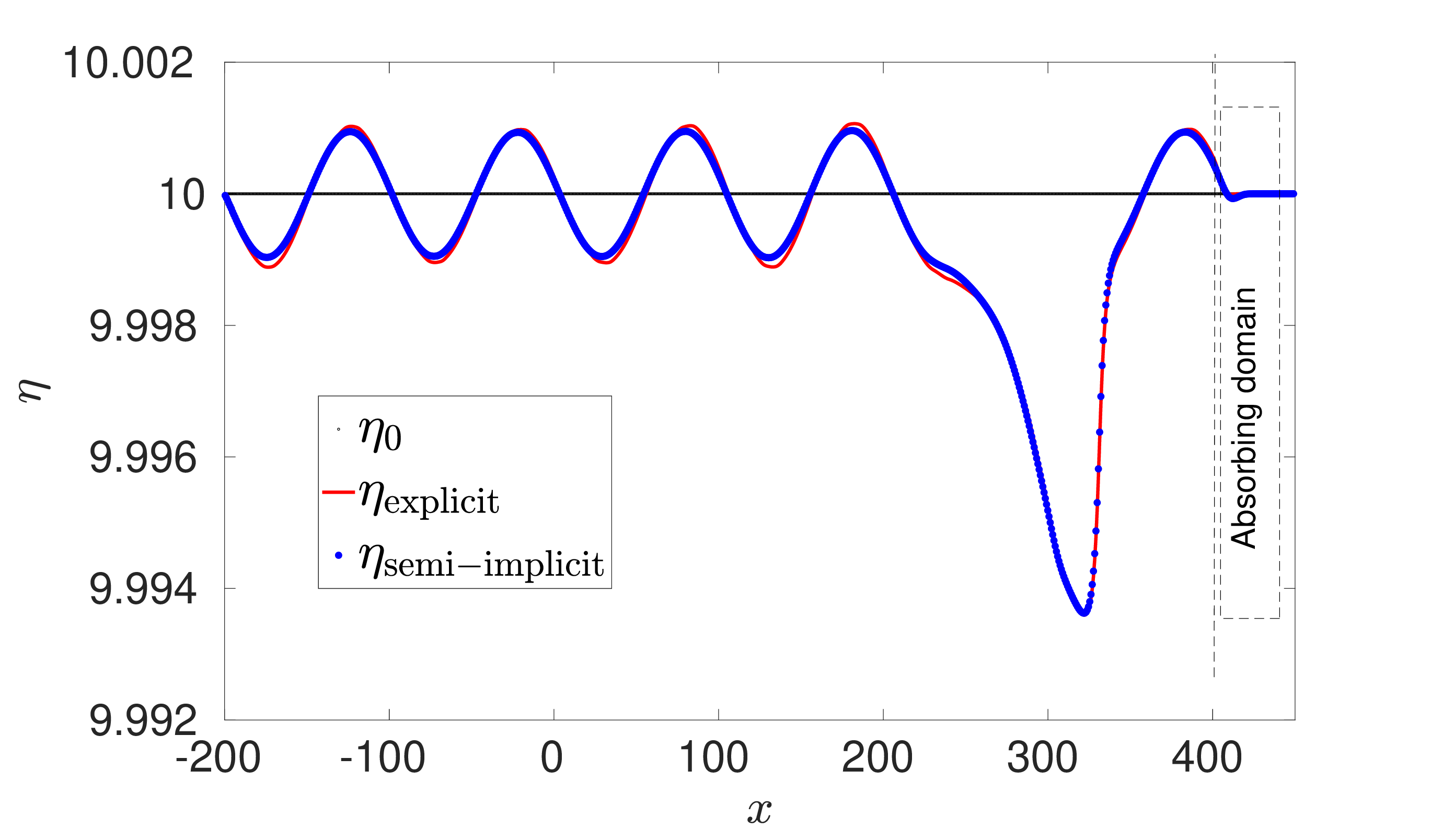}
      \includegraphics[width=0.49\textwidth]{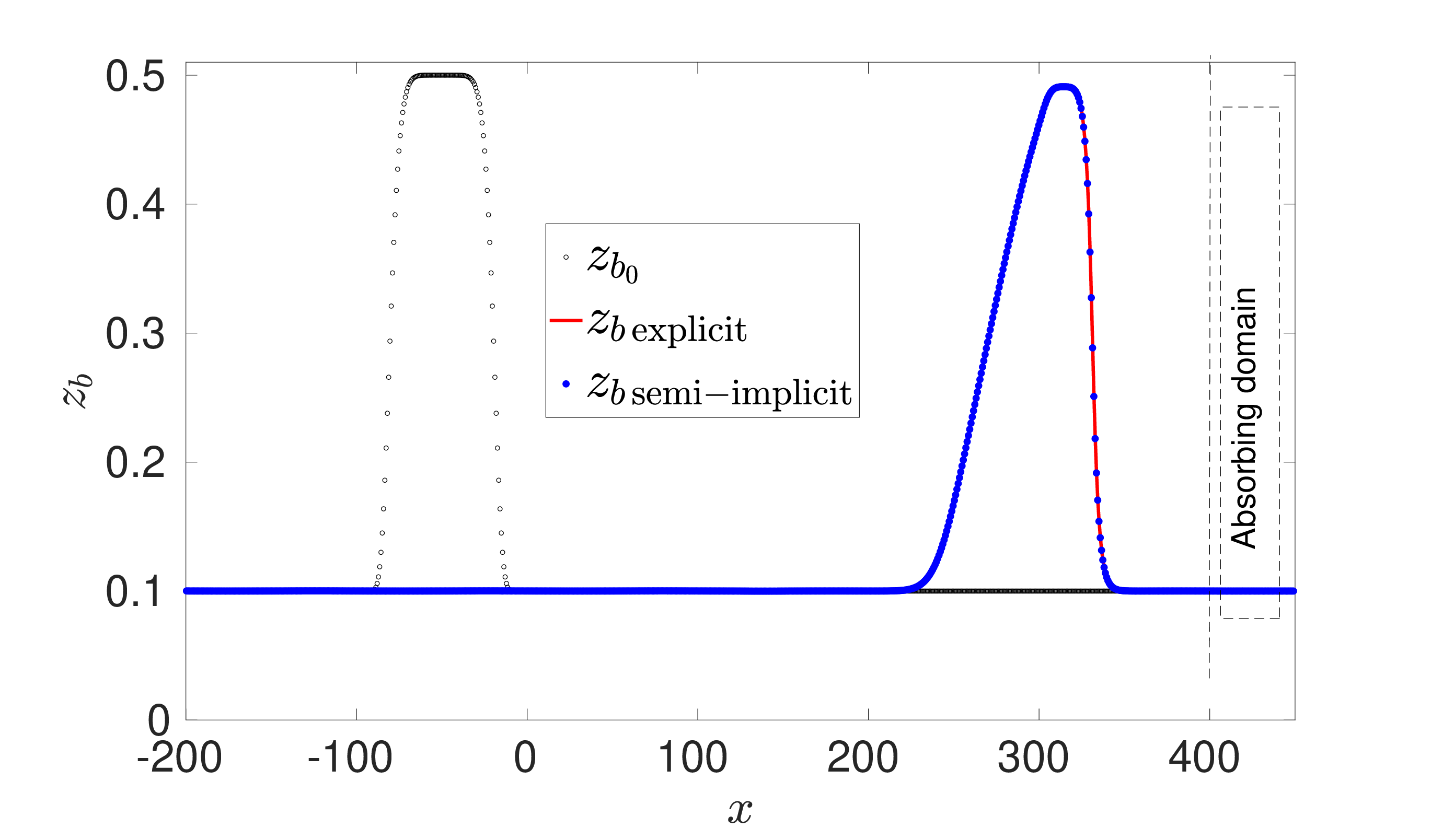}
      \caption{Test \ref{sssec:Ex:Waves}. Numerical solutions for $\eta$ and $z_b$ obtained with third-order semi-implicit and explicit scheme on a $800$ uniform mesh at final time $t_{\rm final}=5000\ s;$ where CFL$_{\rm IMEX} = 6$ and such that the MCFL$\leq0.4$; and CFL$_{\rm explicit} = 0.45$. An additional absorbing domain $[x_b,x_c] = [400,450]$ has been adopted to avoid wave reflections due to the boundary condition. The waves angular frequency $\omega$ is $0.7$ s$^{-1}$. }
     \label{Exp:Waves:low}
\end{figure}
%-------------
\begin{itemize}
    \item Initially, the explicit and semi-implicit schemes were compared both qualitatively and quantitatively, with a wave frequency set such that the spatial and temporal discretization was sufficient to correctly capture the waves, even if in the semi-implicit time discretization the CFL is larger than 1;
    \item Subsequently, the wave frequency was increased, and numerical solutions were computed for both the explicit and semi-implicit schemes. For the explicit solution, the space-time discretization ensured that all waves were perfectly captured adopting a sufficiently fine-mesh, whereas the semi-implicit solution was obtained using a coarse discretization in space and time.
\end{itemize}
These choices were made to rigorously evaluate both the accuracy and computational efficiency of the semi-implicit scheme.

%-------------
\begin{figure}[!ht]
     \centering
     \begin{subfigure}[b]{0.48\textwidth}
         \centering
         \includegraphics[width=\textwidth]{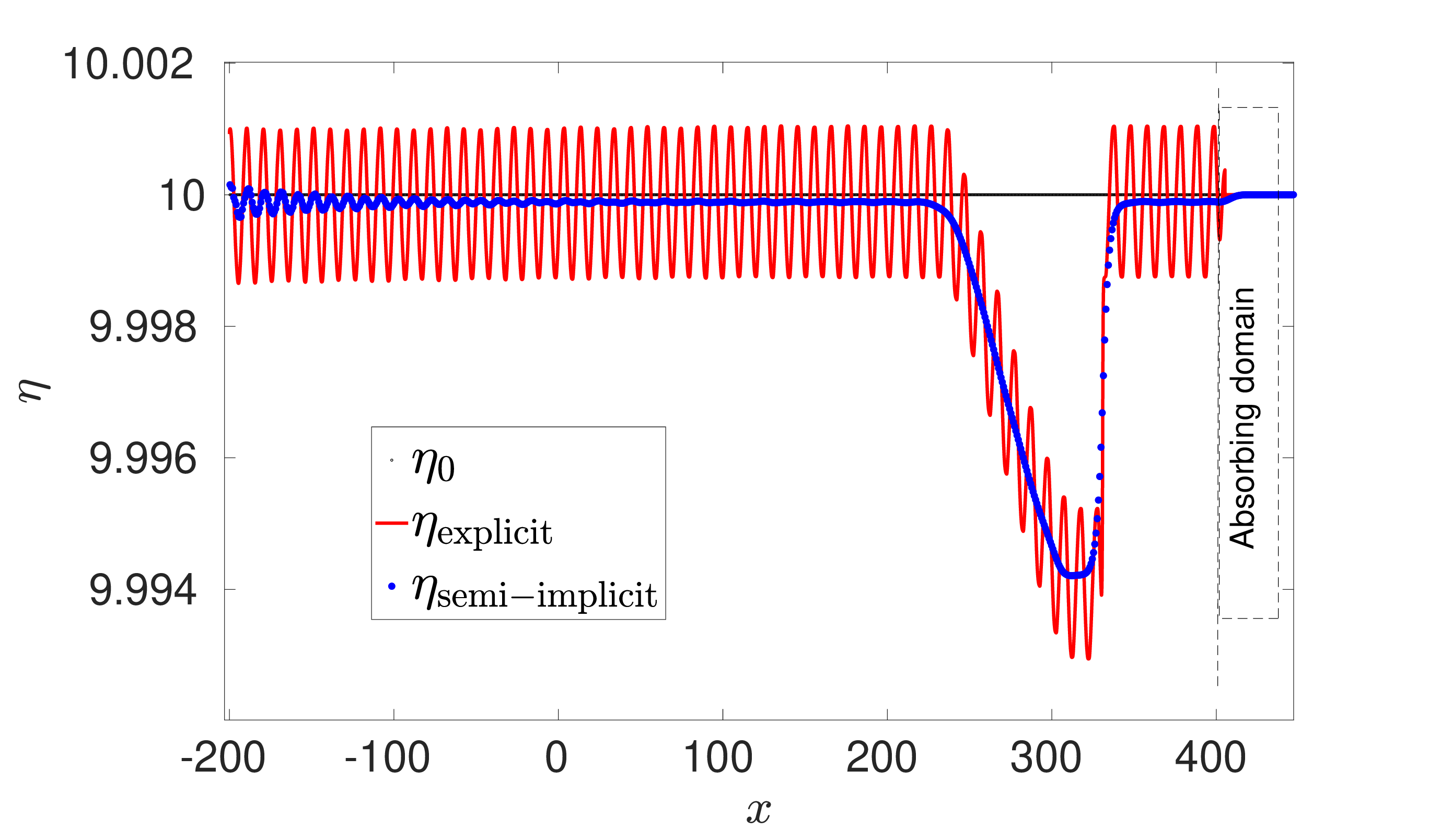}
         \caption{$\eta$}
         \label{Exp:Waves:eta:high}
     \end{subfigure}
     \hfill
     \begin{subfigure}[b]{0.48\textwidth}
         \centering
         \includegraphics[width=\textwidth]{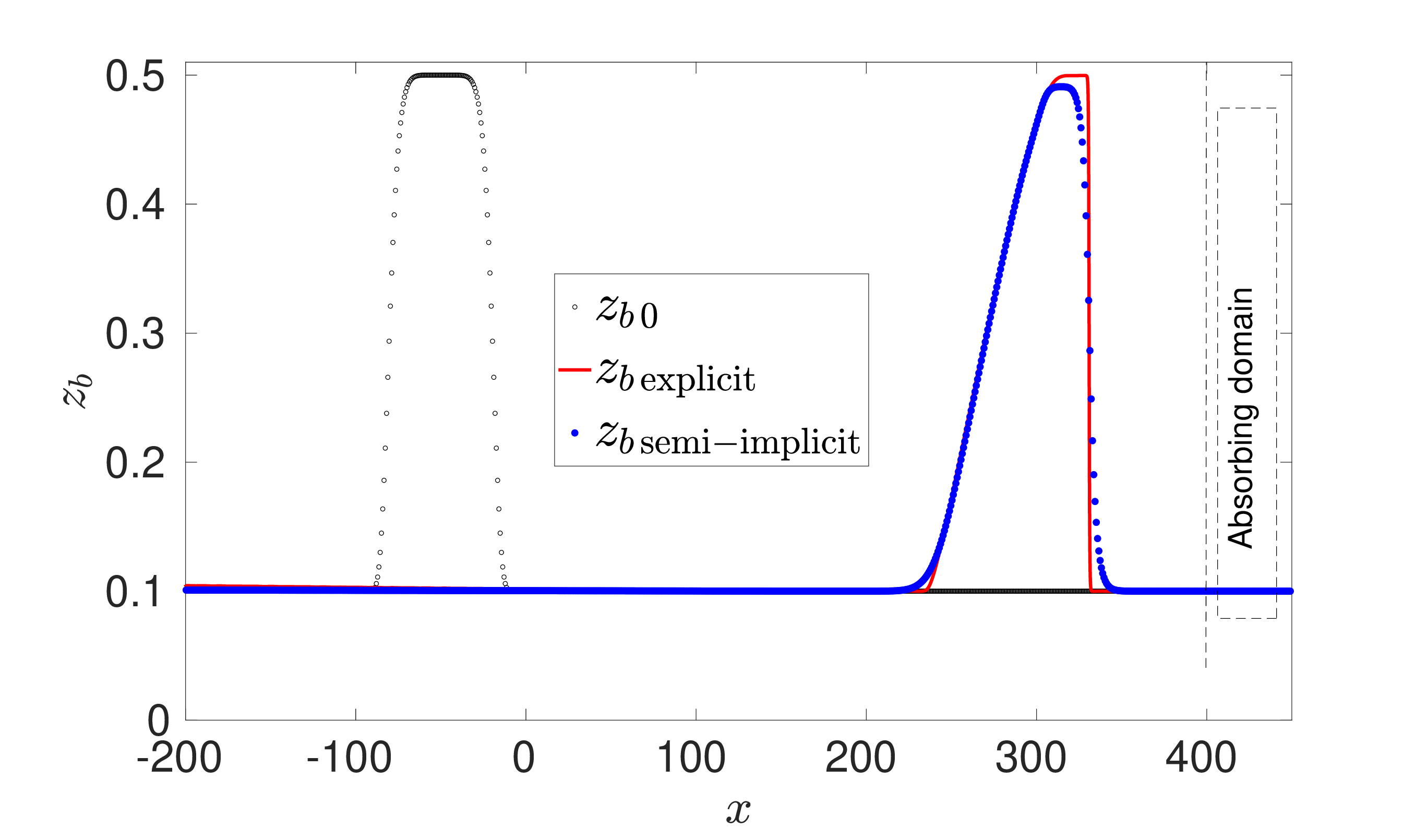}
         \caption{$z_b$}
         \label{Exp:Waves:zb:high}
     \end{subfigure}
     \caption{Test \ref{sssec:Ex:Waves}. Numerical solutions for $\eta$ and $z_b$ obtained with third-order semi-implicit and explicit scheme on a, respectively, $800$ and $10000$ uniform mesh at final time $t_{\rm final}=5000\ s;$ where CFL$_{\rm IMEX} = 6$ and such that the MCFL$=0.4$; and CFL$_{\rm EXPLICIT} = 0.45$. An additional absorbing domain $[x_b,x_c] = [400,450]$ has been adopted to avoid wave reflections due to the boundary condition. The waves frequency is $7$  s$^{-1}$.  %\giovanni{Questo esempio è il migliore dell'articolo!} \ema{Sono pienamente d'accordo.}
     }    \label{Exp:Waves:high}
\end{figure}
%--------------

For this purpose, the conditions of the test are defined as follows: the spatial domain is set as $[x_a,x_b] = [-200, 400]$ the absorption domain $[x_b,x_c] = [400,450]$ to prevent wave reflection due to the boundary condition; the CFL number at the initial condition is fixed to $4$; a flat bottom topography is considered with $b(x) \equiv 0.01 \, \text{m}$; and the parameters are $A_g = 0.1$, $m_g = 3$, and $\xi = 1/(1 - \rho_0)$ with $\rho_0 = 0.2$. The initial condition is given by:

\begin{equation}
    \label{IC_waves} 
    W_0(x) = \begin{bmatrix}
    	\eta_0(x) \\ q_0(x) \\ z_{b_0}(x)
    	\end{bmatrix}
    =
    \begin{bmatrix}
        10 \, \text{m} \\ 
        10 \, \text{m}^2/\text{s} \\ 
        0.1 + 2e^{-x^8/10^{14}} \, \text{m}
    \end{bmatrix}.
\end{equation}
The boundary conditions are imposed as:
\begin{equation}
    \label{BC_waves}
    \begin{bmatrix}
    u(x_a,t) \\
    h(x_a,t)
    \end{bmatrix} = 
    \begin{bmatrix}
    u_0 + \mathcal{A} \sin(x_a - \omega t) \\
    \frac{q(x_a,0) + q_b(x_a,0)}{u(x_a,t)} - \xi A_g u(x_a,t)^{m_g-1}
    \end{bmatrix},
\end{equation}
where $\mathcal{A} = 0.001$ and $\omega\in\{0.7,7\}.$

Figures~\ref{Exp:Waves:low}-\ref{Exp:Waves:high} present the numerical solutions for the sediment and free-surface evolution obtained using the third-order explicit and semi-implicit schemes, with wave angular frequencies $\omega= 0.7$ s$^{-1}$, and $\omega=7$ s$^{-1}$, respectively. It is important to highlight that, $\omega = 0.7$ s$^{-1}$, a mesh with 800 nodes is sufficient to accurately capture the surface waves, whereas when $\omega = 7$ s$^{-1}$, the explicit scheme adopts 10000 nodes to get a proper wave resolution. Nonetheless, the semi-implicit scheme, applied in both cases on a 800 uniform mesh, successfully captures the sediment evolution, despite being more dissipative compared to the explicit solution. Moreover, a CFL$_{\rm IMEX} = 6$ was adopted for the semi-implicit scheme, compared to 0.45 for the explicit one.

\section{Conclusions}\label{se:conclusions}
In this paper, a third-order semi-implicit scheme for the Saint Venant and Saint-Venant-Exner models has been introduced. To our knowledge, this is the first semi-implicit scheme following the technique of \cite{Casulli1992} with order $s>2$. A key ingredient is the combination of CWENO reconstructions and a centered third-order polynomial preserving the average of the free surface to implicitly discretize the free surface gradient in the momentum conservation equation, as explained in Subsection \ref{Sec:space_reconst}. After, a third-order IMEX time discretization is considered. The proposed scheme is well balanced for \emph{lake-at-rest} states, which has been proven both analytically and numerically. Absorbing boundary conditions to impose an outflow condition have also been used. The high performance of this method has been illustrated in the numerical tests section, where several test cases with/without sediment layer have been considered. Among others, an accuracy test have been performed. The results of the semi-implicit method have been also compared with a novel time-dependent semi-analytical solution introduced in Subsection \ref{subsec:semi-analitical-sol-SVE}, and the results of an explicit version of the proposed scheme. In the future, it would be interesting to apply this method to enhanced sediment models as those in \cite{GarresDiaz2022}, including also gravitational effects.

\section*{Acknowledgments}
This research has been partially supported by Junta de Andalucía research project ProyExcel\_00525 and by the European Union - NextGenerationEU program and by grant PID2022-137637NB-C22 funded by  MCIN/AEI/10.13039/501100011033 and ``ERDF A way of making Europe''.
This research has received funding from the European Union’s NextGenerationUE – Project: Centro Nazionale HPC, Big Data e Quantum Computing, “Spoke 1” (No. CUP E63C22001000006). E. Macca was partially supported by: GNCS No. CUP E53C23001670001 Research Project “Metodi numerici per le dinamiche incerte”; GNCS No. CUP E53C24001950001 Research Project "Soluzioni Innovative per Sistemi Complessi: Metodi Numerici e Approcci Multiscala"; PRIN 2022 PNRR “FIN4GEO: Forward and Inverse Numerical Modeling of hydrothermal systems in volcanic regions with application to geothermal energy exploitation”, No. P2022BNB97; PRIN 2022 “Efficient numerical schemes and optimal control methods for time-dependent partial differential equations”, No. 2022N9BM3N - Finanziato dall’Unione europea - Next Generation EU – CUP: E53D23005830006. . E. Macca and G. Russo would like to thank the Italian Ministry of Instruction, University and Research (MIUR) to support this research with funds coming from PRIN Project 2022  (2022KA3JBA, entitled “Advanced numerical methods for time dependent parametric partial differential equations and applications”). E. Macca and G. Russo are members of the INdAM Research group GNCS.

\appendix
%\bigskip 

\section*{\bf Appendix}
\section{Semi-analytical solution: Scalar PDE}\label{App_A}
In this section, we present two semi-analytical solutions for the Exner system with the Grass formula that are obtained by solving the scalar equation corresponding to the sediment evolution under the hypothesis of a steady flow surrounding the sediment. 

\subsection{Type 1}
A semi-analytical solution of the Exner model \eqref{Ex_sis} is derived by assuming a weak interaction between the fluid and the sediment layer. In the Grass model \eqref{sediment_equation}, this assumption holds when $A_g$ is a small number (e.g., $A_g = 10^{-3}$), implying that the morphodynamic timescale is much larger than the hydrodynamic timescale. A second assumption is that the water free surface and the water discharge are nearly constant. We remark that, although the Grass formula is adopted here, this approach is also valid for general definitions of the solid transport discharge (see \cite{GarresDiaz2022}).

Then, let $\eta_0,q_0\in\mathbb{R}$ denote the constant free surface level and the water discharge. Neglecting the effect of the friction term, the sediment layer of the Exner model \eqref{Ex_sis} can be computed as the solution of the scalar partial differential equation:
\begin{equation}\label{svesemianalytical}
\partial_t z_b + \xi A_g\partial_x \left(   \frac{q_0^{m_g}}{(\eta_0-z_b)^{m_g}}\right)=0.
\end{equation}
with initial condition $z(x,t) = z_0(x)$ a known function.

\subsection{Type 2} \label{App_A_2}
This case is more complex than the previous one. 
Here, quasi-stationary state, which is not neccessary constant, will be considered for the hydrodynamic variables. Again, we assume $A_g \ll 1$ (weak coupling), when the motion of the sediment occurs on a much longer timescale than that of surface waves. Consequently, surface waves propagate over a bathymetry defined by the bottom and the sediment, which is nearly constant over time. Thus, to detect the slow motion of the sediment, a reasonable approximation involves monitoring the sediment motion in a sequence of quasi-stationary states. This is achieved by setting the time derivative to zero in the first two equations of the Exner model \eqref{Ex_sis}. Our starting point is therefore to neglect the time derivative of $\eta$ and $q$ in the first two equations, corresponding to stationary flow when the sediment is immobile.
\begin{equation}
    \label{Start_point}
    \begin{cases}
        (q+q_b)_x = 0\\
        (qu)_x + gh(h + z_b + b)_x = 0\\
        (z_b)_t + (q_b)_x = 0
    \end{cases}
\end{equation}

We shall use the first two equations to express all quantities $\eta$, $q$, and $z_b$ in terms of $u$. From the first equation of \eqref{Start_point} we get:
\begin{equation}
    \label{q+q_b=Q}
    q + q_b = Q,
\end{equation} 
and assuming $u > 0$,
\begin{equation}
    \label{h}
    h = \frac{Q}{u} - \xi A_g u^{m_g-1},
\end{equation}
where, for simplicity a net flux is assumed, i.e., $Q > 0$.

Dividing the second equation by $h$, assuming $h$ is non-zero (as without water flux there would be no sediment flux, so $q$ and hence $h$ are non-zero), we obtain:
\begin{equation}
    \label{2_mod}
    \frac{u}{q}(qu)_x + g(h+z_b+b)_x=0.
\end{equation}

Let define $G(u)$ a solution of the following ODE 
\begin{equation}
    \label{G_x}
    \frac{\partial G}{\partial x} = \frac{u}{q}(qu)_x.
\end{equation}
Since $\frac{\partial G}{\partial x} = \frac{dG}{d u}u_x$, it follows that $\frac{u}{q}(qu)_x = G'u_x$, and therefore,
\begin{equation}
    G'(u)u_x = \frac{u}{q}(q'u + q)u_x \quad \Rightarrow \quad G'(u) = \frac{q'}{q}u^2 + u.
\end{equation}

As a consequence of the first equation in system \eqref{Start_point}, $q' = -q_b' = -m_g \xi A_g u^{m_g-1}$, thus $G'$ takes the form:
\begin{equation}
    \label{G'}
    G'(u) = \frac{Q - (m_g + 1) \xi A_g u^{m_g}}{Q - \xi A_g u^{m_g}}u.
\end{equation}

From equation \eqref{2_mod} we obtain $G + g(h + z_b + b) = C$, where $C$ is a constant, hence $z_b = \frac{(C - G)}{g} - h - b$. Furthermore, from the third equation of \eqref{Start_point} we have $z_b' u_t + q_b' u_x = 0$, which implies:
\begin{equation}
    \label{z_b'}
    z_b' = -\frac{G'}{g} + \frac{Q + (m_g - 1) \xi A_g u^{m_g}}{u^2}.
\end{equation}

Finally, combining all the results obtained, we derive the non-linear scalar equation:
\begin{equation}
    \label{Scal_eq}
    u_t + \lambda(u)u_x = 0,
\end{equation}
where
\begin{equation}
    \label{lambda}
    \lambda(u) = \frac{m_g \xi A_g u^{m_g-1}}{ \frac{Q + (m_g - 1) \xi A_g u^{m_g}}{u^2} - \frac{G'(u)}{g}}.
\end{equation}

Equations \eqref{Scal_eq} and \eqref{lambda}, together with equations \eqref{q+q_b=Q}, \eqref{h}, and \eqref{z_b'}, provide a solution for sediment transport in the quasi-static approximation. This solution will lose validity upon shock formation.

\section{CWENO coefficients}\label{App_B}
For completeness, this section will briefly review the CWENO(2,3) parameters used in this study \cite{levy1999central}.
The goal of CWENO(2,3) is to recontruct a piecewise smooth function
$u(x)$ in each cell $i$, $i=1,\ldots,N$, from its cell averages. 
In each cell $i$ it makes use of the cell averages in cells $(i-1,i,i+1)$. 
Let be $\Omega = \{L,C,R\}$. The reconstruction in cell $i$ is given by 
\[
    R_i(x) = \sum_{r\in \Omega} w^rP^r(x)
\]
where 
for all $r\in\Omega$ 
\begin{equation}
\omega^r = \frac{\alpha^r}{\sum_{r\in\Omega} \alpha^s}, \quad \quad \alpha^r = \frac{d^r}{(I[P^r] + \tau)^2}
\end{equation}
in which, $\tau$ is a small number to prevent too small denominators (in our case we used $\tau = \ema{2.2204e-16}$); $d^C = 0.5$ and $d^R = d^L = (1-d^C)/2$.
Here $P^L$ and $P^R$ are the first degree polynomials that match the cell average of a function $u$ we want to reconstruct in cell $i$, respectively in cells $(i-1,i)$ and in cells $(i,i+1)$. The polynomial $P_{\rm opt}$ is the second degree one that matches the cell averages of $u$ in cells $(i-1,i,i+1)$, while polynomial $P^C$ is computed by the relation $P_{\rm opt} = d^L P^L + d^C P^C + d^R P^R$. The smoothness indicators are defined as:
\begin{equation}
I[P^r] := \sum_{l=1}^{2}\int_{x_{i -\frac{1}{2}}}^{x_{i+\frac{1}{2}}} \Delta x^{2l-1}\left( \frac{\partial^l P^r(x)}{\partial^l x}\right)^2dx, \quad r\in\Omega.
\end{equation}
Notice that $P_L$ and $P_R$ are polynomials of degree 1, therefore their second derivative is zero. 

\section{Fully-Discrete high-order scheme}\label{App_C}
To present the fully–discrete semi–implicit IMEX–RK scheme, we now detail the time integration at each implicit–explicit stage. Building upon the semi–discrete first–order formulation (\ref{eta_1-3},\ref{q_1-3}) and a generic double Butcher tableau, the $\ell$-th stage of the high–order semi–implicit method can be written as:

\begin{comment}\giovanni{Nelle formule seguenti preferirei usare una sommatoria anzicché la somma per elencazione dei termini. La formula dovrebbe risultare un pochino più compatta. In questo modo non dobbiamo immaginare quale sia la dimensione corretta delle parentesi quadre, visto che il numero parentesi quadre che si aprono su una riga e si chiudono sulla successiva sarà ridotto drasticamente.
Infine, mi pare che manchino i termini con $a^I_{\ell,j}$ con $j<\ell$ nelle ultime due righe della (C.4). Per il resto mi sembra che vada bene. Stesso discorso sulle sommatorie vale per la soluzione numerica. }
\end{comment}
\begin{align}
    \label{l-stages}
    & \eta_{E,i}^{(\ell)} = \eta_i^n \\
    & q_{E,i+\ha}^{(\ell)} = q_{i+\ha}^n -  \frac{\dt}{\dx}\sum_{j=1}^{\ell-1}a_{\ell,j}^E\Bigl(\F_{i+1}(U_E^{(j)}) - \F_{i}(U_E^{(j)})\Bigr) \\
    & \eta_{I,i}^{(\ell)} = \eta_i^n -  \frac{\dt}{\dx}\sum_{j=1}^\ell a_{\ell,j}^I\Bigl(q_{I,i+\ha}^{(j)} - q_{I,i-\ha}^{(j)}\Bigr) \\ 
    & q_{I,i+\ha}^{(\ell)} = q_{i+\ha}^n - \frac{\dt}{\dx}\sum_{j=1}^{\ell-1}a_{\ell,j}^I\Bigl(\F_{i+1}(U_E^{(j)}) - \F_{i}(U_E^{(j)})\Bigr) +  \\
    & \quad \qquad - g\dt\sum_{j=1}^\ell a_{\ell,j}^I \Bigl[ \Bigl(\xi_{E,i+\ha,i-1}^{L,(j)} - \xi_{E,i+\ha,i-1}^{R,(j)}\Bigr)\eta_{I,i-1}^{(j)} + \Bigl(\xi_{E,i+\ha,i}^{L,(j)} - \xi_{E,i+\ha,i}^{R,(j)}\Bigr)\eta_{I,i}^{(j)}  + \nonumber \\ 
    & \quad\qquad \qquad\qquad+ \Bigl(\xi_{E,i+\ha,i+1}^{L,(j)} - \xi_{E,i+\ha,i+1}^{R,(j)}\Bigr)\eta_{I,i+1}^{(j)} + \Bigl(\xi_{E,i+\ha,i+2}^{L,(j)} - \xi_{E,i+\ha,i+2}^{R,(j)}\Bigr)\eta_{I,i+2}^{(j)} \Bigr]\nonumber
\end{align}
for $\ell = 1,\ldots,s$, where $s$ is the total number of stages in the tableau. The subscript $_E$ in the $\xi$-coefficients denotes the explicit reconstruction of $h$, distinguishing it from the implicit discretization of $\partial_x\eta\,$.

Finally, the numerical solution is:
\begin{align}    
    & \eta^{n+1}_i = \eta_i^n - \frac{\dt}{\dx}\sum_{\ell = 1}^s b_\ell\Bigl(q_{I,i+\ha}^{(\ell)} - q_{I,i+\ha}^{(\ell)}\Bigr)  \\ 
    & q^{n+1}_{i+\ha} = q_{i+\ha}^n - \frac{\dt}{\dx}\sum_{\ell = 1}^s b_\ell\Bigl(\F_{i+1}(U_E^{(\ell)}) - \F_{i}(U_E^{(\ell)})\Bigr) + \label{num_sol_imex}\\
    & \quad\quad\; - g\dt\sum_{\ell = 1}^s b_\ell\Bigl[\Bigl(\xi_{E,i+\ha,i-1}^{L,(\ell)} - \xi_{E,i+\ha,i-1}^{R,(\ell)}\Bigr)\eta_{I,i-1}^{(\ell)} + \Bigl(\xi_{E,i+\ha,i}^{L,(\ell)} - \xi_{E,i+\ha,i}^{R,(\ell)}\Bigr)\eta_{I,i}^{(\ell)} \Bigr.  + \nonumber \\ 
    & \quad\qquad\qquad\qquad\quad + \Bigl.\Bigl(\xi_{E,i+\ha,i+1}^{L,(\ell)} - \xi_{E,i+\ha,i+1}^{R,(\ell)}\Bigr)\eta_{I,i+1}^{(\ell)} + \Bigl(\xi_{E,i+\ha,i+2}^{L,(\ell)} - \xi_{E,i+\ha,i+2}^{R,(\ell)}\Bigr)\eta_{I,i+2}^{(\ell)} \Bigr] \nonumber
\end{align}

\bibliographystyle{ieeetr}
\bibliography{biblio}

\begin{thebibliography}{10}

\bibitem{Casulli1992}
V.~Casulli and R.~T. Cheng, ``Semi-implicit finite difference methods for
  three-dimensional shallow water flow,'' {\em International {J}ournal for
  {N}umerical {M}ethods in {F}luids}, vol.~15, no.~6, pp.~629--648, 1992.

\bibitem{Bonaventura2002}
L.~Bonaventura and G.~Rosatti, ``A cascadic conjugate gradient algorithm for
  mass conservative, semi-implicit discretization of the shallow water
  equations on locally refined structured grids,'' {\em International Journal
  for Numerical Methods in Fluids}, vol.~40, no.~1-2, pp.~217--230, 2002.

\bibitem{Rosatti2011}
G.~Rosatti, L.~Bonaventura, A.~Deponti, and G.~Garegnani, ``An accurate and
  efficient semi-implicit method for section-averaged free-surface flow
  modelling,'' {\em International Journal for Numerical Methods in Fluids},
  vol.~65, pp.~448--473, jan 2011.

\bibitem{Tumolo2013}
G.~Tumolo, L.~Bonaventura, and M.~Restelli, ``A semi-implicit,
  semi-{L}agrangian, $p-$adaptive discontinuous {G}alerkin method for the
  shallow water equations.,'' {\em Journal of Computational Physics}, vol.~232,
  pp.~46--67, 2013.

\bibitem{Busto2020}
S.~Busto, M.~Tavelli, W.~Boscheri, and M.~Dumbser, ``Efficient high order
  accurate staggered semi-implicit discontinuous {G}alerkin methods for natural
  convection problems,'' {\em Computers $\&$ Fluids}, vol.~198, p.~104399, Feb.
  2020.

\bibitem{Busto2021}
S.~Busto, L.~Río-Martín, M.~Vázquez-Cendón, and M.~Dumbser, ``A
  semi-implicit hybrid finite volume/finite element scheme for all mach number
  flows on staggered unstructured meshes,'' {\em Applied Mathematics and
  Computation}, vol.~402, p.~126117, Aug. 2021.

\bibitem{GarresDiaz2021}
J.~Garres-D{\'{\i}}az and L.~Bonaventura, ``Flexible and efficient
  discretizations of multilayer models with variable density,'' {\em Applied
  Mathematics and Computation}, vol.~402, p.~126097, aug 2021.

\bibitem{Orlando2022}
G.~Orlando, P.~F. Barbante, and L.~Bonaventura, ``An efficient {IMEX-DG} solver
  for the compressible {N}avier-{S}tokes equations for non-ideal gases,'' {\em
  Journal of Computational Physics}, vol.~471, p.~111653, Dec. 2022.

\bibitem{Frolkovic2022}
P.~Frolkovič, S.~Krišková, M.~Rohová, and M.~Žeravý, ``Semi-implicit
  methods for advection equations with explicit forms of numerical solution,''
  {\em Japan Journal of Industrial and Applied Mathematics}, vol.~39, no.~3,
  p.~843 – 867, 2022.

\bibitem{GomezBueno2023}
I.~Gómez-Bueno, S.~Boscarino, M.~Castro, C.~Parés, and G.~Russo, ``Implicit
  and semi-implicit well-balanced finite-volume methods for systems of balance
  laws,'' {\em Applied Numerical Mathematics}, vol.~184, pp.~18--48, Feb. 2023.

\bibitem{CaballeroCardenas2023}
C.~Caballero-Cárdenas, M.~Castro, T.~Morales~de Luna, and M.~Mu{\~{n}}oz-Ruiz,
  ``Implicit and implicit-explicit {L}agrange-projection finite volume schemes
  exactly well-balanced for 1{D} shallow water system,'' {\em Applied
  Mathematics and Computation}, vol.~443, p.~127784, Apr. 2023.

\bibitem{Grosso2023}
A.~{Del Grosso}, M.~J. {Castro Díaz}, C.~Chalons, and T.~{Morales de Luna},
  ``On {L}agrange-projection schemes for shallow water flows over movable
  bottom with suspended and bedload transport,'' {\em Numerical Mathematics:
  Theory, Methods and Applications}, vol.~16, pp.~1087--1126, June 2023.

\bibitem{Bonaventura2020}
L.~Bonaventura, F.~Casella, L.~{Delpopolo Carciopolo}, and A.~Ranade, ``A self
  adjusting multirate algorithm for robust time discretization of partial
  differential equations,'' {\em Computers $\&$ Mathematics with Applications},
  vol.~79, pp.~2086--2098, Apr. 2020.

\bibitem{Casulli1994}
V.~Casulli and E.~Cattani, ``Stability, accuracy and efficiency of a
  semi-implicit method for three-dimensional shallow water flow,'' {\em
  Computers $\&$ Mathematics with Applications}, vol.~27, no.~4, pp.~99 -- 112,
  1994.

\bibitem{Bonaventura2018}
L.~Bonaventura, E.~D. Fernández-Nieto, J.~Garres-Díaz, and G.~Narbona-Reina,
  ``Multilayer shallow water models with locally variable number of layers and
  semi-implicit time discretization,'' {\em Journal of Computational Physics},
  vol.~364, pp.~209--234, July 2018.

\bibitem{Busto2022}
S.~Busto and M.~Dumbser, ``A staggered semi-implicit hybrid finite
  volume/finite element scheme for the shallow water equations at all {F}roude
  numbers,'' {\em Applied Numerical Mathematics}, vol.~175, pp.~108--132, May
  2022.

\bibitem{Macca2024}
E.~Macca, S.~Avgerinos, M.-J. Castro-Diaz, and G.~Russo, ``A semi-implicit
  finite volume method for the {E}xner model of sediment transport,'' {\em
  Journal of Computational Physics}, vol.~499, 2024.

\bibitem{GarresDiaz2022}
J.~Garres-Díaz, E.~Fernández-Nieto, and G.~Narbona-Reina, ``A semi-implicit
  approach for sediment transport models with gravitational effects,'' {\em
  Applied Mathematics and Computation}, vol.~421, p.~126938, May 2022.

\bibitem{CFL}
R.~Courant, K.~Friedrichs, and H.~Lewy, ``Über die partiellen
  differenzengleichungen der mathematischen physik,'' {\em Mathematische
  Annalen}, vol.~100, no.~1, pp.~32--74, 1928.

\bibitem{Boscarino-Filbet}
S.~Boscarino, F.~Filbet, and G.~Russo, ``High order semi-implicit schemes for
  time dependent partial differential equations,'' {\em Journal of Scientific
  Computing}, vol.~68, no.~8, pp.~975--1001, 2016.

\bibitem{BoscarinoSemplice}
S.~Boscarino, G.~Russo, and M.~Semplice, ``High order finite volume schemes for
  balance laws with stiff relaxation,'' {\em Computers and Fluids}, vol.~169,
  p.~155 – 168, 2018.

\bibitem{levy1999central}
D.~Levy, G.~Puppo, and G.~Russo, ``Central weno schemes for hyperbolic
  systemsof conservation laws,'' {\em ESAIM: Mathematical Modelling and
  Numerical Analysis}, vol.~33, no.~3, pp.~547--571, 1999.

\bibitem{boscarino2024implicit}
S.~Boscarino, L.~Pareschi, and G.~Russo, {\em Implicit-explicit methods for
  evolutionary partial differential equations}.
\newblock SIAM, 2024.

\bibitem{PareschiRusso}
L.~Pareschi and G.~Russo, ``Implicit–{E}xplicit {R}unge–{K}utta {S}chemes
  and {A}pplications to {H}yperbolic {S}ystems with {R}elaxation,'' {\em
  Journal of Scientific Computing}, vol.~25, pp.~129--155, 2005.

\bibitem{horn2013}
R.~A. Horn and C.~R. Johnson, ``Positive definite and semidefinite matrices,''
  in {\em Matrix Analysis}, pp.~425--515, Cambridge Univ. Press New York, NY,
  USA, 2013.

\bibitem{Grass}
A.~J. Grass, ``Sediments transport by waves and currents,'' {\em SERC London
  Cent Mar Technol}, vol.~Report No. FL29, 1981.

\bibitem{CastroNieto}
M.~Castro, E.~D. Fernández-Nieto, and A.~M. Ferreiro, ``Sediment transport
  models in shallow water equations and numerical approach by high order finite
  volume methods,'' {\em Comput. \& Fluids}, vol.~37, no.~3, pp.~299--316,
  2008.

\bibitem{QianLi}
S.~Qian, G.~Li, F.~Shao, and Q.~Niu, ``Well-balanced central {WENO} schemes for
  the sediment transport model in shallow water,'' {\em Comput. Geosci},
  vol.~22, no.~3, pp.~763--773, 2018.

\bibitem{deVries1965}
M.~de~Vries, ``Considerations about non-steady bed-load-transport in open
  channels,'' {\em Hydraulics Laboratory}, 1965.

\bibitem{Armanini2018}
A.~Armanini, {\em Principles of River Hydraulics}.
\newblock Springer, 2018.

\bibitem{MaccaRussoBumi}
E.~Macca and G.~Russo, ``Boundary effects on wave trains in the exner model of
  sedimental transport,'' {\em Bollettino dell'Unione Matematica Italiana},
  vol.~17, no.~2, p.~417 – 433, 2024.

\bibitem{Hudson2005}
J.~Hudson, J.~Damgaard, N.~Dodd, T.~Chesher, and A.~Cooper, ``Numerical
  approaches for 1{D} morphodynamic modelling,'' {\em Coastal Engineering},
  vol.~52, pp.~691--707, Aug. 2005.

\bibitem{Mayer1998}
S.~Mayer, A.~Garapon, and L.~S. Sørensen, ``A fractional step method for
  unsteady free-surface flow with applications to non-linear wave dynamics,''
  {\em International Journal for Numerical Methods in Fluids}, vol.~28, no.~2,
  pp.~293--315, 1998.

\bibitem{CarlosMunoz}
C.~Mu{\~{n}}oz-Moncayo, ``Efficient absorbing boundary conditions for
  conservation laws: Application to the piston problem of gas dynamics,'' {\em
  In preparation}, 2024.

\bibitem{Karni}
S.~Karni, ``Far-{F}ield {F}iltering {O}perators for {S}uppression of
  {R}eflection {F}rom {A}rtificial {B}oundaries,'' {\em SIAM Journal on
  Numerical Analysis}, vol.~33, no.~3, pp.~1014--1047, 1996.

\bibitem{CHOI2020}
Y.-M. Choi, Y.~J. Kim, B.~Bouscasse, S.~Seng, L.~Gentaz, and P.~Ferrant,
  ``Performance of different techniques of generation and absorption of
  free-surface waves in {C}omputational {F}luid {D}ynamics,'' {\em Ocean
  Engineering}, vol.~214, p.~107575, 2020.

\bibitem{Engsig-Karup2006}
A.~P. Engsig-Karup, J.~S. Hesthaven, H.~B. Bingham, and P.~A. Madsen, ``Nodal
  dg-fem solution of high-order boussinesq-type equations,'' {\em J Eng Math},
  vol.~56, 2006.

\bibitem{Bosca2024}
E.~Macca and S.~Boscarino, ``Semi-implicit-type order-adaptive cat2 schemes for
  systems of balance laws with relaxed source term,'' {\em Communications on
  Applied Mathematics and Computation}, 2025.

\bibitem{Toro2009}
E.~Toro, {\em {{Riemann}} Solvers and Numerical Methods for Fluid Dynamics}.
\newblock Springer, third~ed., 2009.

\bibitem{Quinpi}
G.~Puppo, M.~Semplice, and G.~Visconti, ``Quinpi: Integrating conservation laws
  with {CWENO} implicit methods,'' {\em Communications on Applied Mathematics
  and Computation}, vol.~5, 2023.

\bibitem{Macca2024CATMOOD}
E.~Macca, R.~Loubère, C.~Parés, and G.~Russo, ``An almost fail-safe
  a-posteriori limited high-order cat scheme,'' {\em Journal of Computational
  Physics}, vol.~498, 2024.

\bibitem{Loubere2024}
R.~Loubère, E.~Macca, C.~Parés, and G.~Russo, ``Cat-mood methods for
  conservation laws in one space dimension,'' {\em SEMA SIMAI Springer Series},
  vol.~35, p.~171 – 183, 2024.

\end{thebibliography}

\end{document}